\documentclass[a4paper]{amsart}
\usepackage[leqno]{amsmath}
\usepackage{enumerate}
\usepackage{hyphenat}
\usepackage{ifthen}
\usepackage{amstext,amsbsy,amsopn,amsthm,amsgen}
\usepackage{amsfonts,amscd,amsxtra,upref}
\usepackage{mathrsfs}
\usepackage{euscript}
\usepackage{amssymb}
\usepackage{hyperref}

\swapnumbers
\theoremstyle{plain}
\newtheorem{Thm}{Theorem}[section]
\newtheorem{Lem}[Thm]{Lemma}
\newtheorem{Cor}[Thm]{Corollary}
\newtheorem{Pro}[Thm]{Proposition}
\theoremstyle{definition}
\newtheorem{Def}[Thm]{Definition}
\newtheorem{Exm}[Thm]{Example}
\newtheorem{Prb}[Thm]{Problem}
\theoremstyle{remark}
\newtheorem{Rem}[Thm]{Remark}

\numberwithin{equation}{section}

\newcommand{\ITE}[3]{\ifthenelse{#1}{#2}{#3}}\newcommand{\ITEE}[4][]{\ITE{\equal{#2}{#3}}{#4}{#1}}
 \ITE{\isundefined{\texorpdfstring}}{\newcommand{\texorpdfstring}[2]{#1}}{}

\newenvironment{cor}[2][]{\ITEE[{\begin{Cor}[#1]}]{#1}{}{\begin{Cor}}\label{cor:#2}}{\end{Cor}}
\newenvironment{dfn}[2][]{\ITEE[{\begin{Def}[#1]}]{#1}{}{\begin{Def}}\label{def:#2}}{\end{Def}}
\newenvironment{exm}[2][]{\ITEE[{\begin{Exm}[#1]}]{#1}{}{\begin{Exm}}\label{exm:#2}}{\end{Exm}}
\newenvironment{lem}[2][]{\ITEE[{\begin{Lem}[#1]}]{#1}{}{\begin{Lem}}\label{lem:#2}}{\end{Lem}}
\newenvironment{prb}[2][]{\ITEE[{\begin{Prb}[#1]}]{#1}{}{\begin{Prb}}\label{prb:#2}}{\end{Prb}}
\newenvironment{pro}[2][]{\ITEE[{\begin{Pro}[#1]}]{#1}{}{\begin{Pro}}\label{pro:#2}}{\end{Pro}}
\newenvironment{rem}[2][]{\ITEE[{\begin{Rem}[#1]}]{#1}{}{\begin{Rem}}\label{rem:#2}}{\end{Rem}}
\newenvironment{thm}[2][]{\ITEE[{\begin{Thm}[#1]}]{#1}{}{\begin{Thm}}\label{thm:#2}}{\end{Thm}}

\newcommand{\COR}[2][!]{\ITEE{#1}{!}{Corollary~}\ITEE{#1}{s}{Corollaries~}\textup{\ref{cor:#2}}}
\newcommand{\DEF}[2][!]{\ITEE{#1}{!}{Definition~}\ITEE{#1}{s}{Definitions~}\textup{\ref{def:#2}}}
\newcommand{\EXM}[2][!]{\ITEE{#1}{!}{Example~}\ITEE{#1}{s}{Examples~}\textup{\ref{exm:#2}}}
\newcommand{\LEM}[2][!]{\ITEE{#1}{!}{Lemma~}\ITEE{#1}{s}{Lemmas~}\textup{\ref{lem:#2}}}
\newcommand{\PRO}[2][!]{\ITEE{#1}{!}{Proposition~}\ITEE{#1}{s}{Propositions~}\textup{\ref{pro:#2}}}
\newcommand{\REM}[2][!]{\ITEE{#1}{!}{Remark~}\ITEE{#1}{s}{Remarks~}\textup{\ref{rem:#2}}}
\newcommand{\THM}[2][!]{\ITEE{#1}{!}{Theorem~}\ITEE{#1}{s}{Theorems~}\textup{\ref{thm:#2}}}

\newcommand{\CCC}{\mathbb{C}}
\newcommand{\NNN}{\mathbb{N}}
\newcommand{\RRR}{\mathbb{R}}
\newcommand{\EeE}{\EuScript{E}}
\newcommand{\FfF}{\EuScript{F}}
\newcommand{\JjJ}{\EuScript{J}}
\newcommand{\LlL}{\EuScript{L}}
\newcommand{\WwW}{\EuScript{W}}
\newcommand{\Aa}{\mathfrak{A}}
\newcommand{\Dd}{\mathfrak{D}}
\newcommand{\Ee}{\mathfrak{E}}
\newcommand{\Ff}{\mathfrak{F}}
\newcommand{\Hh}{\mathfrak{H}}
\newcommand{\Kk}{\mathfrak{K}}
\newcommand{\Ss}{\mathfrak{S}}
\newcommand{\Tt}{\mathfrak{T}}
\newcommand{\Uu}{\mathfrak{U}}
\newcommand{\Xx}{\mathfrak{X}}
\newcommand{\Zz}{\mathfrak{Z}}
\newcommand{\aA}{\mathfrak{a}}
\newcommand{\bB}{\mathfrak{b}}
\newcommand{\dD}{\mathfrak{d}}
\newcommand{\hH}{\mathfrak{h}}
\newcommand{\sS}{\mathfrak{s}}
\newcommand{\tT}{\mathfrak{t}}
\newcommand{\xX}{\mathfrak{x}}
\newcommand{\yY}{\mathfrak{y}}
\newcommand{\zZ}{\mathfrak{z}}
\newcommand{\aaA}{\mathscr{A}}
\newcommand{\bbB}{\mathscr{B}}
\newcommand{\ddD}{\mathscr{D}}
\newcommand{\jjJ}{\mathscr{J}}
\newcommand{\kkK}{\mathscr{K}}
\newcommand{\llL}{\mathscr{L}}
\newcommand{\mmM}{\mathscr{M}}
\newcommand{\qqQ}{\mathscr{Q}}
\newcommand{\rrR}{\mathscr{R}}
\newcommand{\ssS}{\mathscr{S}}
\newcommand{\ttT}{\mathscr{T}}
\newcommand{\uuU}{\mathscr{U}}
\newcommand{\vvV}{\mathscr{V}}
\newcommand{\wwW}{\mathscr{W}}

\newcommand{\ueA}{\textup{\textsf{A}}}
\newcommand{\ueB}{\textup{\textsf{B}}}
\newcommand{\ueT}{\textup{\textsf{T}}}
\newcommand{\ueW}{\textup{\textsf{W}}}
\newcommand{\ueX}{\textup{\textsf{X}}}
\newcommand{\ueY}{\textup{\textsf{Y}}}
\newcommand{\ueZ}{\textup{\textsf{Z}}}

\newcommand{\club}{{\scriptsize$\clubsuit$}}
\newcommand{\dd}{\colon}
\newcommand{\df}{\stackrel{\textup{def}}{=}}
\newcommand{\dint}[1]{\,\textup{d} #1}
\newcommand{\epsi}{\varepsilon}
\newcommand{\geqsl}{\geqslant}
\newcommand{\grp}[1]{\langle#1\rangle}
\newcommand{\leqsl}{\leqslant}
\newcommand{\spade}{{\scriptsize$\spadesuit$}}
\newcommand{\varempty}{\varnothing}

\newcommand{\card}{\operatorname{card}}
\newcommand{\conv}{\operatorname{conv}}
\newcommand{\core}{\operatorname{core}}
\newcommand{\dist}{\operatorname{dist}}
\newcommand{\hgt}{\operatorname{ht}}
\newcommand{\id}{\operatorname{id}}
\newcommand{\pr}{\operatorname{pr}}
\newcommand{\stab}{\operatorname{stab}}

\newcommand{\tfcae}{the following conditions are equivalent:}
\newcommand{\iaoi}{if and only if}

\begin{document}

\title{Models for subhomogeneous $C^*$-algebras}
\author[P. Niemiec]{Piotr Niemiec}
\address{Instytut Matematyki\\{}Wydzia\l{} Matematyki i~Informatyki\\{}
 Uniwersytet Jagiello\'{n}ski\\{}ul. \L{}ojasiewicza 6\\{}30-348 Krak\'{o}w\\{}Poland}
\email{piotr.niemiec@uj.edu.pl}
\thanks{The author gratefully acknowledges the assistance of the Polish Ministry of Sciences
 and Higher Education grant NN201~546438 for the years 2010--2013.}
\begin{abstract}
A new category of topological spaces with additional structures, called m-towers, is introduced.
It is shown that there is a covariant functor which establishes a one-to-one correspondences between
unital (resp.\ arbitrary) subhomogeneous $C^*$-algebras and \textit{proper} (resp.\ \textit{proper
pointed}) m-towers of finite \textit{height}, and between all $*$-homomorphisms between two such
algebras and morphisms between m-towers corresponding to these algebras.
\end{abstract}
\maketitle

\section{Introduction}

Subhomogeneous $C^*$-algebras play a central role in operator algebras. For example, they appear
in the definition of ASH algebras, which have been widely investigated for a long time, and
in composition series of GCR (and CCR) algebras. Moreover, subhomogeneity is a property which well
behaves in almost all \textit{finite} operations on $C^*$-algebras, such as direct sums and
products, tensor products, quotients, extensions and passing to subalgebras (which is
in the opposite to \textit{homogeneity}; models for homogeneous $C^*$-algebras were described more
than 50 years ago by Fell \cite{fe3} and Tomiyama and Takesaki \cite{t-t}). This causes that
subhomogeneous $C^*$-algebras form a flexible and rich category. Taking into account all these
remarks, any information about the structure of this category is of significant importance. Besides,
a deep and long-term research on \textit{duals} or \textit{spectra} (or \textit{Jacobson spaces})
of $C^*$-algebras, carried on by Kaplansky \cite{kap}, Fell \cite{fe1,fe2,fe3}, Dixmier \cite{di1}
(consult also Chapter~3 in \cite{di2}), Effros \cite{eff}, Bunce and Deddens \cite{b-d} and others,
led to a conclusion that there is no reasonable way of describing \textit{all} subhomogeneous
$C^*$-algebras by means of their spectra. (There are known partial results in this direction, e.g.\
for $C^*$-algebras with continuous traces or Hausdorff spectra.) The main aim of the paper is
to describe all such algebras (as well as $*$-homomorphisms between them) by means of a certain
category of topological spaces, called by us \textit{m-towers}. Taking this into account, our work
may be seen as a solution of a long-standing problem in operator algebras. Our description resembles
the commutative Gelfand-Naimark theorem and is transparent and easy to use, which is in contrast
to the concept of Vasil'ev \cite{vas}, who proposed a totally different description
of subhomogeneous $C^*$-algebras. His description is complicated, artificial and hardly applicable,
and thus his paper is fairly unregarded. To formulate our main results, first we introduce necessary
notions and notations.\par
For each $n > 0$, let $\mmM_n$ and $\uuU_n$ denote, respectively, the $C^*$-algebra of all complex
$n \times n$ matrices and its unitary group. For $U \in \uuU_n$ and $X \in \mmM_n$, we denote
by $U . X$ and $d(X)$ the matrix $U X U^{-1}$ and the degree of $X$ (that is, $d(X) \df n$),
respectively. Finally, when $X \in \mmM_n$ and $Y \in \mmM_k$, $X \oplus Y$ stands for the block
matrix $\begin{pmatrix}X & 0\\0 & Y\end{pmatrix} \in \mmM_{n+k}$.\par
By an \textit{$n$-dimensional representation} of a $C^*$-algebra $\aaA$ we mean any (possibly zero)
$*$-homomorphism of $\aaA$ into $\mmM_n$.\par
For every unital $C^*$-algebra $\aaA$ and each $n > 0$, we denote by $\Xx_n(\aaA)$ the space of all
unital $n$-dimensional representations of $\aaA$, equipped with the pointwise convergence topology.
The \textit{concrete tower} of $\aaA$ is the topological disjoint union $\Xx(\aaA) \df
\bigsqcup_{n=1}^{\infty} \Xx_n(\aaA)$ of the collection $\{\Xx_n(\aaA)\dd\ n > 0\}$, equipped with
additional ingredients:
\begin{itemize}
\item \textit{degree} map $d\dd \Xx(\aaA) \to \NNN \df \{1,2,\ldots\}$; $d(\pi) \df n$ for any $\pi
 \in \Xx_n(\aaA)$;
\item \textit{addition} $\oplus\dd \Xx(\aaA) \times \Xx(\aaA) \to \Xx(\aaA)$ defined pointwise (that
 is, $(\pi_1 \oplus \pi_2)(a) \df \pi_1(a) \oplus \pi_2(a)$ for each $a \in \aaA$);
\item \textit{unitary action}; that is, a collection of functions $\uuU_n \times \Xx_n(\aaA) \ni
 (U,\pi) \mapsto U . \pi \in \Xx_n(\aaA)\ (n > 0)$ where, for $U \in \uuU_n$ and $\pi \in
 \Xx_n(\aaA)$, $U . \pi\dd \aaA \to \mmM_n$ is defined pointwise (that is, $(U . \pi)(a) \df U .
 \pi(a)$).
\end{itemize}
Based on the concept of concrete towers (introduced above), one may define \textit{abstract}
m-towers with analogous ingredients (they need to satisfy some additional axioms, which for concrete
towers are automatically fulfilled). For any m-tower $\Xx$, the space $C^*(\Xx)$ is defined
as the space of all matrix-valued functions $f$ defined on $\Xx$ such that:
\begin{itemize}
\item $f(\xX) \in \mmM_{d(\xX)}$ for any $\xX \in \Xx$;
\item for each $n > 0$, $f\bigr|_{\Xx_n}\dd \Xx_n \to \mmM_n$ is continuous (where $\Xx_n \df
 d^{-1}(\{n\})$);
\item $f(U . \xX) = U . f(\xX)$ for all $(U,\xX) \in \bigcup_{n=1}^{\infty} (\uuU_n \times \Xx_n)$;
\item $f(\xX_1 \oplus \xX_2) = f(\xX_1) \oplus f(\xX_2)$ for any $\xX_1, \xX_2 \in \Xx$;
\item $(\|f\| \df\,) \sup_{\xX\in\Xx} \|f(\xX)\| < \infty$.
\end{itemize}
It is readily seen that $C^*(\Xx)$ is a unital $C^*$-algebra when all algebraic operations are
defined pointwise. Further, the \textit{height} of $\Xx$, denoted by $\hgt(\Xx)$, is defined
as follows. An element $\xX$ of $\Xx$ \textit{irreducible} if each unitary operator $U \in
\uuU_{d(\xX)}$ for which $U . \xX = \xX$ is a scalar multiple of the unit matrix. Then:
\begin{equation*}
\hgt(\Xx) \df \sup\{d(\xX)\dd\ \xX \in \Xx \textup{ is irreducible}\} \in \{1,2,\ldots,\infty\}.
\end{equation*}
The m-tower $\Xx$ is called \textit{proper} if each of the subspaces $\Xx_n$ is compact.\par
Finally, if $\aaA$ is (again) an arbitrary unital $C^*$-algebra and $a \in \aaA$, $\hat{a}$ is
a matrix-valued function defined on $\Xx(\aaA)$ defined by $\hat{a}(\pi) = \pi(a)$. It is a typical
property that $J_{\aaA}\dd \aaA \ni a \mapsto \hat{a} \in C^*(\Xx(\aaA))$ is a correctly defined
unital $*$-homomorphism. It is also clear that $J_{\aaA}$ is one-to-one iff $\aaA$ is residually
finite-dimensional. Our main result on m-towers is

\begin{thm}{subh1}
For any unital subhomogeneous $C^*$-algebra $\aaA$, $\Xx(\aaA)$ is a proper m-tower,
$\hgt(\Xx(\aaA)) < \infty$ and $J_{\aaA}\dd \aaA \to C^*(\Xx(\aaA))$ is a $*$-isomorphism.\par
Conversely, if $\Tt$ is a proper m-tower of finite height, then $C^*(\Tt)$ is subhomogeneous and for
every its finite-dimensional unital representation $\pi$ there exists a unique point $\tT \in
C^*(\Tt)$ such that $\pi = \pi_{\tT}$ where
\begin{equation*}
\pi_{\tT}\dd C^*(\Tt) \ni u \mapsto u(\tT) \in \mmM_{d(\tT)}.
\end{equation*}
Moreover, the assignment $\tT \mapsto \pi_{\tT}$ correctly defines an isomorphism between the towers
$\Tt$ and $\Xx(C^*(\Tt))$.
\end{thm}

\THM{subh1} establishes a one-to-one correspondence between unital subhomogeneous $C^*$-algebras
and proper m-towers of finite heights. It also enables characterizing $*$-homomorphisms between such
algebras by means of so-called \textit{morphisms} between m-towers.\par
To describe models for all (possibly) nonunital subhomogeneous $C^*$-agebras, it suffices to collect
all (possibly nonunital) finite-dimensional representations. More precisely, for any $C^*$-algebra
$\aaA$, the \textit{concrete pointed tower} of $\aaA$, denoted by $(\Zz(\aaA),\theta_{\aaA})$, is
the tower $\Zz(\aaA)$ of all its (possibly zero) finite-dimensional representations together with
a distinguished element $\theta_{\aaA}$ which is the zero one-dimensional representation of $\aaA$.
An \textit{abstract} pointed m-tower is any pair $(\Zz,\theta)$ where $\Zz$ is an m-tower and
$\theta \in \Zz$ is such that $d(\theta) = 1$. For each such a pair one defines $C^*(\Zz,\theta)$
as the subspace of $C^*(\Zz)$ consisting of all functions $f \in C^*(\Zz)$ for which $f(\theta) =
0$. For any $C^*$-algebra $\aaA$, $\JjJ_{\aaA}\dd \aaA \to C^*(\Zz(\aaA),\theta)$ is defined
in the same manner as for unital $C^*$-algebras. Then we have:

\begin{thm}{subh0}
For any subhomogeneous $C^*$-algebra $\aaA$, $(\Zz(\aaA),\theta_{\aaA})$ is a proper pointed m-tower
of finite height and $J_{\aaA}\dd \aaA \to C^*(\Zz(\aaA),\theta_{\aaA})$ is
a $*$\hyp{}isomorphism.\par
Conversely, if $(\Ss,\kappa)$ is an arbitrary proper pointed m-tower of finite height, then
$C^*(\Ss,\kappa)$ is subhomogeneous and for every its finite-dimensional representation $\pi$ there
is a unique point $\sS \in \Ss$ such that $\pi = \pi_{\sS}$ where
\begin{equation*}
\pi_{\sS}\dd C^*(\Ss,\kappa) \ni u \mapsto u(\sS) \in \mmM_{d(\sS)}.
\end{equation*}
Moreover, the assignment $\sS \mapsto \pi_{\sS}$ correctly defines an isomorphism between
the pointed towers $(\Ss,\kappa)$ and $\Zz(C^*(\Ss,\kappa),\theta_{C^*(\Ss,\kappa)})$.
\end{thm}

From \THM[s]{subh1} and \THM[]{subh0} one may easily deduce the commutative Gelfand-Naimark theorem
as well as a result on models for homogeneous $C^*$-algebras due to Fell \cite{fe3} and Tomiyama and
Takesaki \cite{t-t} (this is discussed in \REM{homo} in Section~5). Since $*$-homomorphisms between
subhomogeneous $C^*$-algebras may naturally be interpreted, in a functorial manner, as morphisms
between m-towers, therefore our approach to homogeneous (as a part of subhomogeneous) $C^*$-algebras
may be seen as better than the original proposed by Fell, and Tomiyama and Takesaki (because using
their models it is far from `naturalness' to describe $*$-homomorhisms between $n$-homogeneous and
$m$-homogeneous $C^*$-algebras for different $n$ and $m$). A one-to-one \textit{functorial}
correspondence between the realms of subhomogeneous $C^*$-algebras and (proper) m-towers gives new
tools in investigations of other classes of $C^*$-algebras, such as approximately subhomogeneous
(ASH), approximately homogeneous (AH) or CCR (which are ASH, due to a result of Sudo \cite{sud})
as well as GCR $C^*$-algebras (taking into account their \textit{composition series}; see
\cite{kap}).\par
Although \THM{subh1} is very intuitive, its proofs is difficult and involves our recent result
\cite{pn3} on norm closures of convex sets in uniform spaces of vector-valued functions.\par
The paper is organized as follows. Section~2 we discuss special kinds of towers, which are very
natural and intuitive. Based on their properties, in the third part we introduce abstract towers and
distinguish m-towers among them. We establish fundamental properties of m-towers, which shall find
applications in further sections. In Section~4 we study so-called essentially locally compact
m-towers and prove for them a variation of the classical Stone-Weierstrass theorem (see \THM{SW!}).
Its special case, \THM{SW}, is crucial for proving \THM{subh1}. We also characterize there all
(closed two-sided) ideals (see \COR{ideal}) and finite-dimensional nondegenerate representations
of $C^*$-algebras associated with essentially locally compact m-towers (cosult \COR{nondeg}).
The fifth part is devoted to the proofs of \THM[s]{subh1} and \THM[]{subh0} and some of their
consequences. We also give there a description of models for all inverse limits of subhomogeneous
$C^*$-algebras (consult \THM{inv} as well as \REM{nonuniq}). In Section~6 we generalize \THM{subh0}
to the case of all so-called shrinking $C^*$-algebras (that include, among other algebras, direct
products of subhomogeneous $C^*$-algebras). We establish man interesting properties and give a few
characterizations of them. The last, seventh part, is devoted to the concept of so-called shrinking
multiplier algebras (associated only with shrinking $C^*$-algebras). The main result of this section
is \THM{closed} which enables characterizing those separable shrinking $C^*$-algebras which coincide
with their shrinking multiplier algebras.

\subsection*{Notation and terminology}
Representations of $C^*$-algebras need not be nonzero and $*$-homomorphisms between unital
$C^*$-algebras need not preserve unities. Ideals in $C^*$-algebras are, by definition, closed and
two-sided. All topological spaces are assumed to be Hausdorff (unless otherwise stated).
A \textit{map} is a continuous function. A map $u\dd X \to Y$ is \textit{proper} if $u^{-1}(L)$ is
compact for any compact $L \subset Y$. By a \textit{topological semigroup} we mean any semigroup
(possibly with no unit) equipped with a (Hausdorff) topology with respect to which the action
of the semigroup is continuous. A \textit{clopen} set in a topological space is a set which is
simultaneously open and closed.\par
All notations and terminologies introduced earlier in this section are obligatory. Additionally,
we denote by $I_n$ the unit $n \times n$ matrix.

\section{Standard towers}

This brief section is devoted to a very special type of m-towers called \textit{standard towers}
(general m-towers shall be discussed in the next part), which are intuitive and natural. We start
from them in order to elucidate main ideas and make it easier to assimilate further definitions
(which may be seen strange or somewhat artificial). Actually, we shall prove in the sequel that each
proper m-tower is \textit{isomorphic} to a standard tower.\par
Let $\Lambda$ be an arbitrary nonempty set (of indices). We define $\mmM[\Lambda]$
as the topological disjoint union $\bigsqcup_{n=1}^{\infty} \mmM_n^{\Lambda}$ of the spaces
$\mmM_n^{\Lambda}$ each of which is equipped with the pointwise convergence topology. Further, for
two members $\ueX = (X_{\lambda})_{\lambda\in\Lambda}$ and $\ueY =
(Y_{\lambda})_{\lambda\in\Lambda}$ of $\mmM[\Lambda]$, we define $d(\ueX)$ and $\ueX \oplus \ueY$
as follows: $d(\ueX) \df n$ where $n > 0$ is such that $\ueX \in \mmM_n^{\Lambda}$ and $\ueX \oplus
\ueY \df (X_{\lambda} \oplus Y_{\lambda})_{\lambda\in\Lambda}$. If, in addition, $U \in
\uuU_{d(\ueX)}$, $U . \ueX$ is defined coordinatewise; that is, $U . \ueX \df
(U . X_{\lambda})_{\lambda\in\Lambda}$. In this way we have obtained functions $d\dd \mmM[\Lambda]
\to \NNN$ and $\oplus\dd \mmM[\Lambda] \times \mmM[\Lambda] \to \mmM[\Lambda]$ and a collection
\begin{equation*}
\uuU_n \times \mmM_n^{\Lambda} \ni (U,\ueX) \mapsto U . \ueX \in \mmM_n^{\lambda}, \qquad
n=1,2,\ldots
\end{equation*}
of group actions. We call $d$, $\oplus$ and the above collection the \textit{degree} map,
the \textit{addition} and the \textit{unitary action} (respectively). For a further use, we put
$\mmM \df \mmM[\{1\}]$ and we shall think of $\mmM$ as of the space of all (square complex) single
matrices.

\begin{dfn}{standard}
A \textit{standard tower} is any subspace $\ttT$ of $\mmM[\Lambda]$ (for some set $\Lambda$) which
satisfies all the following conditions:
\begin{enumerate}[(ST1)]\addtocounter{enumi}{-1}
\item $\ttT$ is a closed subset of $\mmM[\Lambda]$;
\item whenever $(X_{\lambda})_{\lambda\in\Lambda} \in \mmM[\Lambda]$ belongs to $\ttT$, then
 $\|X_{\lambda}\| \leqsl 1$ for all $\lambda \in \Lambda$;
\item for each $\ueX \in \ttT$ and $U \in \uuU_{d(\ueX)}$, $U . \ueX$ belongs to $\ttT$ as well;
\item for any two elements $\ueX$ and $\ueY$ of $\mmM[\Lambda]$,
 \begin{equation}\label{eqn:oplus}
 \ueX \oplus \ueY \in \ttT \iff \ueX, \ueY \in \ttT.
 \end{equation}
\end{enumerate}
For any standard tower $\ttT$ and each integer $n > 0$, the subspace $\ttT_n$ is defined
as $\{\ueX \in \ttT\dd\ d(\ueX) = n\}$.\par
By a \textit{standard pointed tower} we mean any pair of the form $(\ttT,\theta_{\ttT})$ where $\ttT
\subset \mmM[\Lambda]$ is a tower which contains the point $\theta_{\ttT} =
(Z_{\lambda})_{\lambda\in\Lambda}$ with $Z_{\lambda} = 0$ for each $\lambda \in \Lambda$. Note that
$d(\theta_{\ttT}) = 1$ for any standard pointed tower $(\ttT,\theta_{\ttT})$.
\end{dfn}

The following result is an immediate consequence of the Tychonoff theorem (on products of compact
spaces) and the definition of m-towers. We therefore omit its proof.

\begin{pro}{std1}
Let $\ttT$ be a standard tower. Then:
\begin{enumerate}[\upshape(T1)]\addtocounter{enumi}{-1}
\item $d\dd \ttT \to \NNN$ is a proper map;
\item $(\ttT,\oplus)$ is a topological semigroup and for each $n > 0$, the map $\ttT_n \times \ttT
 \ni (\ueX,\ueY) \mapsto \ueX \oplus \ueY \in \ttT$ is a closed embedding;
\item $d(\ueX \oplus \ueY) = d(\ueX) + d(\ueY)$ and $d(U . \ueX) = d(\ueX)$ for any $\ueX, \ueY \in
 \ttT$ and $U \in \uuU_{d(\ueX)}$;
\item for each $n > 0$, the function $\uuU_n \times \ttT_n \ni (U,\ueX) \mapsto U . \ueX \in \ttT_n$
 is a continuous group action \textup{(}that is, $(UV) . \ueX = U . (V . \ueX)$ and $I_n . \ueX =
 \ueX$ for all $U, V \in \uuU_n$ and $\ueX \in \ttT_n$\textup{)};
\item $U . \ueX = \ueX$ provided $\ueX$ is an arbitrary member of $\ttT$ and $U \in \uuU_{d(\ueX)}$
 is a scalar multiple of the unit matrix;
\item for any $\ueX \in \ttT_n$, $\ueY \in \ttT_k$, $U \in \uuU_n$ and $V \in \uuU_k$, $(U \oplus V)
 . (\ueX \oplus \ueY) = (U . \ueX) \oplus (V . \ueY)$;
\item for any two elements $\ueX$ and $\ueY$ of $\ttT$, $U_{d(\ueX),d(\ueY)} . (\ueX \oplus \ueY) =
 \ueY \oplus \ueX$ where, for any positive integers $p$ and $q$, $U_{p,q} \in \uuU_{p+q}$ is defined
 by the rule:
 \begin{equation*}
 \begin{pmatrix}w_1 & \dots & w_q & z_1 & \dots & z_p\end{pmatrix} \cdot U_{p,q} \df
 \begin{pmatrix}z_1 & \dots & z_p & w_1 & \dots & w_q\end{pmatrix}.
 \end{equation*}
\end{enumerate}
In particular, the subspaces $\ttT_n$ are compact and $\ttT$ is locally compact.
\end{pro}

Properties (T1)--(T6) will serve as axioms of \textit{towers} (see the beginning of the next
section), which are more general than m-towers. The latter structures will have also some other
properties, which we shall now establish for standard towers. To this end, we introduce

\begin{dfn}{irr}
Let $\ttT$ be a standard tower. An element $\ueX$ of $\ttT$ is said to be
\begin{itemize}
\item \textit{reducible} if there are $\ueA, \ueB \in \ttT$ and $V \in \uuU_{d(\ueX)}$ such that
 $V . \ueX = \ueA \oplus \ueB$;
\item \textit{irreducible} if $\ueX$ is not reducible.
\end{itemize}
The \textit{stabilizer} $\stab(\ueX)$ of $\ueX$ is defined as $\stab(\ueX) \df \{U \in
\uuU_{d(\ueX)}\dd\ U . \ueX = \ueX\}$.\par
Finally, for two elements $\ueX$ and $\ueY$ of $\ttT$ we shall write
\begin{itemize}
\item $\ueX \equiv \ueY$ if $U . \ueX = \ueY$ for some $U \in \uuU_{d(\ueX)}$;
\item $\ueX \preccurlyeq \ueY$ if either $\ueX \equiv \ueY$ or there are $\ueA, \ueB \in \ttT$ such
 that $\ueX \equiv \ueA$ and $\ueY \equiv \ueA \oplus \ueB$;
\item $\ueX \perp \ueY$ if there is no $\ueA \in \ttT$ for which $\ueA \preccurlyeq \ueX$ as well as
 $\ueA \preccurlyeq \ueY$; if this happens, we call $\ueX$ and $\ueY$ \textit{disjoint}.
\end{itemize}
\end{dfn}

Key properties of standard towers are established below.

\begin{pro}{std2}
Let $\ueX$, $\ueY$ and $\ueZ$ be three elements of a standard tower $\ttT$. Then
\begin{enumerate}[\upshape(mT1)]
\item if $\ueZ \perp \ueX$ and $\ueZ \perp \ueY$, then $\ueZ \perp \ueX \oplus \ueY$;
\item $\ueX \equiv \ueY$ provided $\ueZ \oplus \ueX \equiv \ueZ \oplus \ueY$;
\item if $\ueX \perp \ueY$, then $\stab(\ueX \oplus \ueY) = \{U \oplus V\dd\ U \in \stab(\ueX),\
 V \in \stab(\ueY)\}$;
\item if $\ueX$ is irreducible and $p \df d(\ueX)$, then for each $n > 0$,
 \begin{equation*}
 \stab(\underbrace{\ueX \oplus \ldots \oplus \ueX}_n) = \{I_p \otimes U\dd U \in \uuU_n\}.
 \end{equation*}
\end{enumerate}
\end{pro}

Before passing to a proof, we recall that if $U = [z_{jk}] \in \uuU_n$, then $I_p \otimes U$
coincides with the block matrix $[z_{jk} I_p] \in \uuU_{np}$ (see, for example, Section~6.6
in \cite{k-r}).

\begin{proof}
All items of the proposition are known consequences of (well-known) facts on von Neumann algebras
and may be shown as follows. Let $\Lambda$ be a set such that $\ttT \subset \mmM[\Lambda]$. Further,
for any $\ueX = (X_{\lambda})_{\lambda\in\Lambda} \in \ttT$ denote by $\WwW'(\ueX)$ the set of all
matrices $A \in \mmM_{d(\ueX)}$ which commute with both $X_{\lambda}$ and $X_{\lambda}^*$ for any
$\lambda \in \Lambda$. Then $\WwW'(\ueX)$ is a von Neumann algebra such that
\begin{equation}\label{eqn:aux1}
\WwW'(\ueX) \cap \uuU_{d(\ueX)} = \stab(\ueX).
\end{equation}
To show (mT1) and (mT2), we consider $\ueW = (W_{\lambda})_{\lambda\in\Lambda} \df \ueX \oplus \ueY
\oplus \ueZ$ and three projections $P, Q, R \in \WwW'(\ueW)$ given by $P \df I_{d(\ueX)} \oplus
0_{d(\ueY)} \oplus 0_{d(\ueZ)}$, $Q \df 0_{d(\ueX)} \oplus I_{d(\ueY)} \oplus 0_{d(\ueZ)}$ and
$R \df 0_{d(\ueX)} \oplus 0_{d(\ueY)} \oplus I_{d(\ueZ)}$ (where $0_n$ denotes the zero $n \times n$
matrix). Since the restrictions of $\ueW$ to the ranges of $P$, $Q$ and $R$ are unitarily equivalent
to, respectively, $\ueX$, $\ueY$ and $\ueZ$ (which means, for example, that, for $H$ denoting
the range of $P$, $U W_{\lambda}\bigr|_H = X_{\lambda} U$ for all $\lambda$ and a single unitary
operator $U\dd H \to \CCC^{d(\ueX)}$), we conclude from Proposition~2.3.1 in \cite{pn1} (see also
Proposition~1.35 in \cite{ern}; these results are formulated for finite tuples of operators and
single operators, but the assumption of finiteness of tuples is superfluous) that
\begin{itemize}
\item $\ueX \equiv \ueY$ iff the projections $P$ and $Q$ are Murray-von Neumann equivalent
 in $\WwW'(\ueW)$; similarly, $\ueZ \oplus \ueX \equiv \ueZ \oplus \ueY$ iff $R+P$ and $R+Q$ are
 Murray-von Neumann equivalent in $\WwW'(\ueW)$;
\item $\ueZ \perp \ueX$ (resp.\ $\ueZ \perp \ueY$; $\ueZ \perp \ueX \oplus \ueY$) iff the central
 carriers $c_R$ and $c_P$ of $R$ and $P$ (resp.\ $c_R$ and $c_Q$; $c_R$ and $c_{P+Q}$)
 in $\WwW'(\ueW)$ are mutually orthogonal.
\end{itemize}
So, under the assumption of (mT1), we conclude that $c_R c_P = c_R c_Q$. So, $c_R (c_P + c_Q) = 0$
and consequently $c_R c_{P+Q} = 0$, which means that $Z \perp \ueX \oplus \ueY$. This yields (mT1).
Further, if $\ueZ \oplus \ueX \equiv \ueZ \oplus \ueY$, then the projections $R+P$ and $R+Q$ are
Murray-von Neumann equivalent in $\WwW'(\ueW)$. But $\WwW'(\ueW)$ is a finite von Neumann algebra
and therefore the equivalence of $R+P$ and $R+Q$ implies the equivalence of $P$ and $Q$, which
translates into $\ueX \equiv \ueY$. In this way we showed (mT2).\par
To prove (mT3), observe that the inclusion ``$\supset$'' holds with no additional assumptions
on $\ueX$ and $\ueY$ (which follows e.g.\ from (T5)). To see the reverse inclusion (provided $\ueX
\perp \ueY$), put $n \df d(\ueX)$, $k \df d(\ueY)$ and take an arbitrary matrix $W \in \stab(\ueX
\oplus \ueY)$. Write $W$ as a block matrix $\begin{pmatrix}U & A\\B & V\end{pmatrix}$ where $U \in
\mmM_n$, $V \in \mmM_k$, $A$ is an $n \times k$ and $B$ is an $k \times n$ matrix. Since $W . (\ueX
\oplus \ueY) = \ueX \oplus \ueY$ (and consequently $W^* . (\ueX \oplus \ueY) = \ueX \oplus \ueY$),
we conclude that $U X_{\lambda} = X_{\lambda} U$, $V Y_{\lambda} = Y_{\lambda} V$ and
\begin{equation}\label{eqn:aux2}
\begin{cases}
A Y_{\lambda} = X_{\lambda} A,&\\
A Y_{\lambda}^* = X_{\lambda}^* A,&\\
B X_{\lambda} = Y_{\lambda} B,&\\
B X_{\lambda}^* = Y_{\lambda}^* B&\\
\end{cases}
\end{equation}
for any $\lambda \in \Lambda$ (where $(X_{\lambda})_{\lambda\in\Lambda} \df \ueX$ and
$(Y_{\lambda})_{\lambda\in\Lambda} \df \ueY$). So, it suffices to check that both $A$ and $B$ are
zero matrices. One infers from the first two equations in \eqref{eqn:aux2} and Schur's lemma
on intertwinning operators (see, for example, Theorem~1.5 in \cite{ern}) that the range $V$ of $A$
is a reducing subspace for each $X_{\lambda}$, the orthogonal complement $W$ of the kernel of $A$ is
a reducing subspace for each $Y_{\lambda}$ and $Q Y_{\lambda}\bigr|_W = X_{\lambda}\bigr|_V
Q\bigr|_W$ where $Q$ is the partial isometry which appears in the polar decomposition of $A$. Since
$Q$ sends isometrically $W$ onto $V$, the above properties imply that if $A \neq 0$, then $\ueX
\not\perp \ueY$ (one uses here (ST2) and (ST3)), which contradicts the assumption in (mT3). So, $A =
0$. In a similar manner, starting from the last two equations in \eqref{eqn:aux2}, one shows that
$B = 0$. This completes the proof of (mT3).\par
Finally, the assertion of (mT3) immediately follows from the facts that if $\ueX$ is irreducible,
then $\WwW'(\ueX)$ consists of scalar multiples of the unit matrix (thanks to \eqref{eqn:aux1} and
(ST2)--(ST3)), and if a matrix $W \in \stab(\underbrace{\ueX \oplus \ldots \oplus \ueX}_n)$ is
expressed as a block matrix $[T_{jk}]$ with $T_{jk} \in \mmM_p$, then $T_{jk} \in \WwW'(\ueX)$ for
any $j$ and $k$.
\end{proof}

\section{M-towers}

A \textit{tower} is a quadruple $\Tt = (\ttT,d,\oplus,.)$ such that each of the following five
conditions is fulfilled:
\begin{itemize}
\item $\ttT$ is a topological space;
\item $d\dd \ttT \to \NNN$ is a map;
\item ``$\oplus$'' is a function of $\ttT \times \ttT$ into $\ttT$;
\item ``$.$'' is a function of $\bigcup_{n=1}^{\infty} (\uuU_n \times \ttT_n)$ into $\ttT$ where
 $\ttT_n \df \{\ueX \in \ttT\dd\ d(\ueX) = n\}$;
\item conditions (T1)--(T6) of \PRO{std1} hold.
\end{itemize}
If, in addition, condition (T0) is fulfilled, the tower $\Tt$ is said to be \textit{proper}.
(If this happens, the subspaces $\ttT_n$ are compact and $\ttT$ is locally compact.) We call $d$,
``$\oplus$'' and ``$.$'' the \textit{ingredients} of the tower $\ttT$.\par
A \textit{pointed tower} is a pair $(\Tt,\theta)$ where $\Tt = (\ttT,d,\oplus,.)$ is a tower and
$\theta$ is a point in $\ttT$ such that $d(\theta) = 1$.\par
By a \textit{morphism} between two towers $(\ttT,d,\oplus,.)$ and $(\ttT',d',\oplus,.)$ we mean any
map $v\dd \ttT \to \ttT'$ such that for any $\ueX, \ueY \in \ttT$ each of the following conditions
is fulfilled:
\begin{enumerate}[(M1)]
\item $d'(v(\ueX)) = d(\ueX)$;
\item $v(U . \ueX) = U . v(\ueX)$ for all $U \in \uuU_{d(\ueX)}$;
\item $v(\ueX \oplus \ueY) = v(\ueX) \oplus v(\ueY)$.
\end{enumerate}
Similarly, a \textit{morphism} between two pointed towers $(\Tt,\theta)$ and $(\Tt',\theta')$ is any
morphism between the towers $\Tt$ and $\Tt'$ which sends $\theta$ onto $\theta'$.\par
It is easy to see that towers (as objects) with morphisms between them form a category. Thus, we can
(and will) speak of \textit{isomorphisms} between towers and \textit{isomorphic} towers. (It is easy
to see that a morphism between two towers is an isomorphism iff it is a homeomorphism between
underlying topological spaces.)\par
A \textit{subtower} of a tower $(\ssS,d,\oplus,.)$ is any closed subset $\ttT$ of $\ssS$ which
fulfills condition (ST2), and \eqref{eqn:oplus} for all $\ueX, \ueY \in \ssS$. It is an easy
observation that a subtower of $\ssS$ is a tower when it is equipped with all ingredients inherited
from $\ssS$. Similarly, a \textit{pointed subtower} of a pointed tower $(\Ss,\theta)$ is any pair
$(\Tt,\theta')$ where $\Tt$ is a subtower of $\Ss$ containing $\theta'$ and $\theta' = \theta$.\par
In the very same way as in \DEF{irr}, we define reducible and irreducible elements of a tower,
stabilizers of its elements as well as relations ``$\equiv$'', ``$\preccurlyeq$'' and ``$\perp$''.
Additionally, for each $n > 0$ and an element $\ueX$ of a tower (or a matrix), we shall denote
$\underbrace{\ueX \oplus \ldots \oplus \ueX}_n$ by $n \odot \ueX$. Also, for simplicity, whenever
$j$ runs over a finite set (known from the context) of integers and $\ueX_j$ for each such $j$
denotes an element of a common tower (or each $\ueX_j$ is a matrix), $\bigoplus_j \ueX_j$ will stand
for the sum of all $\ueX_j$ arranged in accordance with the natural order of the indices $j$.
Similarly, if $\Ss_j$ are arbitrary subsets of a tower (and $j$ runs over a finite set of integers),
$\bigoplus_j \Ss_j$ will denote the set of all elements of the form $\bigoplus_j \sS_j$ where $\sS_j
\in \Ss_j$ for each $j$. Furthermore, for any subset $\Aa$ of a tower, $\uuU . \Aa$ will stand for
the set of all elements of the form $U . \aA$ where $\aA \in \Aa$ and $U \in \uuU_{d(\aA)}$.\par
The \textit{core} of a tower $\Tt$, to be denoted by $\core(\Tt)$, is defined as the set of all
irreducible elements of $\Tt$. Additionally, $\overline{\core}(\Tt)$ will stand for the closure
of the core of $\Tt$. The \textit{height} of the tower $\Tt$ is the quantity
\begin{equation*}
\hgt(\Tt) \df \sup \{d(\ueT)\dd\ \ueT \in \core(\Tt)\} \in \{0,1,2,\ldots,\infty\}
\end{equation*}
(where $\sup(\varempty) \df 0$).

\begin{dfn}{sigma}
A (possibly non-Hausdorff) topological space $X$ is said to have \textit{property $(\sigma)$ with
respect to $(K_n)_{n=1}^{\infty}$} if each of the following conditions is fulfilled:
\begin{enumerate}[($\sigma$1)]
\item $K_n \subset K_{n+1}$ for all $n$, and $X = \bigcup_{n=1}^{\infty} K_n$;
\item each of $K_n$ is a compact Hausdorff space in the topology inherited from $X$;
\item a set $A \subset X$ is closed (in $X$) iff $A \cap K_n$ is closed for each $n$.
\end{enumerate}
The space $X$ is said to have \textit{property $(\sigma)$} if $X$ has property $(\sigma)$ with
respect to some $(K_n)_{n=1}^{\infty}$.
\end{dfn}

Now we may turn to the main topic of the section.

\begin{dfn}{m-tower}
An \textit{m-tower} is a tower $\Tt = (\ttT,d,\oplus,.)$ such that conditions (mT1)--(mT4)
of \PRO{std2} are fulfilled for all $\ueX, \ueY \in \ttT$, and the topological space $\ttT$ has
property $(\sigma)$.
\end{dfn}

In the above naming, the prefix `m' is to emphasize a strong resemblance of m-towers to standard
towers (which consist of tuples of matrices). Property $(\sigma)$ will enable us to extend certain
matrix-valued functions. (To simplify the presentation, we could replace it by requiring that
an m-tower has to be proper. However, we chose property $(\sigma)$ to enable giving a description
of inverse limits of unital subhomogeneous $C^*$-algebras, which shall be done in \THM{inv}.)\par
The strength of property $(\sigma)$ is revealed below.

\begin{lem}{para}
Let $X$ be a \textup{(}possibly non-Hausdorff\textup{)} topological space that have property
$(\sigma)$ with respect to $(K_n)_{n=1}^{\infty}$. Then:
\begin{enumerate}[\upshape(a)]
\item for any topological space $Y$, a function $f\dd X \to Y$ is continuous iff $f\bigr|_{K_n}\dd
 K_n \to Y$ is continuous for any $n$; and
\item $X$ is a paracompact Hausdorff space.
\end{enumerate}
In particular, each m-tower is a paracompact topological space.
\end{lem}
\begin{proof}
Point (a) is immediate and may be shown by testing the closedness of the inverse image under $f$
of a closed set in $Y$. We turn to (b). We see that all finite sets are closed in $X$. It is also
clear that each of $K_n$ is closed in $X$. To show that $X$ is normal, it suffices to check
the assertion of Tietze's theorem; that is, that every real-valued map defined on a closed set
in $X$ is extendable to a real-valued map on $X$. To this end, take an arbitrary closed set $A$
in $X$ and a map $f_0\dd A \to \RRR$. By induction, we may find maps $f_n\dd K_n \cup A \to \RRR$
such that $f_n$ extends $f_{n-1}$ for $n > 0$. Indeed, putting $K_0 \df \varempty$ and assuming
$f_{n-1}$ is defined (for some $n > 0$), first we extend the restriction of $f_{n-1}$ to $K_{n-1}
\cup (A \cap K_n)$ to a map $v_n\dd K_n \to \RRR$ and then define $f_n$ as the union of $v_n$ and
$f_{n-1}$. Finally, we define the extension $f\dd X \to \RRR$ of $f_0$ by the rule $f(x) \df f_n(x)$
for $x \in K_n$. Since $f\bigr|_{K_n}$ is continuous for each $n > 0$, we infer from (a) that $f$
itself is continuous.\par
We have shown above that $X$ is a normal space (with finite sets closed). It follows from our
assumptions that $X$ is $\sigma$-compact (that is, $X$ is a countable union of compact sets) and
therefore $X$ has the Lindel\"{o}f property (which means that every open cover of $X$ has
a countable subcover). So, Theorem~3.8.11 in \cite{eng} implies that $X$ is paracompact.
\end{proof}

\begin{pro}{std-m-tower}
Each standard tower is an m-tower. More generally, a proper tower is an m-tower iff conditions
\textup{(mT1)--(mT4)} hold for any its elements $\ueX$ and $\ueY$.
\end{pro}
\begin{proof}
All we need to explain is that for a proper tower $(\ttT,d,\oplus,.)$, the space $\ttT$ has property
$(\sigma)$, which is immediate---it suffices to put $\kkK_n \df d^{-1}(\{1,\ldots,n\})$.
\end{proof}

From now on, for simplicity, we shall identify an m-tower $\Tt = (\ttT,d,\oplus,.)$ with its
underlying topological space $\ttT$ (and thus we shall write, for example, ``$\xX \in \Tt$'' instead
of ``$\ueX \in \ttT$''). Additionally, $\Tt_n$ will stand for $d^{-1}(\{n\})$.

In m-towers a counterpart of the prime decomposition theorem (for natural numbers) holds, as shown
by

\begin{pro}{pd}
Let $\Tt$ be an m-tower.
\begin{enumerate}[\upshape(A)]
\item For each $\tT \in \Tt$ there is a finite system $\tT_1,\ldots,\tT_n$ of irreducible elements
 of $\Tt$ such that $\tT \equiv \bigoplus_{j=1}^n \tT_j$.
\item If $\tT_1,\ldots,\tT_n$ and $\sS_1,\ldots,\sS_k$ are two systems of irreducible elements
 of $\Tt$ for which $\bigoplus_{j=1}^n \tT_j \equiv \bigoplus_{j=1}^k \sS_j$, then $k = n$ and
 there exists a permutation $\tau$ of $\{1,\ldots,n\}$ such that $\sS_j \equiv \tT_{\tau(j)}$ for
 each $j$.
\end{enumerate}
\end{pro}
\begin{proof}
To show (A), it suffices to mimic a classical proof of the prime decomposition theorem for natural
numbers. We shall prove the assertion of (A) by induction on $d(\tT)$. It is obvious that $\tT$ is
irreducible provided $d(\tT) = 1$. Further, if $\tT$ is irreducible, it suffices to put $n \df 1$
and $\tT_1 \df \tT$. Finally, if $\tT$ is reducible, then $\tT \equiv \aA \oplus \bB$ for some $\aA,
\bB \in \Tt$. Then $d(\aA) < d(\tT)$ and $d(\bB) < d(\tT)$ and thus from the induction hypothesis
it follows that $\aA \equiv \bigoplus_{j=1}^p \aA_j$ and $\bB \equiv \bigoplus_{j=1}^q \bB_j$ for
some irreducible elements $\aA_1,\ldots,\aA_p,\bB_1,\ldots,\bB_q \in \Tt$. Now (T5) yields that $\tT
\equiv (\bigoplus_{j=1}^p \aA_j) \oplus (\bigoplus_{j=1}^q \bB_j)$ and we are done.\par
To show (B), we employ axioms (mT1)--(mT2) and (T6) and use induction on $\ell \df \max(n,k)$. When
$\ell = 1$, we have nothing to do. Now assume $\ell > 1$. Notice that for any two irreducible
elements $\xX$ and $\yY$ of $\Tt$, either $\xX \equiv \yY$ or $\xX \perp \yY$. Thus, if there was
no $j$ for which $\sS_1 \equiv \tT_j$, then we would deduce that $\sS_1 \perp \tT_j$ for all $j$ and
(mT1) would imply that $\sS_1 \perp \tT \df \bigoplus_{j=1}^n \tT_j$, which is false, because $\sS_1
\preccurlyeq \tT$. So, there is $\tau(1) \in \{1,\ldots,n\}$ for which $\sS_1 \equiv \tT_{\tau(1)}$.
In particular, $d(\sS_1) = d(\tT_{\tau(1)})$ and, consequently, $\sum_{j\neq\tau(1)} d(\tT_j) =
\sum_{j=2}^k d(\sS_j)$. Since at least one of these sums is positive, we see that both $n$ and $k$
are greater than $1$. Further, using (T6), one easily shows that $\tT \equiv \tT_{\tau(1)} \oplus
\tT'$ where $\tT' \df \bigoplus_{j\neq\tau(1)} \tT_j$. Finally, (mT2) (combined with (T5)) gives
$\tT' \equiv \bigoplus_{j=2}^k \sS_j$, and now the induction hypothesis finishes the proof.
\end{proof}

The proof of the following consequence of \PRO{pd} is left as an exercise.

\begin{cor}{pd}
Let $\tT_1,\ldots,\tT_n$ be irreducible members of an m-tower $\Tt$.
\begin{enumerate}[\upshape(a)]
\item If $\xX \in \Tt$ is such that $\xX \preccurlyeq \bigoplus_{j=1}^n \tT_j$, then there exists
 a nonempty set $J \subset \{1,\ldots,n\}$ for which $\xX \equiv \bigoplus_{j \in J} \tT_j$.
\item If $\sS_1,\ldots,\sS_p$ are irreducible elements of $\Tt$, then $\bigoplus_{j=1}^n \tT_j
 \perp \bigoplus_{k=1}^p \sS_k$ iff $\tT_j \perp \sS_k$ for all $j$ and $k$.
\end{enumerate}
\end{cor}

\begin{dfn}{C*tower}
Let $\Tt$ be a tower. A set $\Aa \subset \Tt$ is said to be
\begin{itemize}
\item \textit{unitarily invariant} if $\uuU . \Aa \subset \Aa$;
\item a \textit{semitower} if $\Aa$ is unitarily invariant and whenever $\tT, \sS \in \Tt$ are such
 that $\tT \oplus \sS \in \Aa$, then $\tT, \sS \in \Aa$; equivalently, $\Aa$ is a semitower iff each
 $\tT \in \Tt$ for which there is $\sS \in \Aa$ with $\tT \preccurlyeq \sS$ belongs to $\Aa$.
\end{itemize}
A function $f\dd \Aa \to \mmM$ (recall that $\mmM$ is the space of all single matrices) is called
\textit{compatible} if
\begin{itemize}
\item $d(f(\aA)) = d(\aA)$ for any $\aA \in \Aa$; and
\item $f(U . (\bigoplus_{j=1}^n \aA_j)) = U . (\bigoplus_{j=1}^n f(\aA_j))$ whenever $\aA_1,\ldots,
 \aA_n \in \Aa$ and $U \in \uuU_N$ with $N = \sum_{j=1}^n d(\aA_j)$ are such that $U .
 (\bigoplus_{j=1}^n \aA_j) \in \Aa$.
\end{itemize}
Since $\mmM$ is a topological space, we may, of course, speak of compatible maps. Additionally,
we call a compatible function $f\dd \Aa \to \mmM$ \textit{bounded} if
\begin{equation*}
(\|f\| \df\,) \sup_{\aA\in Aa} \|f(\aA)\| < \infty
\end{equation*}
(where $\sup(\varempty) \df 0$).\par
We denote by $C^*_{\Tt}(\Aa)$ the set of all $\mmM$-valued compatible maps defined on $\Aa$ that are
bounded. It is easy to see that $C^*_{\Tt}(\Aa)$ is a $C^*$-algebra when all algebraic operations
are defined pointwise; and $C^*_{\Tt}(\Aa)$ is unital provided $\Aa \neq \varempty$ (the unit
$j_{\Aa}$ of $C^*_{\Tt}(\Aa)$ is constantly equal to $I_n$ on each $\Aa \cap \Tt_n$).\par
The $C^*$-algebra $C^*(\Tt) \df C^*_{\Tt}(\Tt)$ is called the \textit{$C^*$-algebra over $\Tt$}.
Similarly, \textit{$C^*$-algebra over a pointed tower $(\Tt,\theta)$}, denoted by $C^*(\Tt,\theta)$,
consists of all maps $f \in C^*(\Tt)$ that vanish at $\theta$. It is easy to see that
$C^*(\Tt,\theta)$ is an ideal in $C^*(\Tt)$.
\end{dfn}

It is readily seen that a subtower of an m-tower is itself an m-tower, that a finite union
of semitowers (in any tower) is a semitower and that $C^*(\Ss) = C^*_{\Tt}(\Ss)$ for any subtower
$\Ss$ of a tower $\Tt$. A strong property of m-towers is formulated below.

\begin{thm}{extend}
Let $\Tt$ be an m-tower, $\Aa$ a closed set in $\Tt$ such that $\Aa \subset \core(\Tt)$ and let
$\Ss$ be a semitower of $\Tt$. Further, let $f \in C^*_{\Tt}(\Aa)$ and $g \in C^*_{\Tt}(\Ss)$ be
maps such that $f\bigr|_{\Aa \cap \Ss} = g\bigr|_{\Aa \cap \Ss}$. Then there is $h \in C^*(\Tt)$
which extends both $f$ and $g$ and $\|h\| = \max(\|f\|,\|g\|)$.
\end{thm}

In the proof of \THM{extend} we shall involve the next five lemmas, in which we assume $\Tt$ is
an m-tower. The proof of the first of them is left as an exercise.

\begin{lem}{unitary}
Let $\Aa$ be a closed set in a tower. Then $\uuU . \Aa$ is closed and any map $u \in C^*_{\Tt}(\Aa)$
admits a unique extension to a map $v \in C^*_{\Tt}(\uuU . \Aa)$. Moreover, $\|v\| = \|u\|$.
\end{lem}

The following is a key lemma.

\begin{lem}{key}
Let $\tT_1,\ldots,\tT_n$ be a system of irreducible elements of $\Tt$, $\tau$ be a permutation
of $\{1,\ldots,n\}$ and let $N \df \sum_{j=1}^n d(\tT_j)$. Further, let $A_1,\ldots,A_n$ be any
system of matrices such that
\begin{enumerate}[\upshape({a}x1)]
\item $d(A_j) = d(\tT_j)$ for each $j$; and
\item whenever $j, k \in \{1,\ldots,n\}$ and $V \in \uuU_{d(\tT_j)}$ are such that $V . \tT_j =
 \tT_k$, then $V . A_j = A_k$.
\end{enumerate}
Then, $U. (\bigoplus_{j=1}^n A_j) = \bigoplus_{j=1}^n A_{\tau(j)}$ for any $U \in \uuU_N$ for which
$U . (\bigoplus_{j=1}^n \tT_j) = \bigoplus_{j=1}^n \tT_{\tau(j)}$.
\end{lem}
\begin{proof}
First we assume $\tau$ is the identity. For any positive integers $p$ and $q$, let $U_{p,q}$ be
as specified in (T6). Let $\xi$ be a permutation of $\{1,\ldots,n\}$ such that for some finite
increasing sequence $\nu_0,\ldots,\nu_N$ of integers with $\nu_0 = 0$ and $\nu_N = n$ one has (for
any $s \in \{1,\ldots,N\}$):
\begin{enumerate}[({A}1)]
\item $\tT_{\xi(j)} \equiv \tT_{\xi(\nu(s))}$ whenever $\nu_{s-1} < j < \nu_s$; and
\item $\tT_{\xi(\nu(s))} \perp \tT_{\xi(\nu(s'))}$ whenever $s' \in \{1,\ldots,N\}$ differs from
 $s$.
\end{enumerate}
For each $s \in \{1,\ldots,N\}$, we put $\beta_s \df \nu(s) - \nu(s-1)$. Employing (T6), we see that
there is $W \in \uuU_N$ that is the product of a finite number of the matrices $U_{p,q}$ and
satisfies
\begin{equation}\label{eqn:aux3}
W . \Bigl(\bigoplus_{j=1}^n \tT_j\Bigr) = \bigoplus_{j=1}^n \tT_{\xi(j)}.
\end{equation}
A specific form of $W$ yields that also
\begin{equation}\label{eqn:aux4}
W . \Bigl( \bigoplus_{j=1}^n A_j\Bigr) = \bigoplus_{j=1}^n A_{\xi(j)}.
\end{equation}
Further, for any $s \in \{1,\ldots,N\}$ and each $j$ with $\nu_{s-1} < j \leqsl \nu_s$ we choose
$V_j \in \uuU_{d(\tT_{\xi(\nu_s)})}$ such that $V_j . \tT_{\xi(j)} = \tT_{\xi(\nu(s))}$ (see (A1)).
Put $V \df \bigoplus_{j=1}^n V_j$. We infer from (T5) that
\begin{equation}\label{eqn:aux5}
V . \Bigl(\bigoplus_{j=1}^n \tT_{\xi(j)}\Bigr) = \bigoplus_{s=1}^N (\beta_s \odot \tT_{\nu(s)})
\end{equation}
and from (ax2) that
\begin{equation}\label{eqn:aux6}
V . \Bigl(\bigoplus_{s=1}^N A_{\xi(s)}\Bigr) = \bigoplus_{s=1}^N (\beta_s \odot A_{\nu(s)}).
\end{equation}
Further, (A2) and (mT1) yield that $\beta_s \odot \tT_{\nu(s)} \perp \bigoplus_{s' \neq s}
(\beta_{s'} \odot \tT_{\nu(s')})$ for each $s \in \{1,\ldots,N\}$. So, one deduces from (mT3)--(mT4)
that $\stab(\bigoplus_{s=1}^N \beta_s \odot \tT_s) = \{\bigoplus_{s=1}^N (I_{d(\tT_s)} \otimes
U_s)\dd\ U_s \in \uuU_{\beta_s}\}$. Consequently, it is easy to verify that then
\begin{equation}\label{eqn:aux7}
U . \Bigl(\bigoplus_{s=1}^N (\beta_s \odot A_{\nu(s)})\Bigr) = \bigoplus_{s=1}^N (\beta_s \odot
A_{\nu(s)}) \qquad \Bigl(U \in \stab\Bigl(\bigoplus_{s=1}^N \beta_s \odot \tT_{\nu(s)}\Bigr)\Bigr).
\end{equation}
Now if $U \in \stab(\bigoplus_{j=1}^n \tT_j)$, then $V W U W^{-1} V^{-1} \in \stab(\bigoplus_{s=1}^N
\beta_s \odot \tT_{\nu(s)})$, by (T3), \eqref{eqn:aux3} and \eqref{eqn:aux5}. And a combination
of \eqref{eqn:aux7}, \eqref{eqn:aux6} and \eqref{eqn:aux4} yields that $U . (\bigoplus_{j=1}^n A_j)
= (\bigoplus_{j=1}^n A_j)$. So, in the case when $\tau$ is the identity, the proof is finished.\par
Now assume $\tau$ is arbitrary. We may find $W \in \uuU_N$ that is the product of a finite number
of the matrices $U_{p,q}$ and satisfies $W . (\bigoplus_{j=1}^n \tT_j) = \bigoplus_{j=1}^n
\tT_{\tau(j)}$ (by (T5)) as well as $W . (\bigoplus_{j=1}^n A_j) = \bigoplus_{j=1}^n A_j$. We see
hat if $U \in \uuU_N$ is as specified in the lemma, then $W^{-1} U \in \stab(\bigoplus_{j=1}^n
\tT_j)$. So, it follows from the first part of the proof that $(W^{-1} U) . (\bigoplus_{j=1}^n A_j)
= \bigoplus_{j=1}^n A_j$ and consequently $U . (\bigoplus_{j=1}^n A_j) = W . (\bigoplus_{j=1}^n
A_j)$, which completes the proof.
\end{proof}

\begin{lem}{extend}
If $\Ff$ is a closed set in $\Tt_N$, then every map $u \in C^*_{\Tt}(\Ff)$ extends to a map $v \in
C^*_{\Tt}(\Tt_N)$ such that $\|v\| = \|u\|$.
\end{lem}
\begin{proof}
It follows from \LEM{unitary} that $u$ extends to a map $u_1 \in C^*_{\Tt}(\uuU . \Ff)$ such that
$\|u_1\| = \|u\|$. We conclude from \LEM{para} that $\Tt_N$ is a normal topological space. Denote
by $\Omega$ the set of all matrices $A \in \mmM_N$ whose norms do not exceed $\|u\|$. Since $\Omega$
is convex and compact, it is a retract of $\mmM_N$. So, it follows from Tietze's extension theorem
that there is a map $w\dd \Tt_N \to \Omega$ which extends $u_1$. We define $v\dd \Tt_N \to \Omega$
by $v(\tT) \df \int_{\uuU_N} U^{-1} . w(U . \tT) \dint{\lambda(U)}$ where $\lambda$ is
the (probabilistic) Haar measure of $\uuU_N$ and the integral is understood entrywise. We leave
to the reader a verification of all desired properties of $v$.
\end{proof}

\begin{lem}{sub}
Let $\Dd$ be a closed set in $\Tt$ that contains all irreducible elements $\tT$ of $\Tt$ for which
there is $\dD \in \Dd$ with $\tT \preccurlyeq \dD$. Then the set
\begin{multline}\label{eqn:subtower}
\grp{\Dd} \df \Bigl\{U . \Bigl(\bigoplus_{j=1}^n \dD_j\Bigr)\dd\ n > 0,\ \dD_1,\ldots,\dD_n \in
\Dd \cap \core(\Tt),\\U \in \uuU_N \textup{ where } N = \sum_{j=1}^n d(\dD_j)\Bigr\}
\end{multline}
is a subtower of $\Tt$ that contains $\Dd$.
\end{lem}
\begin{proof}
It follows from the assumptions of the lemma, (T3), (T5) and \COR{pd} that $\grp{\Dd}$ contains
$\Dd$ and satisfies (ST2) and \eqref{eqn:oplus}. So, the main difficulty is the closedness
of $\grp{\Dd}$. Since the sets $\Tt_n$ are clopen (because $d$ is a map), it suffices to show that
$\grp{\Dd} \cap \Tt_N$ is closed for every $N > 0$. Note that $\grp{\Dd} \cap \Tt_N$ is the union
of a finite number of sets of the form $\uuU . (\bigoplus_{j=1}^k (\Dd \cap \Tt_{\nu_j}))$ where
$\sum_{j=1}^k \nu_j = N$. But each of the aforementioned sets is closed, thanks to (T1) and
\LEM{unitary}.
\end{proof}

\begin{lem}{semi}
For any semitower $\Dd$ of $\Tt$ there exists a subtower $\Ee$ such that each $u \in C^*_{\Tt}(\Dd)$
extends to a map $w \in C^*(\Ee)$ with $\|w\| = \|u\|$.
\end{lem}
\begin{proof}
Let $\Ee \df \grp{\Dd}$ where $\grp{\Dd}$ is given by \eqref{eqn:subtower}. It follows from
\LEM{sub} that $\Ee$ is a subtower. We define a function $w\dd \Ee \to \mmM$ by the rule:
\begin{equation*}
w\Bigl(U . \Bigl(\bigoplus_{j=1}^n \dD_j\Bigr)\Bigr) \df U . \Bigl(\bigoplus_{j=1}^n u(\dD_j)\Bigr)
\end{equation*}
where $n > 0$; $\dD_j \in \Dd$ are irreducible; and $U \in \uuU_N$ with $N = \sum_{j=1}^n d(\dD_j)$.
The main point is that $w$ is well defined. To show this, assume
\begin{equation}\label{eqn:aux8}
U . \Bigl(\bigoplus_{j=1}^n \tT_j\Bigr) = V . \Bigl(\bigoplus_{j=1}^k \sS_j\Bigr)
\end{equation}
(where $\tT_j, \sS_j \in \Dd$ are irreducible and $U$ and $V$ are unitary matrices of respective
degrees). We need to prove that
\begin{equation}\label{eqn:aux9}
U . \Bigl(\bigoplus_{j=1}^n u(\tT_j)\Bigr) = V . \Bigl(\bigoplus_{j=1}^k u(\sS_j)\Bigr).
\end{equation}
Relation \eqref{eqn:aux8} means, in particular, that $\bigoplus_{j=1}^n \tT_j \equiv
\bigoplus_{j=1}^k \sS_j$. So, we infer from \PRO{pd} that $k = n$ and there is a permutation $\tau$
of $\{1,\ldots,n\}$ such that $\sS_j \equiv \tT_{\tau(j)}$. Let $W_j \in \uuU_{d(\sS_j)}$ be such
that $W_j . \sS_j = \tT_{\tau(j)}$. Put $W \df \bigoplus_{j=1}^n W_j$. We conclude from (T5), (T3)
and \eqref{eqn:aux8} that
\begin{equation}\label{eqn:aux10}
(W V^{-1} U) . \Bigl(\bigoplus_{j=1}^n \tT_j\Bigr) = \bigoplus_{j=1}^n \tT_{\tau(j)}.
\end{equation}
Further, we claim that for $A_j \df u(\tT_j)$ conditions (ax1) and (ax2) of \LEM{key} are fulfilled.
Indeed, (ax1) is immediate, whereas (ax2) follows from the fact that $u \in C^*_{\Tt}(\Dd)$. So,
\LEM{key}, combined with \eqref{eqn:aux10}, yields that
\begin{equation*}
(W V^{-1} U) . (\bigoplus_{j=1}^n u(\tT_j)) = \bigoplus_{j=1}^n u(\tT_{\tau(j)}),
\end{equation*}
which simply leads us to \eqref{eqn:aux9}.\par
Further, it follows from the very definition of $w$ that $w$ is bounded, $\|w\| = \|u\|$, $w$
extends $u$; and that $w \in C^*(\Ee)$ provided $w$ is continuous. Taking these remarks into
account, it remains to establish the continuity of $w$. Notice that $w$ is continuous iff it is
so on each of the sets $\Ee_N (= \Ee \cap \Tt_N)$. Arguing as in the proof of \LEM{sub}, we see that
it is sufficient that $w$ be continuous on each set of the form $\uuU . (\bigoplus_{j=1}^n
\Ee_{\nu_j})$ (which is closed), which readily follows from axiom (T1) and \LEM{unitary}.
\end{proof}

\begin{proof}[Proof of \THM{extend}]
By induction, we shall construct sequences $\Ee_0,\Ee_1,\ldots$ and $h_0,h_1,\ldots$ of subtowers
of $\Tt$ and maps (respectively) such that, for each $k \leqsl 0$,
\begin{enumerate}[(1$_k$)]
\item $\Ee_k \supset \Ee_{k-1} \cup \bigcup_{j=1}^k \Tt_j$ provided $k > 0$; and $\Ee_0 \supset
 \Aa \cup \Ss$;
\item $h_k \in C^*(\Ee_k)$ and $h_k$ extends $h_{k-1}$ provided $k > 0$; and $h_0$ extends both $f$
 and $g$;
\item $\|h_k\| \leqsl R \df \max(\|f\|,\|g\|)$.
\end{enumerate}
It follows from \LEM{unitary} that there is $f_1 \in C^*_{\Tt}(\uuU . \Aa)$ which extends $f$ and
has the same norm. Notice that $\uuU . \Aa$ is a semitower (since $\uuU . \Aa \subset \core(\Tt)$)
and the maps $f_1$ and $g$ agree (because $\Ss$ is unitary invariant) and their union belongs
to $C^*_{\Tt}((\uuU . \Aa) \cup \Ss)$. So, $(\uuU . \Aa) \cup \Ss$ is a semitower as weel and
we conclude from \LEM{semi} that there is a subtower $\Ee_0$ of $\Tt$ and a map $h_0 \in C^*(\Ee_0)$
that extends both $f_1$ and $g$. We see that conditions (1$_0$)--(3$_0$) are fulfilled. Now assume
$\Ee_{k-1}$ and $h_{k-1}$ are already defined (for some $k > 0$). First apply \LEM{extend} with $N =
k$, $\Ff = \Ee_{k-1} \cap \Tt_k$ and $u = h_{k-1}\bigr|_{\Ff}$ to obtain a map $w_k \in
C^*_{\Tt}(\Tt_k)$ that agrees with $h_{k-1}$ on $\Ee_{k-1} \cap \Tt_k$ and satisfies $\|w_k\| \leqsl
R$. Observe that $\Dd_k \df \Ee_{k-1} \cup \Tt_k$ is a semitower (by (1$_{k-1}$)) and the union
$u_k$ of $w_k$ and $h_{k-1}$ belongs to $C^*_{\Tt}(\Dd_k)$ (because $\Ee_{k-1}$ is a semitower).
Next apply \LEM{semi} (with $\Dd = \Dd_k$ and $u = u_k$) to obtain $\Ee_k$ and $h_k$. Finally,
we define a function $h\dd \Tt \to \mmM$ by the rule: $h(\tT) \df h_n(\tT)$ for $\tT \in \Tt_n$ (see
(1$_n$)). We leave it to the reader that $h$ is a map we searched for.
\end{proof}

\begin{cor}{dist}
Let $\Tt$ be an m-tower.
\begin{enumerate}[\upshape(A)]
\item For any sequence $\tT_1,\tT_2,\ldots$ of mutually disjoint irreducible elements of $\Tt$ such
 that the set $\{j\dd\ d(\tT_j) = n\}$ is finite for any $n$, and for any sequence $A_1,A_2,\ldots$
 of matrices with $d(A_j) = d(\tT_j)\ (j \geqsl 1)$ and $M \df \sup_{j\geqsl1} \|A_j\| < \infty$,
 there is $f \in C^*(\Tt)$ such that $f(\tT_j) = A_j$ for all $j$ and $\|f\| = M$.
\item If $\Ss$ is a subtower of $\Tt$ and $\tT \in \core(\Tt) \setminus \Ss$, then for any matrix
 $A$ with $d(A) = d(\tT)$ and each $u \in C^*(\Ss)$ there is $v \in C^*(\Tt)$ that extends $u$ and
 satisfies $\|v\| = \max(\|A\|,\|u\|)$.
\item If $\tT_1$ and $\tT_2$ are two distinct elements of $\Tt$, then there exists $g \in C^*(\Tt)$
 for which $g(\tT_1) \neq g(\tT_2)$.
\end{enumerate}
\end{cor}
\begin{proof}
Both (A) and (B) are special cases of \THM{extend}. We turn to (C). First assume that $\tT_1
\not\equiv \tT_2$ (which is a simpler case). Let $\aA_1^{(1)},\ldots,\aA_{p_1}^{(1)}$ and
$\aA_1^{(2)},\ldots,\aA_{p_2}^{(2)}$ be two systems of mutually disjoint irreducible elements
of $\Tt$ such that, for $\epsi=1,2$, $\tT_{\epsi} = U_{\epsi} . (\bigoplus_{j=1}^{p_{\epsi}}
(\alpha_j^{(\epsi)} \odot \aA_j^{(\epsi)}))$ for some unitary matrix $U_{\epsi}$ and positive
integers $\alpha_1^{(\epsi)},\ldots,\alpha_{p_{\epsi}}^{(\epsi)}$ (consult \PRO{pd}). Since $\tT_1
\not\equiv \tT_2$, one concludes that either some of $\aA_j^{(1)}$ is disjoint from each
of $\aA_k^{(2)}$ (or conversely), or there are $k_1 \in \{1,\ldots,p_1\}$ and $k_2 \in
\{1,\ldots,p_2\}$ such that $\aA_{k_1}^{(1)} \equiv \aA_{k_2}^{(2)}$ and $\alpha_{k_1}^{(1)} \neq
\alpha_{k_2}^{(2)}$. Thus, we may assume, with no loss of generality, that either $\aA_1^{(1)}$ is
disjoint from each of $\aA_k^{(2)}$, or $\aA_1^{(1)} \equiv \aA_1^{(2)}$ and $\alpha_1^{(1)} \neq
\alpha_1^{(2)}$. In the former case we conclude that there is $g \in C^*(\Tt)$ that vanishes at each
$\aA_k^{(2)}$ and each $\aA_j^{(1)}$ with $j \neq 1$, and $g(\aA_1^{(1)}) \neq 0$. Then $g(\tT_1)
\neq 0 = g(\tT_2)$ and we are done. In the latter case we may find $g \in C^*(\Tt)$ that vanishes
at each $\aA_j^{(1)}$ and $\aA_k^{(2)}$ with $j > 1$ and $k > 1$, and $g(\aA_1^{(1)}) =
g(\aA_1^{(2)}) = I_n$ where $n \df d(\aA_1^{(1)})$. Then, for $\epsi = 1,2$, $g(\tT_{\epsi})$ is
a projection of rank $\alpha_{\epsi} n$ and hence $g(\tT_1) \neq g(\tT_2)$. This finishes the proof
in the case when $\tT_1 \not\equiv \tT_2$.\par
Finally, we assume $\tT_1 \equiv \tT_2$. Then we can find two unitary matrices $U_1$ and $U_2$, and
finite systems $\aA_1,\ldots,\aA_p$ and $\alpha_1,\ldots,\alpha_p$ of, respectively, mutually
disjoint irreducible elements of $\Tt$ and positive integers such that $\tT_{\epsi} = U_{\epsi} .
(\bigoplus_{j=1}^p (\alpha_j \odot \aA_j))$ for $\epsi=1,2$. Since $\tT_1 \neq \tT_2$, we see that
\begin{equation}\label{eqn:aux11}
U_2^{-1} U_1 \notin \stab\Bigl(\bigoplus_{j=1}^p (\alpha_j \odot \aA_j)\Bigr).
\end{equation}
Axioms (mT1)--(mT4) imply that
\begin{equation}\label{eqn:aux12}
\stab\Bigl(\bigoplus_{j=1}^p (\alpha_j \odot \aA_j)\Bigr) = \Bigl\{\bigoplus_{j=1}^p (I_{d(\aA_j)}
\otimes V_j)\dd\ V_j \in \uuU_{\alpha_j}\Bigr\}.
\end{equation}
Now take a system $A_1,\ldots,A_p$ of mutually disjoint (that is, mutually unitarily inequivalent)
irreducible matrices such that $d(A_j) = d(\aA_j)$. Let $g \in C^*(\Tt)$ be such that $g(\aA_j) =
A_j$ for each $j$. We claim that then $g(\tT_1) \neq g(\tT_2)$, which follows from
\eqref{eqn:aux11}, \eqref{eqn:aux12} and the relation $\stab(\bigoplus_{j=1}^p (\alpha_j \odot
\aA_j)) = \WwW'(\bigoplus_{j=1}^p (\alpha_j \odot A_j)) \cap \uuU_N$ (where $N = \sum_{j=1}^p
d(A_j)$). The details are left to the reader.
\end{proof}

\begin{lem}{core}
The core of any tower is an open set. If $\Tt$ is an m-tower, then an element $\xX \in \Tt$ is
irreducible iff $\stab(\xX)$ consists only of scalar multiples of the unit matrix.
\end{lem}
\begin{proof}
Let $\Ss$ be a tower. Note that $\core(\Ss)$ is open iff $\Ss_N \setminus \core(\Ss)$ is closed for
any $N$. To show the latter statement, observe that $\Ss_N \setminus \core(\Ss)$ coincides with
a finite union of all sets of the form $\uuU . (\bigoplus_{j=1}^n \Ss_{\nu_j})$ where $\nu_1,\ldots,
\nu_n$ are positive integers which sum up to $N$, and $n > 1$. It follows from (T1) and \LEM{core}
that each of the aforementioned sets is closed, which proves the first claim of the lemma. In order
to show the second, we only need to focus on the `if' part (because the `only if' part is included
in (mT4)). To this end, assume $\xX$ is reducible. Then there are a matrix $U \in \uuU_{d(\xX)}$ and
two elements $\tT$ and $\sS$ of $\Tt$ such that $\xX = U . (\tT \oplus \sS)$. One may then easily
deduce from (T3) and (T5) that $\stab(\xX)$ contains matrices different from scalar multiplies
of the unit matrix.
\end{proof}

We skip the proofs of the next two results (since they are immediate consequences of \COR{dist} and
\LEM{core}).

\begin{cor}{repr}
For any element $\tT$ of an m-tower $\Tt$, the assignment $f \mapsto f(\tT)$ correctly defines
a unital finite-dimensional representation $\pi_{\tT}\dd C^*(\Tt) \to \mmM_{d(\tT)}$. Moreover,
$\pi_{\tT}$ is irreducible iff $\tT$ is irreducible; and the assignment $\tT \mapsto \pi_{\tT}$ is
one-to-one.
\end{cor}

\begin{cor}{center}
For any m-tower $\Tt$, the center of $C^*(\Tt)$ coincides with the set of all $u \in C^*(\Tt)$ such
that $u(\tT)$ is a scalar multiple of the unit matrix for any $\tT \in \core(\Tt)$.
\end{cor}

The next result explains the importance of axioms (mT1)--(mT4).

\begin{pro}{char}
A tower is isomorphic to a standard tower iff it is a proper m-tower.
\end{pro}
\begin{proof}
The `only if' part readily follows from \PRO{std-m-tower}. Here we shall focus only on the `if'
part. We shall apply \THM{extend}. Assume $\Tt$ is a proper m-tower. In particular, $\Tt_n$ is
a compact space for each $n$. Denote by $\Lambda$ the closed unit ball of $C^*(\Tt)$ and define
$\Phi\dd \Tt \to \mmM[\Lambda]$ by $(\Phi(\tT))(\lambda) \df \lambda(\tT)\ (\tT \in \Tt,\ \lambda
\in \Lambda)$. It follows from the very definition of $\Phi$ that $\Phi$ is well defined, continuous
and satisfies axioms (M1)--(M3) of a morphism. We claim that $\Phi$ is a closed embedding.
To convince oneself of that, observe that it suffices that $\Phi$ is one-to-one on each of the sets
$\Tt_n$ (because $\Tt_n$ is compact and clopen, and $\Phi(\Tt_n) \subset \mmM_n^{\Lambda}$), which
is covered by item (C) of \COR{dist}. Finally, let us briefly show that $\Tt' \df \Phi(\Tt)$ is
a standard tower. Axioms (ST0)--(ST2) and implication ``$\impliedby$'' in (ST3) are transparent.
To end the proof, assume $\Phi(\tT) = \aA \oplus \bB$ where $\tT \in \Tt$, and $\aA =
(A_{\lambda})_{\lambda\in\Lambda}$ and $\bB = (B_{\lambda})_{\lambda\in\Lambda}$ are two members
of $\mmM[\Lambda]$. We only need to prove that $\aA, \bB \in \Tt'$. To this end, we denote by $\Ss$
the largest standard tower contained in $\mmM[\Lambda]$, that is, $\Ss$ consists of all
$(X_{\lambda})_{\lambda\in\Lambda} \in \mmM[\Lambda]$ with $\|X_{\lambda}\| \leqsl 1$ for each
$\lambda$. (It is easy to check that $\Ss$ is indeed a standard tower and that $\Tt' \subset \Ss$.)
It follows from \PRO{pd} that there are irreducible elements $\tT_1,\ldots,\tT_n$ of $\Tt$ such that
$\tT \equiv \bigoplus_{j=1}^n \tT_j$. Since $\Phi$ is one-to-one, one easily infers that each
of $\Phi(\tT_j)$ is irreducible in $\Ss$ (because $\stab(\Phi(\xX)) = \stab(\xX)$ for any $\xX \in
\Tt$). Observe that $(\aA \oplus \bB =) \Phi(\tT) \equiv \bigoplus_{j=1}^n \Phi(\tT_j)$. So,
\COR{pd} (applied for the m-tower $\Ss$) yields that there are two nonempty subets $J$ and $J'$
of $\{1,\ldots,n\}$ such that $\aA \equiv \bigoplus_{j \in J} \Phi(\tT_j) (=
\Phi(\bigoplus_{j \in J} \tT_j))$ and $\bB \equiv \bigoplus_{j \in J'} \Phi(\tT_j) (=
\Phi(\bigoplus_{j \in J'} \tT_j))$. These two relations easily imply that both $\aA$ and $\bB$ are
values of $\Phi$, which is equivalent to the fact that $\aA, \bB \in \Tt'$.
\end{proof}

The following simple result will prove useful in the sequel. We skip its proof.

\begin{lem}{norm}
Let $\Tt$ be an m-tower. Then, for every $f \in C^*(\Tt)$,
\begin{equation*}
\|f\| = \sup \{\|f(\tT)\|\dd\ \tT \in \core(\Tt)\}.
\end{equation*}
\end{lem}

\begin{rem}{superfl}
A careful reader have noticed that all results of the section are true for any tower that satisfies
conditions (mT1)--(mT4) and whose underlying topological space is normal (which was essential
in \LEM{extend}). Property $(\sigma)$ was added by us only in order to make m-towers similar
to standard towers. For example, according to \LEM{para}, m-tower are paracompact. In \THM{inv}
we shall show that the class of all algebras of the form $C^*(\Tt)$ (where $\Tt$ is an m-tower)
contains, up to $*$-isomorphism, all unital $C^*$-algebras that are inverse limits of subhomogeneous
$C^*$-algebras and only such algebras. The proof of the aforementioned result shows also that
$C^*(\Tt)$ is $*$-isomorphic to the inverse limit of a sequence of unital subhomogeneous
$C^*$-algebras for any tower $\Tt$. Thus, it is a reasonable idea to distinguish towers which,
on the one hand, form a rich class (with respect to applications in operator algebras), and,
on the other hand, are as good as possible. This way of thinking motivated us to add
to the definition of m-towers property $(\sigma)$.\par
It is also worth adding here a note that we have never used axiom (T4). This is because of the fact
that (T4) is a consequence of (mT1)--(mT4) and \PRO{pd}. So, in the definition of m-towers, one may
omit axiom (T4).
\end{rem}

\section{Essentially locally compact m-towers}

\begin{dfn}{ess_loc_comp}
An m-tower $\Tt$ is said to be \textit{essentially locally compact} (briefly, \textit{elc})
if the space $\overline{\core}(\Tt)$ is locally compact. If this happens, the $C^*$-algebra
$C^*_0(\Tt)$ is defined as the $*$-subalgebra of $C^*(\Tt)$ consisting of all maps whose
restrictions to $\overline{\core}(\Tt)$ vanish at infinity. More precisely, a map $f \in C^*(\Tt)$
belongs to $C^*_0(\Tt)$ iff for each $\epsi > 0$ there is a compact set $K \subset
\overline{\core}(\Tt)$ such that $\|f(\tT)\| \leqsl \epsi$ for any $\tT \in \overline{\core}(\Tt)
\setminus K$.\par
Similarly, we call a pointed m-tower $(\Tt,\theta)$ \textit{essentially locally compact}
(\textit{elc}) provided the tower $\Tt$ is so. In that case we put $C^*_0(\Tt,\theta) \df
C^*(\Tt,\theta) \cap C^*_0(\Tt)$.
\end{dfn}

The reader should notice that every proper m-tower is elc and that a subtower of an elc m-tower is
elc as well. Observe also that for an elc tower $\Tt$, $C^*_0(\Tt) = C^*(\Tt)$ \iaoi{} $\Tt$ has
finite height.\par
For $C^*$-algebras of the form $C^*_0(\Tt)$ where $\Tt$ is an elc m-tower all ideals as well as all
nondegenerate finite-dimensional representations can simply be characterized. The results on these
topics shall be derived from the next result, which is a special case of our
\textit{Stone-Weierstrass theorem for elc m-towers} (see \THM{SW!} at the end of the section).

\begin{thm}{SW}
Let $\Tt$ be an elc m-tower and $\EeE$ a $*$-subalgebra of $C^*_0(\Tt)$. Assume for any irreducible
element $\tT$ of $\Tt$ one the following two conditions is fulfilled:
\begin{enumerate}[\upshape(1$_{\tT}$)]
\item either $g(\tT) = 0$ for any $g \in \EeE$; or
\item whenever $\sS \in \Tt$ is irreducible and $\sS \perp \tT$, then there is $g \in \EeE$ with
 $g(\tT) = I_{d(\tT)}$ and $g(\sS) = 0$.
\end{enumerate}
Then the uniform closure of $\EeE$ in $C^*_0(\Tt)$ consists of all maps $u \in C^*_0(\Tt)$ such that
\begin{itemize}
\item[($*$)] for any $\xX \in \core(\Tt)$ there exists $v \in \EeE$ with $v(\xX) = u(\xX)$.
\end{itemize}
\end{thm}

As we have already said, the above result is a special case of \THM{SW!}. But, it is also a key step
in the proof of the latter result.\par
In the proof of \THM{SW} we shall apply the next two results, recently discovered by us. The first
of them is another variation of the Stone-Weierstrass theorem.

\begin{thm}[\cite{pn2}]{S-W}
Let $K$ be a compact space and $\ddD$ be a unital $C^*$-algebra. Let $\LlL$ be a $*$-subalgebra
of $C(K,\ddD)$ such that for any two points $x$ and $y$ of $K$, one of the following two conditions
is fulfilled:
\begin{enumerate}[\upshape(SW1)]
\item either there exists $u \in \LlL$ such that $u(x)$ and $u(y)$ are normal elements of $\ddD$
 with disjoint spectra; or
\item the spectra of $f(x)$ and $f(y)$ \textup{(}computed in $\ddD$\textup{)} coincide for any
 selfadjoint $f \in \LlL$.
\end{enumerate}
Then the \textup{(}uniform\textup{)} closure of $\LlL$ in $C(K,\ddD)$ coincides with the $*$-algebra
$\Delta_2(\LlL)$ of all maps $u \in C(K,\ddD)$ such that for any $x, y \in K$ and each $\epsi > 0$
there exists $v \in \LlL$ with $\|v(z) - u(z)\| < \epsi$ for $z \in \{x,y\}$.
\end{thm}

(The above theorem is a special case of a result formulated and proved in \cite{pn2}.) For a more
general result in this direction the reader is referred to celebrated papers by Longo \cite{lon} and
Popa \cite{pop} (consult also \cite{gli}, \S4.7 in \cite{sak} and Chapter~11 in \cite{di2}).\par
The second tool reads as follows.

\begin{thm}[\cite{pn3}]{C0}
Let $\Omega$ be a locally compact space and $\bbB$ a countable collection of pairwise disjoint Borel
subsets of $\Omega$ that cover $\Omega$. Further, let $\aaA$ be a $C^*$-algebra and $V$
a $*$-subalgebra of $C_0(\Omega,\aaA)$. The \textup{(}uniform\textup{)} closure of $V$ consists
of all maps $f \in C_0(\Omega,\aaA)$ such that $f\bigr|_L$ belongs to the uniform closure
of $V\bigr|_L \df \{g\bigr|_L\dd\ g \in V\} \subset C(L,\aaA)$ for each $L \subset \Omega$ such that
\begin{itemize}
\item[($**$)] the set $L \cap B$ is compact for each $B \in \bbB$ and nonempty only for a finite
 number of such $B$.
\end{itemize}
\end{thm}

For similar, but much stronger and more general results than \THM{C0}, consult \cite{pn3}.\par
For a better transparency, let us separate a part of the proof of \THM{SW} in the lemma below.

\begin{lem}{Borel}
Let $\Tt$ be an elc m-tower. For a finite monotone increasing sequence $\nu = (\nu_1,\ldots,\nu_n)$
of positive integers, put
\begin{equation}\label{eqn:B}
B(\nu) \df \Bigl\{U . \Bigl(\bigoplus_{j=1}^n \tT_j\Bigr)\dd\ U \in \uuU_{|\nu|},\ \tT_j \in
\core(\Tt) \cap \Tt_{\nu_j}\Bigr\} \cap \overline{\core}(\Tt)
\end{equation}
where $|\nu| \df \sum_{j=1}^n \nu_j$. Then all sets of the form $B(\nu)$ \textup{(}where $\nu$ runs
over all finite monotone increasing sequences of positive integers\textup{)} are pairwise disjoint,
Borel and cover $\overline{\core}(\Tt)$.
\end{lem}
\begin{proof}
(To avoid misunderstandings, we recall that a sequence $(\nu_1,\ldots,\nu_n)$ is monotone increasing
if $\nu_{j-1} \leqsl \nu_j$ for any $j \in \{1,\ldots,n\}$ different from $1$.) For simplicity,
denote by $\bbB$ the collection of all sets of the form $B(\nu)$. It follows from \PRO{pd} and (T6)
that $\bbB$ consists of pairwise disjoint sets that cover $\overline{\core}(\Tt)$. So, we only need
to show that $B(\nu)$ is a Borel set. To this end, denote by $D(\nu)$ and $F(\nu)$ the sets,
respectively, $\{U . \bigoplus_{j=1}^n \tT_j\dd\ U \in \uuU_{|\nu|},\ \tT_j \in \core(\Tt)
\cap \Tt_{\nu_j}\}$ and $\{U . (\bigoplus_{j=1}^n \tT_j)\dd\ U \in \uuU_{|\nu|},\ \tT_j \in
\Tt_{\nu_j}\}$ (where $n$ is the number of entries of $\nu$) and note that
\begin{itemize}
\item $B(\nu) = D(\nu) \cap \overline{\core}(\Tt)$;
\item $F(\nu)$ is closed in $\Tt$ (see the proof of \LEM{sub});
\item $D(\nu) \subset F(\nu)$;
\item $\bigcup_{\mu} F(\mu) = \Tt$ where $\mu$ runs over all finite monotone increasing sequences
 of positive integers.
\end{itemize}
Taking into account the above remarks, in order to conclude that $B(\nu)$ is Borel, it is enough
to verify that $D(\nu) \cap F(\mu) = \varempty$ for distinct sequences $\nu$ and $\mu$ (because then
$D(\nu) = F(\nu) \setminus \bigcup_{\mu\neq\nu} F(\mu)$), which simply follows e.g.\ from \PRO{pd}.
\end{proof}

\begin{proof}[Proof of \THM{SW}]
Our only task is to prove that each map $u \in C^*_0(\Tt)$ that satisfies ($*$) belongs
to the uniform closure of $\EeE$. First of all, observe that, under the assumptions of the theorem,
($*$) is equivalent to:
\begin{itemize}
\item[($*$')] for any $\xX, \yY \in \core(\Tt)$ there exists $v \in \EeE$ with $v(\xX) = u(\xX)$ and
 $v(\yY) = u(\yY)$.
\end{itemize}
Indeed, if either $\xX \equiv \yY$ or ($1_{\xX}$) (resp.\ ($1_{\yY}$)) holds, the conclusion
of ($*$') for $\xX$ and $\yY$ readily follows from ($*$) applied to $\yY$ (resp.\ to $\xX$).
On the other hand, if $\xX \perp \yY$ and both conditions ($2_{\xX}$) and ($2_{\yY}$) are fulfilled,
there are functions $g_1,g_2,v_1,v_2 \in \EeE$ with $g_1(\xX) = I_{d(\xX)}$, $g_1(\yY) = 0$,
$g_2(\yY) = I_{d(\yY)}$, $g_2(\xX) = 0$, $v_1(\xX) = u(\xX)$ and $v_2(\yY) = u(\yY)$ (the last two
properties are derived from ($*$)). Then $v \df g_1 v_1 + g_2 v_2$ belongs to $\EeE$ and satisfies
$v(\xX) = u(\xX)$ and $v(\yY) = u(\yY)$.\par
For simplicity, put $\Omega \df \overline{\core}(\Tt)$ and let
$\aaA$ stand for the product $\prod_{n=1}^{\infty} \mmM_n$ of the $C^*$-algebras $\mmM_n$. For each
$A \in \mmM_k$, we denote by $A^{\#}$ the element $(X_n)_{n=1}^{\infty}$ of $\aaA$ such that $X_k =
A$ and $X_n = 0$ for $n \neq k$. Further, for any $u \in C^*_0(\Tt)$ we define a function
$\Psi(u)\dd \Omega \to \aaA$ by $\Psi(u)(\tT) = (u(\tT))^{\#}$. In this way one obtains
a $*$-homomorphism $\Psi\dd C^*_0(\Tt) \to C_0(\Omega,\aaA)$. What is more, \LEM{norm} asserts that
$\Psi$ is isometric.\par
Let $\bbB$ be the family of all sets $B(\nu)$ defined by \eqref{eqn:B} (where $\nu$ runs over all
finite sequences of positive integers). \LEM{Borel} yields that $\bbB$ is a countable collection
of pairwise disjoint Borel subsets of $\Omega$ that cover $\Omega$. Fix $u \in C^*_0(\Tt)$ for which
($*$') holds. Instead of showing that $u$ belongs to the uniform closure of $\EeE$, it suffices
to check that $w \df \Psi(u)$ belongs to the uniform closure of $V \df \Psi(\EeE)$. We now employ
\THM{C0}. Let $L \subset \Omega$ satisfy ($**$). Our only task is to show that $w\bigr|_L$ belongs
to the uniform closure of $V\bigr|_L$. Equivalently, we only have to show that
\begin{itemize}
\item[(\club)] there is a sequence $v_1,v_2,\ldots$ of elements of $\EeE$ such that
 \begin{equation*}
 \lim_{n\to\infty} \sup\{\|v_n(\tT) - u(\tT)\|\dd\ \tT \in L\} = 0.
 \end{equation*}
\end{itemize}
Let $K$ consist of all $\tT \in \core(\Tt)$ for which there exists $\sS \in L$ with $\tT
\preccurlyeq \sS$. We claim that
\begin{equation}\label{eqn:aux13}
K \textup{ is compact}.
\end{equation}
We can prove this as follows. We infer from ($**$) that there are only a finite number of (finite)
sequences $\nu$ for which $L \cap B(\nu)$ is nonempty. Let $\nu^{(1)},\ldots,\nu^{(p)}$ denote all
such sequences. For $j \in \{1,\ldots,p\}$, put $L_j \df L \cap B(\nu^{(j)})$ and denote by $K_j$
the set of all $\tT \in \core(\Tt)$ for which there is $\sS \in L_j$ with $\tT \preccurlyeq \sS$.
Then $L = \bigcup_{j=1}^p L_j$ and each of $L_j$ is compact (see ($**$)). Observe that $K =
\bigcup_{j=1}^p K_j$ and thus, to conclude \eqref{eqn:aux13}, it suffices to show that the sets
$K_j$ are compact. To this end, fix $j$, put (for simplicity) $\nu \df \nu^{(j)}$, $F \df \{U .
\tT\dd\ U \in \uuU_{|\nu|},\ \tT \in L_j\}$ and express $\nu$ in the form $(\nu_1,\ldots,\nu_p)$.
Note that $F$ is a compact subset of $B(\nu)$. Further, let $\Delta$ consist of all $(\tT_1,\ldots,
\tT_p) \in \Tt_{\nu_1} \times \ldots \times \Tt_{\nu_p}$ for which $\bigoplus_{k=1}^p \tT_k \in F$.
It follows from (T1) that $\Delta$ is compact. Finally, since $F \subset B(\nu)$, \PRO{pd} implies
that $\Delta \subset (\core(\Tt))^p$. So, we conclude that $K_j = \bigcup_{k=1}^p \pr_k(\Delta)$
(see \COR{pd}) where $\pr_k\dd \Delta \to \Tt$ is the projection onto the $k$th coordinate, which
finishes the proof of \eqref{eqn:aux13}.\par
We come back to the main part of the proof. Observe that
\begin{equation}\label{eqn:aux14}
\|f\|_L = \|f\|_K
\end{equation}
for any $f \in C^*_0(\Tt))$ (where, for $A \subset \Tt$, $\|f\|_A \df \sup\{\|f(\aA)\|\dd\ \aA \in
A\}$). Indeed, for each $\xX \in K$ there are $\yY, \zZ \in \Tt$ such that $\xX \oplus \yY \equiv
\zZ \in L$ and then $\|f(\xX)\| \leqsl \|f(\zZ)\| \leqsl \|f\|_L$ for any such $f$; conversely,
if $\zZ \in L$, there are irreducible elements $\xX_1,\ldots,\xX_s \in \Tt$ for which $\zZ \equiv
\bigoplus_{j=1}^s \xX_j$, and then $\xX_1,\ldots,\xX_s \in K$ and $\|f(\zZ)\| \leqsl
\max(\|f(\xX_1)\|,\ldots,\|f(\xX_s)\|) \leqsl \|f\|_K$ for all $f \in C^*_0(\Tt)$.\par
It follows from \eqref{eqn:aux13} that there is $N > 0$ such that $K \subset \bigcup_{n=1}^N \Tt_n$.
Put $R \df N!$ and define a $C^*$-algebra $\ddD$ as $\mmM_R$. For each $f \in C^*_0(\Tt)$ let
$\Phi(f)\dd K \to \ddD$ be given by $(\Phi(f))(\xX) \df (R / d(\xX)) \odot f(\xX)$. It is easy
to see that in this way one obtains a $*$-homomorphism $\Phi\dd C^*_0(\Tt) \to C(K,\ddD)$. Moreover,
$\|\Phi(f)\| = \|f\|_K$ for any $f \in C^*_0(\Tt)$ which, combined with \eqref{eqn:aux14}, yields
\begin{equation}\label{eqn:aux15}
\|\Phi(f)\| = \|f\|_L \qquad (f \in C^*(\Tt)).
\end{equation}
The above connection lead us to a conclusion that (\club) is equivalent to
\begin{itemize}
\item[(\spade)] $g$ belongs to the uniform closure (in $C(K,\ddD)$) of $\LlL$
\end{itemize}
where $g \df \Phi(u)$ and $\LlL \df \Phi(\EeE)$. To show (\spade), we employ \THM{S-W}. To this end,
let us check that for any $\xX, \yY \in K$ one of (SW1) or (SW2) holds. We know that $\xX, \yY \in
\core(\Tt)$. So, either $\xX \equiv \yY$ (and in that case (SW2) holds) or $\xX \perp \yY$.
In the latter case, either both the conditions (1$_{\xX}$) and (1$_{\yY}$) are fulfilled and then
(SW2) holds, or at least one of conditions (2$_{\xX}$) and (2$_{\yY}$) is satisfied and then (SW1)
is fulfilled. So, we infer from \THM{S-W} that (\spade) holds provided $g \in \Delta_2(\LlL)$.
Finally, that $g$ belongs to $\Delta_2(\LlL)$ may simply be deduced from ($*$'). Thus, the proof is
complete.
\end{proof}

\begin{dfn}{zero}
The \textit{vanishing tower} of an elc m-tower $\Tt$, denoted by $\Tt_{(0)}$, is defined as the set
of all $\tT \in \Tt$ such that $f(\tT) = 0$ for any $f \in C^*_0(\Tt)$. It is easy to see that
$\Tt_{(0)}$ is indeed a (possibly empty) subtower.\par
Since $C^*_0(\Tt) = C^*(\Tt)$ for proper m-towers of finite height $\Tt$, \COR{dist} implies that
$\Tt_{(0)} = \varempty$ for all such m-towers.
\end{dfn}

The next consequence of \THM{SW} shall be used in the proof of the first part of \THM{subh1}.

\begin{cor}{dense}
A $*$-subalgebra $\EeE$ of $C^*_0(\Tt)$ \textup{(}where $\Tt$ is an elc m-tower\textup{)} is dense
in $C^*_0(\Tt)$ iff for any $\tT \in \core(\Tt) \setminus \Tt_{(0)}$:
\begin{enumerate}[\upshape(d1)]
\item $\{f(\tT)\dd\ f \in \EeE\} = \mmM_{d(\tT)}$; and
\item if $\sS \in \core(\Tt) \setminus \Tt_{(0)}$ is disjoint from $\tT$, then there exists
 $g \in \EeE$ for which $g(\tT) = I_{d(\tT)}$ and $g(\sS) = 0$.
\end{enumerate}
\end{cor}

The proof of \COR{dense} is left to the reader (to show the necessity of (d1), use \COR{dist} and
the fact that $C^*_0(\Tt)$ is an ideal in $C^*(\Tt)$).

\begin{cor}{ideal}
For any ideal $\JjJ$ in $C^*_0(\Tt)$ \textup{(}where $\Tt$ is an elc m-tower\textup{)} there exists
a unique subtower $\Ss$ of $\Tt$ such that $\Ss \supset \Tt_{(0)}$ and $\JjJ$ coincides with
the ideal $\JjJ_{\Ss}$ of all maps $f \in C^*_0(\Tt)$ that vanish at each point of $\Ss$.
\end{cor}
\begin{proof}
Let $\Ss$ be the set of all points $\tT \in \Tt$ which each map from $\JjJ$ vanishes at. It is
easily seen that $\Ss \supset \Tt_{(0)}$ and $\Ss$ is a subtower; and that $\JjJ \subset
\JjJ_{\Ss}$. We shall now show the reverse inclusion. To this end, observe that $C^*_0(\Tt)$ is
an ideal in $C^*(\Tt)$ and, consequently, $\JjJ$ is an ideal in $C^*(\Tt)$ as well. Fix $\tT \in
\core(\Tt) \setminus \Ss$. We shall check that condition (2$_{\tT}$) of \THM{SW} if fulfilled (with
$\JjJ$ in place of $\EeE$). We know from \COR{dist} that the map $C^*(\Tt) \ni f \mapsto f(\tT) \in
\mmM_{d(\tT)}$ is surjective. Thus, $\JjJ(\tT) \df \{f(\tT)\dd\ f \in \JjJ\}$ is an ideal
in $\mmM_{d(\tT)}$. Since this ideal is nonzero (because $\tT \notin \Ss$), we infer that
\begin{equation}\label{eqn:aux16}
\JjJ(\tT) = \mmM_{d(\tT)}.
\end{equation}
So, there exists $g_0 \in \JjJ$ such that $g_0(\tT) = I_{d(\tT)}$. Now let $\sS \in \core(\Tt)$ be
disjoint from $\tT$. Then it follows from \COR{dist} that there exists $g_1 \in C^*(\Tt)$ with
$g_1(\tT) = I_{d(\tT)}$ and $g_1(\sS) = 0$. It suffices to put $g \df g_0 g_1$ to obtain a map from
$\JjJ$ such that $g(\tT) = I_{d(\tT)}$ and $g(\sS) = 0$. Hence, \THM{SW} yields that each map $u \in
C^*_0(\Tt)$ for which condition ($*$) holds with $\EeE$ replaced by $\JjJ$ belongs to $\JjJ$ (which
is closed). But \eqref{eqn:aux16} holds for any $\tT \in \core(\Tt) \setminus \Ss$ and therefore
($*$) is (trivially) fulfilled for all $u \in \JjJ_{\Ss}$. Consequently, $\JjJ_{\Ss} = \JjJ$.\par
To show the uniqueness of $\Ss$, assume $\Ss_1$ and $\Ss_2$ are two subtowers of $\Tt$ that contain
$\Tt_{(0)}$ and satisfy $\Ss_1 \not\subset \Ss_2$. Then there exists $\tT \in \Ss_1 \setminus \Ss_2$
which is irreducible. Now \THM{extend} implies that there exists $u \in C^*(\Tt)$ such that $u$
vanishes at each point of $\Ss_2$ and $u(\tT) = I_{d(\tT)}$. Moreover, since $\tT \notin \Tt_{(0)}$,
there is $v \in C^*_0(\Tt)$ which does not vanish at $\tT$. Then $uv \in \JjJ_{\Ss_2} \setminus
\JjJ_{\Ss_1}$ and we are done.
\end{proof}

\begin{cor}{nondeg}
For every nondegenerate finite-dimensional representation $\pi$ of $C^*_0(\Tt)$ \textup{(}where
$\Tt$ is an elc m-tower\textup{)} there exists a unique element $\sS \in \Tt$ for which $\pi(f) =
f(\sS)$ for any $f \in C^*_0(\Tt)$. What is more, $\sS \perp \zZ$ for each $\zZ \in \Tt_{(0)}$.\par
Conversely, for any $\tT \in \Tt$ which is disjoint from any element of $\Tt_{(0)}$,
the representation $C^*_0(\Tt) \ni f \mapsto f(\tT) \in \mmM_{d(\tT)}$ is nondegenerate.
\end{cor}
\begin{proof}
Let $\tT \in \core(\Tt) \setminus \Tt_{(0)}$. Since $C^*_0(\Tt)$ is an ideal in $C^*(\Tt)$,
\COR{dist} yields that
\begin{equation}\label{eqn:non0}
\{f(\tT)\dd\ f \in C^*_0(\Tt)\} = \mmM_{d(\tT)}
\end{equation}
(since the left-hand side of \eqref{eqn:non0} is a nonzero ideal in $\mmM_{d(\tT)}$). Now combining
the above formula again with \COR{dist}, one infers that for any system $\tT_1,\ldots,\tT_n$
of mutually disjoint elements of $\core(\Tt) \setminus \Tt_{(0)}$,
\begin{equation}\label{eqn:2non0}
\{(f(\tT_1),\ldots,f(\tT_n))\dd\ f \in C^*_0(\Tt)\} = \mmM_{d(\tT_1)} \times \ldots \times
\mmM_{d(\tT_n)},
\end{equation}
and consequently,
\begin{equation}\label{eqn:values}
\Bigl\{f\Bigl(\bigoplus_{j=1}^n (\alpha_j \odot \tT_j)\Bigr)\dd\ f \in C^*_0(\Tt)\Bigr\} =
\Bigl\{\bigoplus_{j=1}^n (\alpha_j \odot A_j)\dd\ A_j \in \mmM_{d(\tT_j)}\Bigr\}
\end{equation}
for any positive integers $\alpha_1,\ldots,\alpha_n$. This shows, thanks to \PRO{pd}, the second
claim of the corollary.\par
Now let $\pi\dd C^*_0(\Tt) \to \mmM_N$ be a nonzero irreducible representation. \COR{ideal} implies
that there is a subtower $\Ss$ of $\Tt$ that contains $\Tt_{(0)}$ and satisfies
\begin{equation}\label{eqn:ker}
\ker(\pi) = \JjJ_{\Ss}.
\end{equation}
Further, since $\Tt_{(0)} \subsetneq \Ss$ and $\Tt_{(0)}$ is a subtower, we conclude that $\Ss
\setminus \Tt_{(0)}$ contains irreducible elements. Further, one deduces from \eqref{eqn:2non0} that
every system of mutually disjoint irreducible elements of $\Ss \setminus \Tt_{(0)}$ is finite
(because the quotient algebra $C^*_0(\Tt) / \JjJ_{\Ss}$ is finite-dimensional). Let $\sS_1,\ldots,
\sS_n$ be a maximal such system. We claim that
\begin{equation}\label{eqn:aux17}
\JjJ_{\Ss} = \{f \in C^*_0(\Tt)\dd\ (f(\sS_1),\ldots,f(\sS_n)) = (0,\ldots,0)\}.
\end{equation}
The inclusion ``$\subset$'' is clear. To show the reverse, it suffices to observe that for any
element $\tT$ of $\Ss$ there are irreducible elements $\tT_1,\ldots,\tT_k \in \Tt_{(0)} \cup
\{\sS_1,\ldots,\sS_n\}$ such that $\tT \equiv \bigoplus_{j=1}^k \tT_j$.\par
Further, since $\pi$ is surjective (being irreducible and nonzero), a combination
of \eqref{eqn:ker}, \eqref{eqn:aux17} and \eqref{eqn:2non0} yields that $\mmM_N \cong C^*_0(\Tt) /
\JjJ_{\Ss} \cong \prod_{j=1}^n \mmM_{d(\sS_j)}$, which is possible only when $n = 1$ and $d(\sS_1) =
N$. So,
\begin{equation}\label{eqn:aux18}
\ker(\pi) = \{f \in C^*_0(\Tt)\dd\ f(\sS_1) = 0\}.
\end{equation}
The above connection enables us to define correctly a $*$-homomorphism $\kappa\dd \mmM_N \to \mmM_N$
by the rule $\kappa(A) \df f(\sS_1)$ provided $f \in C^*_0(\Tt)$ is such that $\pi(f) = A$. We infer
from \eqref{eqn:aux18} that $\kappa$ is one-to-one. So, $\kappa\dd \mmM_N \to \mmM_N$ is
a $*$-isomorphism and hence there is $U \in \uuU_N$ for which $\kappa(A) = U . A$ for any $A \in
\mmM_N$ (in the algebra of matrices this is quite an elementary fact; however, this follows also
from Corollary~2.9.32 in \cite{sak}). But then $U . \pi(f) = \kappa(\pi(f)) = f(\sS_1)$ and
consequently $\pi(f) = f(U^{-1} . \sS_1)$. So, we have shown that each nonzero irreducible
finite-dimensional representation $\pi$ of $C^*_0(\Tt)$ has the form $\pi(f) = f(\sS)$ for some
$\sS \in \core(\Tt) \setminus \Tt_{(0)}$.\par
Finally, when $\pi$ is an arbitrary nondegenerate finite-dimensional representation of $C^*_0(\Tt)$,
there are a finite number of nonzero irreducible finite-dimensional representations $\pi_1,\ldots,
\pi_k$ and a unitary matrix $U$ (of a respective degree) such that $\pi(f) = U . (\bigoplus_{j=1}^k
\pi_j(f))$. It follows from the previous part of the proof that for each $j \in \{1,\ldots,k\}$
there is $\sS_j \in \core(\Tt) \setminus \Tt_{(0)}$ such that $\pi_j(f) = f(\sS_j)$. Then
we conclude that $\pi(f) = f(\sS)$ where $\sS \df U . (\bigoplus_{j=1}^k \sS_j)$. Noticing that such
$\sS$ is disjoint from any element of $\Tt_{(0)}$, it remains to show that $\sS$ is unique. To this
end, take an arbitrary $\sS' \in \Tt$ that is different from $\sS$. We need to find a map $f \in
C^*_0(\Tt)$ such that $f(\sS') \neq f(\sS)$. We conclude from \eqref{eqn:values} that there is
$v \in C^*_0(\Tt)$ for which $v(\sS) = I_{d(\sS)}$. If $v(\sS') \neq v(\sS)$, we are done. Hence,
we assume $v(\sS') = I_{d(\sS)}$. \COR{dist} implies that there exists $u \in C^*(\Tt)$ with
$u(\sS') \neq u(\sS)$. Then $f \df uv \in C^*_0(\Tt)$ is such that $f(\sS') \neq f(\sS)$.
\end{proof}

It turns out that all $C^*$-algebras of the form $C^*_0(\Tt)$ (where $\Tt$ is an elc m-tower) are
CCR (that is, liminal). Even more:

\begin{thm}{irr}
Every irreducible representation of $C^*_0(\Tt)$ \textup{(}where $\Tt$ is an elc m-tower\textup{)}
is finite-dimensional.
\end{thm}
\begin{proof}
For each $n > 0$, denote by $\aaA_n$ the $C^*$-algebra of all bounded maps of $\Tt_n$ into $\mmM_n$.
Since $\aaA_n$ has sufficiently many $n$-dimensional representations, it is subhomogeneous
(to convince oneself of that, consult Proposition~IV.1.4.6 in \cite{bla}). Further, denote by $\aaA$
the \textit{direct product} of all $\aaA_n$; that is, $\aaA$ is the $C^*$-algebra of all sequences
$(f_n)_{n=1}^{\infty} \in \prod_{n=1}^{\infty} \aaA_n$ for which $\lim_{n\to\infty} \|f_n\| = 0$.
Since any irreducible representation of $\aaA$ is of the form $(A_n)_{n=1}^{\infty} \mapsto
\pi(A_k)$ where $k > 0$ is arbitrarily fixed and $\pi$ is an irreducible representation of $\aaA_k$,
we see that each irreducible representation of $\aaA$ is finite-dimensional. Finally, the assignment
$f \mapsto (f\bigr|_{\Tt_n})_{n=1}^{\infty}$ correctly defines a one-to-one $*$-homomorphism
of $C^*_0(\Tt)$ into $\aaA$ and therefore $C^*_0(\Tt)$ is $*$-isomorphic to a subalgebra of $\aaA$.
So, since irreducible representations of $C^*$-subalgebras always admit extensions to irreducible
representations (possibly in greater Hilbert spaces; consult Proposition~2.10.2 in \cite{di2}),
the assertion follows.
\end{proof}

\begin{cor}{irr}
Any nonzero irreducible representation of $C^*_0(\Tt)$ \textup{(}for an elc m-tower $\Tt$\textup{)}
has the form $f \mapsto f(\sS)$ where $\sS \in \core(\Tt) \setminus \Tt_{(0)}$.
\end{cor}
\begin{proof}
Just apply \THM{irr} and \COR{nondeg}.
\end{proof}

\begin{rem}{shr}
According to the terminology of Section~7, $C^*_0(\Tt)$ (for an elc m-tower $\Tt$) is
a \textit{shrinking} $C^*$-algebra (and each such an algebra has all irreducible representations
finite-dimensional). In the aforementioned part we shall describe models for all such algebras.
\end{rem}

Another strong consequence of \THM{SW} is formulated below.

\begin{thm}{ext}
Let $\Tt$ be an elc m-tower and $\Ss$ its arbitrary subtower. A necessary and sufficient condition
for a map $u \in C^*(\Ss)$ to be extendable to a map $v \in C^*_0(\Tt)$ is that $u \in C^*_0(\Ss)$
and $u$ vanishes at each point of $\Ss \cap \Tt_{(0)}$. What is more, if $u \in C^*_0(\Ss)$ vanishes
at each point of $\Ss \cap \Tt_{(0)}$, then for any $M > \|u\|$ there is an extension $v \in
C^*_0(\Tt)$ of $u$ with $\|v\| < M$.
\end{thm}
\begin{proof}
The necessity is clear (note that $\overline{\core}(\Ss) \subset \overline{\core}(\Tt)$). To prove
the sufficiency, we employ \THM{SW}. Define a $*$-homomorphism $\Phi\dd C^*_0(\Tt) \to C^*_0(\Ss)$
by $\Phi(f) \df f\bigr|_{\Ss}$. Since the images of $*$-homomorphisms between $C^*$-algebras are
always closed, it suffice to check that the closure of $\EeE \df \Phi(C^*_0(\Tt))$ in $C^*_0(\Ss)$
coincides with all maps $u \in C^*_0(\Ss)$ that vanish at each point of $\Ss \cap \Tt_{(0)}$. But
this conclusion simply follows from \THM{SW} (we skip the details). Finally, since each
$*$-homomorphism sends the open unit ball onto the open unit ball (of the range), the additional
claim of the theorem follows.
\end{proof}

\begin{rem}{strong}
To convince oneself of the strength of \THM{ext}, the reader may try to construct an extension
directly. The first difficulty with which one needs to face is to show that the restriction to $\Ss
\cap \overline{\core}(\Tt)$ of each map from $C^*_0(\Ss)$ that vanishes at each point of $\Ss \cap
\Tt_{(0)}$ vanishes at infinity. The author has no idea how to prove it directly.
\end{rem}

The results of this section prove that it is relevant to know the description of the vanishing
tower. For proper m-towers, this shall be done in Section~6 (see \PRO{T0} and \COR{T0}).\par
We conclude the section with the variation of the Stone-Weierstrass theorem announced earlier.

\begin{thm}[Stone-Weierstrass Theorem for ELC M-Towers]{SW!}
Let $\EeE$ be a $*$-subalgebra of $C^*_0(\Tt)$ for some elc m-tower $\Tt$. The uniform closure
of $\EeE$ in $C^*_0(\Tt)$ consists of all functions $u \in C^*_0(\Tt)$ such that
\begin{equation}\label{eqn:S-W}
(u(\xX),u(\yY)) \in \{(v(\xX),v(\yY))\dd\ v \in \EeE\}
\end{equation}
for all $\xX, \yY \in \core(\Tt) \setminus \Tt_{(0)}$.
\end{thm}
\begin{proof}
Denote by $\FfF$ the family of all functions $u \in C^*_0(\Tt)$ that satisfy \eqref{eqn:S-W} for
all $\xX, \yY \in \core(\Tt) \setminus \Tt_{(0)}$. It is obvious that $\FfF$ is a $C^*$-algebra that
contains the closure $\bar{\EeE}$ of $\EeE$. To show that $\FfF = \bar{\EeE}$, we shall check that
$\bar{\EeE}$ is a \textit{rich} subalgebra of $\FfF$ in the sense of Dixmier (see Definition~11.1.1
in \cite{di2}). That is, we have to verify the following two conditions:
\begin{enumerate}[(r1)]
\item for every irreducible representation $\pi$ of $\FfF$, $\pi\bigr|_{\bar{\EeE}}$ is irreducible
 as well;
\item if $\pi$ and $\pi'$ are two unitarily inequivalent irreducible representations of $\FfF$, then
 $\pi\bigr|_{\bar{\EeE}}$ and $\pi'\bigr|_{\bar{\EeE}}$ are also inequivalent.
\end{enumerate}
To this end, let $\pi$ be a nonzero irreducible representation of $\FfF$. Since $\pi$ extends
to an irreducible representation of $C^*_{\Tt}$, we infer from \COR{irr} that there are $\xX \in
\core(\Tt) \setminus \Tt_{(0)}$ and a positive integer $k \leqsl d(\xX)$ such that for any $u \in
\FfF$, the top left $k \times k$ submatrix of $u(\xX)$ coincides with $\pi(u)$. Then we infer from
\eqref{eqn:S-W} that $\{u(\xX)\dd\ u \in \FfF\} = \{v(\xX)\dd\ v \in \EeE\}$ and therefore
$\pi\bigr|_{\bar{\EeE}}$ is a nonzero irreducible representation. So, (r1) holds, and it suffices
to check (r2) only for nonzero $\pi$ and $\pi'$. But, if $\pi \not\equiv 0$ and $\pi' \not\equiv 0$
are as specified in (r2), then there is $u_0 \in \FfF$ for which $\pi(u_0) \neq 0$ and $\pi'(u_0) =
0$. Moreover, the previous argument shows that there are two elements $\xX,\yY \in \core(\Tt)
\setminus \Tt_{(0)}$ and two positive integers $k \leqsl d(\xX)$ and $\ell \leqsl d(\yY)$ such that
for any $u \in \FfF$, $\pi(u)$ and $\pi'(u)$ coincide with, respectively, the top left $k \times k$
submatrix of $u(\xX)$ and the top left $\ell \times \ell$ submatrix of $u(\yY)$. Now taking into
account \eqref{eqn:S-W} for these $\xX$ and $\yY$, we see that there is $v_0 \in \EeE$ with
$\pi(v_0) = \pi(u_0)$ and $\pi'(v_0) = \pi'(u_0)$. These two connections imply that the restrictions
of $\pi$ and $\pi'$ to $\bar{\EeE}$ are unitarily inequivalent, which proves (r2).\par
Finally, $C^*_0(\Tt)$ is liminal, by \THM{irr}, and thus also $\FfF$ is liminal,
as a $C^*$-subalgebra of $C^*_0(\Tt)$. So, $\bar{\EeE}$ is a rich subalgebra of the liminal
$C^*$-algebra $\FfF$, and hence $\bar{\EeE} = \FfF$, thanks to Proposition~11.1.6 in \cite{di2}.
\end{proof}

\section{Proper m-towers of finite height}

In this section we prove \THM[s]{subh1} and \THM[]{subh0} and describe $*$-homomorphisms between
subhomogeneous $C^*$-algebras by means of morphisms between corresponding to them m-towers.
The reader interested in subhomogeneous $C^*$-algebras is referred to Subsection~IV.1.4
of \cite{bla}.\par
The first result of this section summarizes (in a special case) some of results of the previous
section.

\begin{pro}{hgt}
Let $\Tt$ be a proper tower with $n \df \hgt(\Tt) < \infty$ and let $\aaA = C^*(\Tt)$.
\begin{enumerate}[\upshape(A)]
\item $\aaA$ is a unital $n$-subhomogeneous $C^*$-algebra \textup{(}here ``\,$0$-subhomogeneous''
 means ``$\aaA = \{0\}$''\textup{)}.
\item The assigment $\Ss \mapsto \JjJ_{\Ss}$ establishes a one-to-one correspondence between all
 subtowers of $\Tt$ and all ideals in $\aaA$.
\item Every unital finite-dimensional representation of $\aaA$ has the form $\pi_{\sS}\dd \aaA \ni f
 \mapsto f(\sS) \in \mmM_{d(\sS)}$ where $\sS \in \Tt$. The assignment $\tT \mapsto \pi_{\tT}$ is
 one-to-one and $\pi_{\tT}$ is irreducible iff $\tT \in \core(\Tt)$.
\end{enumerate}
\end{pro}
\begin{proof}
As noted in the previous section, $\Tt$ is elc and $C^*(\Tt) = C^*_0(\Tt)$. Moreover, $\Tt_{(0)}$ is
empty (according to \COR{dist}). Consequently, (B) follows from \COR{ideal}, (C) from
\COR[s]{nondeg} and \COR[]{irr}, whereas (A) is an immediate consequence of (C) (and the definition
of $n$).
\end{proof}

We recall that for any (unital) $C^*$-algebra $\aaA$, the concrete tower $\Zz(\aaA)$ (resp.\
$\Xx(\aaA)$) was defined in Introduction and it consists of all (resp.\ all unital)
finite-dimensional representations of $\aaA$.

\begin{lem}{concr}
For any \textup{(}unital\textup{)} $C^*$-algebra $\aaA$, $\Zz(\aaA)$ \textup{(}resp.\
$\Xx(\aaA)$\textup{)} is a proper m-tower.
\end{lem}
\begin{proof}
Denote by $\Lambda$ the closed unit ball of $\aaA$. Put $\Tt \df \Zz(\aaA)$ (resp.\ $\Tt \df
\Xx(\aaA)$). To explain that $\Tt$ is a proper m-tower, it suffices to show that $\Tt$ is isomorphic
to a standard m-tower. Consider a function $\Psi\dd \Tt \ni \pi \mapsto \pi\bigr|_{\Lambda} \in
\mmM[\Lambda]$ and observe that $\Psi$ is one-to-one and that $d(\Psi(\pi)) = d(\pi)$, $\Psi(U .
\pi) = U . \Psi(\pi)$ and $\Psi(\pi \oplus \pi') = \Psi(\pi) \oplus \Psi(\pi')$ for any $\pi, \pi'
\in \Tt$ and $U \in \uuU_{d(\pi)}$. What is more, since the space of all (resp.\ all unital)
$n$-dimensional representations of $\aaA$ is compact in the pointwise convergence topology, we see
that condition (ST0) for $\ttT \df \Psi(\Tt)$ is fulfilled. Further, (ST1) follows from
the definition of $\Lambda$ (and the fact that each representation has norm not greater than $1$),
whereas (ST2) and (ST3) are straightforward. So, $\ttT$ is a standard tower, and $\Psi$ is
a homeomorphism (which follows from the very definitions of the topologies on $\Tt$ and $\ttT$).
Thus, $\Tt$ is a proper m-tower.
\end{proof}

\begin{proof}[Proof of \THM{subh1}]
It follows from the previous lemma that $\Tt \df \Xx(\aaA)$ is a proper m-tower. Moreover,
a finite-dimensional representation $\pi$ belongs to the core of $\Tt$ \iaoi{} $\pi$ is
an irreducible representation, which implies that $\hgt(\Tt) < \infty$.\par
To show the main part of the theorem, namely, that $J_{\aaA}$ is surjective, it suffices to employ
\COR{dense} (recall that $\Tt$ is elc and $C^*(\Tt) = C^*_0(\Tt)$) for $\EeE \df J_{\aaA}(\Tt)$.
Conditions (d1)--(d2) transform into:
\begin{enumerate}[(d1')]
\item each unital irreducible representation $\pi$ of $\aaA$ is surjective;
\item if $\pi_1$ and $\pi_2$ are two inequivalent unital irreducible representations of $\aaA$, then
 $\pi_1(a) = I$ and $\pi_2(a) = 0$ for some $a \in \aaA$ (where $I$ denotes the unit matrix
 of a respective degree).
\end{enumerate}
But both the conditions (d1') and (d2') are already known (they are special cases
of Proposition~4.2.5 in \cite{di2}) and thus the proof of this part is finished.\par
Now assume $\Tt$ is an arbitrary proper m-tower of finite height. According to \PRO{hgt}, we only
need to prove that the assignment $\tT \mapsto \pi_{\tT}$ defines a homeomorphism between $\Tt$ and
$\Xx(C^*(\Tt))$. But this simply follows from the fact that each of $\Tt_n$ is compact (and this
assignment is bijective).
\end{proof}

\begin{proof}[Proof of \THM{subh0}]
This result is actually an immediate consequence of \THM{subh1} and may be shown as follows. Let
$\aaA$ be an arbitrary subhomogeneous $C^*$-algebra (with or without unit). Let $\aaA_1$ denote
the unitization of $\aaA$ (which may be constructed even for unital $\aaA$). Notice that $\aaA_1$ is
a subhomogeneous $C^*$-algebra in which $\aaA$ is an ideal such that the quotient space $\aaA_1 /
\aaA$ is one-dimensional. Moreover, there is a natural one-to-one correspondence between all
finite-dimensional representations of $\aaA$ and all unital finite-dimensional representations
of $\aaA_1$. This means that $\Zz(\aaA)$ may naturally be identified with $\Xx(\aaA_1)$. Under such
an identification, the restriction of the $*$-isomorphism $\JjJ_{\aaA_1}\dd \aaA_1 \to
C^*(\Zz(\aaA))$ to $\aaA$ coincides with $\JjJ_{\aaA}$ and sends $\aaA$ onto a subspace $\EeE$
of $C^*(\Zz(\aaA),\theta_{\aaA})$ such that the quotient space $C^*(\Zz(\aaA)) / \EeE$ is
one-dimensional. This implies that $\EeE = C^*(\Zz(\aaA),\theta_{\aaA})$ and we are done.
(The details are left to the reader.)\par
The second part of the theorem may be proved in the same manner (using the fact that $C^*(\Zz)$ is
the unitization of $C^*(\Zz,\theta)$) and is left to the reader.
\end{proof}

Our next aim is to characterize $*$-homomorphisms between subhomogeneous $C^*$-algebras. As before,
we start from unital algebras.

\begin{thm}{homo1}
Let $\Tt$ and $\Ss$ be two proper m-towers of finite height. For every unital $*$-homomorphism
$\Phi\dd C^*(\Tt) \to C^*(\Ss)$ there exists a unique morphism $\tau\dd \Ss \to \Tt$ such that
$\Phi = \Phi_{\tau}$ where
\begin{equation}\label{eqn:Phi-tau}
\Phi_{\tau}\dd C^*(\Tt) \ni f \mapsto f \circ \tau \in C^*(\Ss).
\end{equation}
Conversely, if $\tau\dd \Ss \to \Tt$ is an arbitrary morphism, then \eqref{eqn:Phi-tau} correctly
defines a unital $*$-homomorphism $\Phi_{\tau}$.
\end{thm}
\begin{proof}
For any $\sS \in \Ss$, the function $C^*(\Tt) \ni f \mapsto (\Phi(f))(\sS) \in \mmM_{d(\sS)}$ is
a unital finite-dimensional $*$-representation of $\aaA$. So, it follows from \COR{nondeg} that
there is a unique point $(\tau(\sS) \df\,)\ \tT \in \Tt$ such that $(\Phi(f))(\sS) = f(\tT)$ for any
$f \in C^*(\Tt)$. In this way one obtains a function $\tau\dd \Ss \to \Tt$ such that $\Phi(f) = f
\circ \tau$. Moreover, it is immediate that $\tau$ is unique and satisfies all axioms of a morphism,
apart from continuity, which in turn follows from the fact that $\Tt$ is naturally isomorphic
(as described in \THM{subh1}) to $\Xx(C^*(\Tt))$. The second claim is much simpler and is left
as an exercise.
\end{proof}

Under the notation introduced in \eqref{eqn:Phi-tau}, it is obviously seen that $\Phi_{\tau} \circ
\Phi_{\tau'} = \Phi_{\tau' \circ \tau}$ for any two morphisms $\tau\dd \Tt \to \Tt'$ and $\tau'\dd
\Tt' \to \Tt''$ (and $\Phi_{\id} = \id$ where $\id$ stands for the identity map on an appropriate
space). In particular, $*$-isomorphisms between unital subhomogeneous $C^*$-algebras correspond
to isomorphisms between m-towers.

\begin{cor}{isom}
Two unital subhomogeneous $C^*$-algebras are $*$-isomorphic iff their m-towers of unital
finite-dimensional representations are isomorphic.
\end{cor}

\begin{cor}{homo0}
Let $(\Tt,\theta)$ and $(\Ss,\kappa)$ be two proper m-towers of finite height. For every
$*$-homomorphism $\Phi\dd C^*(\Tt,\theta) \to C^*(\Ss,\kappa)$ there exists a unique morphism
$\tau\dd (\Ss,\kappa) \to (\Tt,\theta)$ such that $\Phi = \Phi_{\tau}$ where
\begin{equation}\label{eqn:Phi-tau0}
\Phi_{\tau}\dd C^*(\Tt,\theta) \ni f \mapsto f \circ \tau \in C^*(\Ss,\kappa).
\end{equation}
Conversely, if $\tau\dd (\Ss,\kappa) \to (\Tt,\theta)$ is an arbitrary morphism, then
\eqref{eqn:Phi-tau0} correctly defines a $*$-homomorphism $\Phi_{\tau}$.
\end{cor}
\begin{proof}
Use the fact that each $*$-homomorphism between $C^*(\Tt,\theta)$ and $C^*(\Ss,\kappa)$ extends
uniquely to a unital $*$-homomorphism between $C^*(\Tt)$ and $C^*(\Ss)$ and then apply \THM{homo1}.
\end{proof}

\begin{pro}{sur}
Let $\tau\dd \Ss \to \Tt$ be a morphism between two proper m-towers of finite height.
The $*$-homomorphism $\Phi_{\tau}$ \textup{(}given by \eqref{eqn:Phi-tau}\textup{)} is surjective
iff $\tau$ is one-to-one. If this happens, $\tau(\Ss)$ is a subtower of $\Tt$ and $\tau$ is
an isomorphism from $\Ss$ onto $\tau(\Ss)$.
\end{pro}
\begin{proof}
If $\Phi_{\tau}$ is surjective, we infer from \COR{dist} that $\tau$ is one-to-one (because maps
from $C^*(\Ss)$ separate points of $\Ss$). Conversely, if $\tau$ is one-to-one, then
$\stab(\tau(\sS)) = \stab(\sS)$ for any $\sS \in \Ss$. Consequently,
\begin{equation}\label{eqn:cor-cor}
\tau(\core(\Ss)) \subset \core(\Tt)
\end{equation}
(by \LEM{core}) and $\tau$ is a closed embedding (because $\Ss$ is proper). So, to convince oneself
that $\Ss' \df \tau(\Ss)$ is a subtower of $\Tt$, it suffices to show that if $\tau(\sS) = \aA
\oplus \bB$ for some $\sS \in \Ss$ and $\aA, \bB \in \Tt$, then both $\aA$ and $\bB$ belong
to the image of $\tau$. This can be shown as follows (cf.\ the proof of \PRO{char}). We infer from
\PRO{pd} that there are $\sS_1,\ldots,\sS_n \in \core(\Ss)$ such that $\sS \equiv \bigoplus_{j=1}^n
\sS_j$. Then also $\aA \oplus \bB = \tau(\sS) \equiv \bigoplus_{j=1}^n \tau(\sS_j)$ and $\tau(\sS_j)
\in \core(\Tt)$ (by \eqref{eqn:cor-cor}). So, it follows from \COR{pd} that there are nonempty
subsets $J_1$ and $J_2$ of $\{1,\ldots,n\}$ such that $\aA \equiv \bigoplus_{j \in J_1} \tau(\sS_j)$
and $\bB \equiv \bigoplus_{j \in J_2} \tau(\sS_j)$, from which one deduces that $\aA, \bB \in
\tau(\Ss)$.\par
Further, since $\tau$ is an embedding, we see that $\tau$ is an isomorphism of $\Ss$ onto $\Ss'$.
So, if $u \in C^*(\Ss)$, then $u \circ \tau^{-1} \in C^*(\Ss')$. Now \THM{extend} yields that there
is $v \in C^*(\Tt)$ which extends $u \circ \tau^{-1}$. Then $\Phi_{\tau}(v) = u$ and hence
$\Phi_{\tau}$ is surjective.
\end{proof}

\begin{rem}{inj}
Under the notation of \PRO{sur}, one may also (simply) check that $\Phi_{\tau}$ is one-to-one iff
$\Tt$ coincides with the smallest subtower of $\Tt$ that contains $\tau(\Ss)$. However,
if $\Phi_{\tau}$ is one-to-one but not surjective, then $\tau(\Ss)$ may \textit{not} be a subtower
of $\Tt$, which causes problems in studying ASH algebras by means of m-towers.
\end{rem}

\begin{rem}{homo}
If $\aaA$ is a unital $n$-homogeneous $C^*$-algebra, then $\Hh \df \core(\Xx(\aaA))$ is compact and
coincides with $\Xx_n(\aaA)$. It is then easy to show that the assignment $f \mapsto f\bigr|_{\Hh}$
defines a $*$-isomorphism from $C^*(\Xx(\aaA))$ onto $C^*(\Hh,.) \df \{f \in C(\Hh,\mmM_n)\dd\
f(U . \hH) = U . f(\hH)\ (U \in \uuU_n,\ \hH \in \Hh)\}$ (to establish surjectivity, apply
\LEM{semi}). Similarly, if $\aaA$ is nonunital $n$-homogeneous $C^*$-algebra, then $\Hh \df
\core(\Zz(\aaA))$ is locally compact and $\Zz_n(\aaA) = \Hh \cup \{n \odot \theta_{\aaA}\}$
($\theta_{\aaA}$ denotes the zero one-dimensional representation of $\aaA$). In that case
the assignment $f \mapsto f\bigr|_{\Hh}$ defines a $*$-isomorphism from
$C^*(\Zz(\aaA),\theta_{\aaA})$ onto $C^*(\Hh,.)$ where $C^*(\Hh,.)$ is the set of all maps $f\dd \Hh
\to \mmM_n$ that vanish at infinity and satisfy $f(U . \hH) = U . \hH$ for all $U \in \uuU_n$ and
$\hH \in \Hh$. In this way one obtains less complicated (and more transparent) models for all
homogeneous $C^*$-algebras, which was first shown in \cite{fe3} and \cite{t-t} (see also
\cite{pn2}). The main disadvantage of these models is that there is no `bridge' between
$n$-homogeneous and $m$-homogeneous $C^*$-algebras for different $n$ and $m$ (and hence these models
are hardly applicable in investigations of approximately homogeneous $C^*$-algebras).
\end{rem}

And now the result announced in \REM{superfl}.

\begin{thm}{inv}
For a unital $C^*$-algebra $\aaA$ \tfcae
\begin{enumerate}[\upshape(i)]
\item there exists an m-tower $\Tt$ such that $C^*(\Tt)$ is $*$-isomorphic to $\aaA$;
\item $\aaA$ is $*$-isomorphic to the inverse limit of subhomogeneous $C^*$-algebras.
\end{enumerate}
\end{thm}

In the proof we shall apply the next two lemmas.

\begin{lem}{char}
For a $C^*$-algebra $\aaA$ \tfcae
\begin{enumerate}[\upshape(i)]
\item $\aaA$ is $*$-isomorphic to the inverse limit of subhomogeneous $C^*$-algebras;
\item there exists a descending sequence $\JjJ_1,\JjJ_2,\ldots$ of ideals in $\aaA$ such that:
 \begin{enumerate}[\upshape({i}d1)]
 \item $\bigcap_{n=1}^{\infty} \JjJ_n = \{0\}$; and
 \item the quotient $C^*$-algebra $\qqQ_n \df \aaA / \JjJ_n$ is subhomogeneous for any $n > 0$; and
 \item whenever $(z_n)_{n=1}^{\infty}$ is a bounded sequence of elements of $\aaA$ such that
  $\pi_k(z_n)$ converges in $\qqQ_k$ as $n \to \infty$ for any $k > 0$ where $\pi_k\dd \aaA \to
  \qqQ_k$ is the canonical projection, then there is $z \in \aaA$ such that $\lim_{n\to\infty}
  \pi_k(z_n) = \pi(z)$ for all $k$.
 \end{enumerate}
\end{enumerate}
\end{lem}
\begin{proof}
To show that (ii) follows from (i), we may and do assume that $\aaA$ is the inverse limit
of subhomogeneous $C^*$-algebras. So, let
$\{(\aaA_{\sigma};\{\varphi_{\sigma,\sigma'}\})\}_{\sigma\in\Sigma}$ be an inverse system
of subhomogeneous $C^*$-algebras. That is:
\begin{itemize}
\item $(\Sigma,\leqsl)$ is a directed set;
\item $\aaA_{\sigma}$ is subhomogeneous for any $\sigma \in \Sigma$;
\item whenever $\sigma \leqsl \sigma'$, there is defined a $*$-homomorphism
 $\varphi_{\sigma,\sigma'}\dd \aaA_{\sigma'} \to \aaA_{\sigma}$;
\item $\varphi_{\sigma,\sigma'} \circ \varphi_{\sigma',\sigma''} = \varphi_{\sigma,\sigma''}$
 whenever $\sigma \leqsl \sigma' \leqsl \sigma''$.
\end{itemize}
Then the inverse limit $(\aaA;\{\pi_{\sigma}\})$ of the above system consists of
\begin{itemize}
\item a $C^*$-algebra $\aaA$ that coincides with the $C^*$-algebra of all bounded systems
 $(a_{\sigma})_{\sigma\in\Sigma} \in \prod_{\sigma\in\Sigma} \aaA_{\sigma}$ such that
 $\varphi_{\sigma,\sigma'}(a_{\sigma'}) = a_{\sigma}$ whenever $\sigma \leqsl \sigma'$; and
\item $*$-homomorphisms $\pi_{\sigma}\dd \aaA \to \aaA_{\sigma}$ (for each $\sigma$), given by
 $\pi_{\sigma}((a_{\tau})_{\tau\in\Sigma}) = a_{\sigma}$.
\end{itemize}
For each $\sigma \in \Sigma$, let $\nu(\sigma)$ be a positive integer such that $\aaA_{\sigma}$ is
$\nu(\sigma)$-subhomogeneous. Put $\Sigma_n \df \{\sigma \in \Sigma\dd\ \nu(\sigma) \leqsl n\}$ and
define $\JjJ_n$ as the set of all $(a_{\sigma})_{\sigma\in\Sigma} \in \aaA$ such that $a_{\sigma} =
0$ for any $\sigma \in \Sigma_n$. It is transparent that $\JjJ_n$ is an ideal and that (id1) holds.
To show (id2), observe that $\aaA / \JjJ_n$ is $*$-isomorphic to the $C^*$-subalgebra
$\{(a_{\sigma})_{\sigma\in\Sigma_n}\dd\ (a_{\sigma})_{\sigma\in\Sigma} \in \aaA\}$
of $\prod_{\sigma\in\Sigma_n} \aaA_{\sigma}$ and the latter algebra is $n$-subhomogeneous. Finally,
we turn to (id3). Let elements $z_n = (a^{(n)}_{\sigma})_{\sigma\in\Sigma}$ of $\aaA$ form
a bounded sequence with the property specified in (id3). It then follows that for any $\sigma$,
the sequence $(a^{(n)}_{\sigma})_{n=1}^{\infty}$ converges in $\aaA_{\sigma}$, say to $a_{\sigma}$.
Then $z \df (a_{\sigma})_{\sigma\in\Sigma}$ is a bounded system which belongs to $\aaA$ and
satisfies the assertion of (id3).\par
Now let $\JjJ_n$, $\qqQ_n$ and $\pi_n$ (for any $n > 0$) be as specified in (ii). For each $n$,
define $\varphi_n\dd \qqQ_{n+1} \to \qqQ_n$ by $\varphi_n(a + \JjJ_{n+1}) \df a + \JjJ_n\ (a \in
\aaA)$. Then $\varphi_n$ is a well defined surjective $*$-homomorphism. Let $\tilde{\aaA}$ stand for
the inverse limit of the inverse sequence $\{(\qqQ_n,\varphi_n)\}_{n=1}^{\infty}$. We define
$\Phi\dd \aaA \to \tilde{\aaA}$ by $\Phi(a) \df (\pi_n(a))_{n=1}^{\infty}$. It is straightforward
to show that $\Phi$ is a well defined $*$-homomorphism. Moreover, it follows from (id1) that $\Phi$
is one-to-one. Thus, it remains to check that $\Phi$ is onto. To this end, let
$(w_n)_{n=1}^{\infty}$ be an arbitrary element of $\tilde{\aaA}$. This means that $w_n \in \qqQ_n$,
\begin{equation}\label{eqn:aux19}
\varphi_n(w_{n+1}) = w_n \qquad (n > 0)
\end{equation}
and $\sup_{n\geqsl1} \|w_n\| < \infty$. For each $n$, take $z_n \in \aaA$ such that $\pi_n(z_n) =
w_n$ and $\|z_n\| < \|w_n\| + 1$. Then the sequence $(z_n)_{n=1}^{\infty}$ satisfies all assumptions
of (id3) and hence there is $z \in \aaA$ such that $\pi_k(z) = \lim_{n\to\infty} \pi_k(z_n) = w_k$
for all $k$ ($\pi_k(z_n) = w_k$ for $k \geqsl n$, thanks to \eqref{eqn:aux19}). Consequently,
$\Phi(z) = (w_n)_{n=1}^{\infty}$ and we are done.
\end{proof}

\begin{lem}{sigma2}
If topological spaces $X$ and $Y$ have property $(\sigma)$ with respect to, respectively,
$(K_n)_{n=1}^{\infty}$ and $(L_n)_{n=1}^{\infty}$, then the space $X \times Y$, equipped with
the product topology, has property $(\sigma)$ with respect to $(K_n \times L_n)_{n=1}^{\infty}$; and
each closed set $A$ in $X$ has property $(\sigma)$ with respect to $(K_n \cap A)_{n=1}^{\infty}$.
\end{lem}
\begin{proof}
We start from the Cartesian product. The closure of a set $B$ (in the whole respective space) will
be denoted by $\bar{B}$. It is readily seen that conditions ($\sigma$1)--($\sigma$2) are fulfilled
and that $F \cap (K_n \times L_n)$ is closed for each $n$ provided $F$ is closed in $X \times Y$.
So, we only need to prove the `if' part in ($\sigma$3), which is equivalent to the following
statement: if $U \subset X \times Y$ is a set such that $U \cap (K_n \times L_n)$ is relatively open
in $K_n \times L_n$ for all $n$, then $U$ is open in $X \times Y$. Let $U$ be a set that satisfies
the assumptions of this statement. Fix a pair $(x,y) \in U$. It is enough to construct open sets $V$
and $W$ in $X$ and $Y$ (respectively) such that $(x,y) \in V \times W \subset U$. Take $N > 0$ such
that $(x,y) \in K_N \times L_N$. Below we shall use the following well-known (and simple)
topological property:
\begin{itemize}
\item[($\star$)] if $C$ and $D$ are compact Hausdorff spaces, $A \subset C$ and $B \subset D$
 are compact, and $E$ is an open subset of $C \times D$ (in the product topology) that includes
 $A \times B$, then there are open sets $G \subset C$ and $H \subset D$ for which $A \times B
 \subset G \times H$ and $\bar{G} \times \bar{H} \subset E$.
\end{itemize}
Using ($\star$) (for $C = K_N$, $D = L_N$, $A = \{x\}$, $B = \{y\}$ and $E = (K_N \times L_N) \cap
U$), we find two sets $V_0 (\subset K_N)$ and $W_0 (\subset L_N)$ that are relatively open
neighbourhoods of $x$ and $y$ in $K_N$ and $L_N$ (respectively) and satisfy $\bar{V}_N \times
\bar{W}_N \subset U$. Now assume we have constructed $V_{n-1} \subset K_{N+n-1}$ and $W_{n-1}
\subset L_{N+n-1}$ for some $n > 0$. Using again ($\star$) (for $C = K_{N+n}$, $D = L_{N+n}$, $A =
\bar{V}_{n-1}$, $B = \bar{W}_{n-1}$ and $E = (K_{N+n} \times L_{N+n}) \cap U$), we find two sets
$V_n$ and $W_n$ that are relatively open in $K_{N+n}$ and $L_{N+n}$ (respectively) and satisfy
$\bar{V}_{n-1} \times \bar{W}_{n-1} \subset V_n \times W_n$ and $\bar{V}_n \times \bar{W}_n \subset
U$. In this way we obtain two ascending sequences $V_0,V_1,\ldots$ and $W_0,W_1,\ldots$ of subsets
of $X$ and $Y$. We claim that the sets $V \df \bigcup_{n=0}^{\infty} V_n$ and $W \df
\bigcup_{n=0}^{\infty} W_n$ are the sets we searched for. It is clear that $(x,y) \in V \times W
\subset U$. So, we only need to verify that $V$ and $W$ are open in $X$ and $Y$. To this end,
it suffices to check that $V \cap K_n$ if relatively open in $K_n$ for any $n > N$ (and similarly
for $W$). But $V \cap K_n = \bigcup_{j=n-N}^{\infty} (V_j \cap K_n)$ and since $V_j$ is relatively
open in $K_{N+j}$ and $K_{N+j} \supset K_n$ for all $j \geqsl n-N$, the conclusion follows.\par
The claim about closed subsets is trivial.
\end{proof}

\begin{proof}[Proof of \THM{inv}]
First take an arbitrary tower $\Tt$. We shall show that $\aaA \df C^*(\Tt)$ satisfies condition (ii)
of \LEM{char}. To this end, define $\JjJ_n$ as the set of all $f \in \aaA$ that vanish at each point
of $\Dd_n \df \bigcup_{k=1}^n \Tt_k$. Since $\aaA / \JjJ_n$ is $*$-isomorphic
to $\{f\bigr|_{\Dd_n}\dd\ f \in \aaA\}$, we see that this quotient algebra is $n$-subhomogeneous.
Further, (id1) is trivial, whereas (id3) transforms into: if $(f_n)_{n=1}^{\infty}$ is a uniformly
bounded sequence of maps from $\aaA$ that converge uniformly on each of $\Dd_n$, then there is
a single map from $\aaA$ such that the maps $f_n$ converge uniformly to $f$ on each of $\Dd_n$,
which is readily seen. This shows that $C^*(\Tt)$ is $*$-isomorphic to the inverse limit
of subhomogeneous $C^*$-algebras for any tower $\Tt$. We shall now derive (ii) from (i).\par
Let $\aaA$ be a unital $C^*$-algebra which is $*$-isomorphic to the inverse limit of subhomogeneous
algebras. The proof of \LEM{char} shows that then $\aaA$ is $*$-isomorphic to the inverse limit
of an inverse sequence $\{(\aaA_n,\varphi_n)\}_{n=1}^{\infty}$ where each $\aaA_n$ is unital (since
$\aaA$ is so) and subhomogeneous, and each $\varphi_n\dd \aaA_{n+1} \to \aaA_n$ is a surjective
unital $*$-homomorphism. Represent $\aaA_n$ as $C^*(\Tt^{(n)})$ where $\Tt^{(n)}$ is a proper
m-tower of finite height. Then, according to \THM{homo1}, $\varphi_n$ has the form $\Phi_{\tau_n}$
where $\tau_n\dd \Tt^{(n+1)} \to \Tt^{(n)}$ is a morphism and $\Phi_{\tau_n}$ is as specified
in \eqref{eqn:Phi-tau}. We now infer from \PRO{sur} that $\Ss^{(n)} \df \tau_n(\Tt^{(n)})$ is
a subtower of $\Tt^{(n+1)}$ and $\tau_n$ is an isomorphism between $\Tt^{(n)}$ and $\Ss^{(n)}$.
Thus, with no loss of generality, we may and do assume that $\Tt^{(n)} \subset \Tt^{(n+1)}$ and
$\tau_n$ is the inclusion map for each $n$. In these settings, $\varphi_n$ operates on maps from
$C^*(\Tt^{(n+1)})$ by restricting them to $\Tt^{(n)}$. Put $\Tt \df \bigcup_{n=1}^{\infty}
\Tt^{(n)}$. $\Tt$ admits unique ingredients (that is, a degree $d$, an addition ``$\oplus$'' and
a unitary action ``$.$'') which extend the ingredients of all $\Tt^{(n)}$. It is straightforward
to check that $(\Tt,d,\oplus,.)$ has all properties, apart from topological, specified in axioms
(T1)--(T6) and (mT1)--(mT4). We shall now define a topology on $\Tt$ which will make $\Tt$
an m-tower: a set $\Aa \subset \Tt$ is closed iff $\Aa \cap \Tt^{(n)}$ is closed in $\Tt^{(n)}$ for
any $n$. It is easily seen that each of $\Tt^{(n)}$ is a closed set in $\Tt$ and
\begin{itemize}
\item[($\Delta$)] the original topology of $\Tt^{(n)}$ coincides with the topology of a subspace
 inherited from $\Tt$.
\end{itemize}
For each $n > 0$, we put
\begin{equation*}
\Kk_n \df \bigcup_{j=1}^n \bigcup_{k=1}^n \Tt^{(k)}_j\ (\subset \Tt^{(n)}).
\end{equation*}
It follows from ($\Delta$) that $\Kk_n$, equipped with the topology inherited from $\Tt$, is
a compact Hausdorff space. Further, since $\Tt^{(m)}$ has property $(\sigma)$ with respect
to $(\Kk_n \cap \Tt^{(m)})_{n=1}^{\infty}$, one deduces that $\Tt$ has property $(\sigma)$ with
respect to $(\Kk_n)_{n=1}^{\infty}$ (and, consequently, it is a Hausdorff space, by \LEM{para}). So,
to show that $\Tt$ is an m-tower, it suffices to check that the degree and the unitary action are
continuous and the addition is a closed embedding on each of the sets $\Tt_n \times \Tt$. To this
end, we employ \LEM[s]{para} and \LEM[]{sigma2}. The continuity of $d$ is very simple: since
$d\bigr|_{\Tt^{(n)}}$ is continuous, so is $d\bigr|_{\Kk_n}$ and we are done. Further, when $N$ is
fixed, $\Tt_N$ is a closed subspace of $\Tt$ and hence it has property $(\sigma)$ with respect
to $(\Kk_n \cap \Tt_N)_{n=1}^{\infty}$. So, we conclude from \LEM{sigma2} that $\uuU_N \times \Tt_N$
has property $(\sigma)$ with respect to $(\uuU_N \times (\Kk_n \cap \Tt_N))_{n=1}^{\infty}$ and thus
the unitary action on $\Tt_N$ is continuous iff it is so on $\uuU_N \times (\Kk_n \cap \Tt_N)$ (for
each $n > 0$), which is immediate, because $\Kk_n \cap \Tt_N \subset \Tt^{(n)}_N$ (see also
($\Delta$)). Similarly, $\Tt \times \Tt$ has property $(\sigma)$ with respect to $(\Kk_n \times
\Kk_n)_{n=1}^{\infty}$ and therefore ``$\oplus$'' is continuous (as continuous on $\Tt^{(n)} \times
\Tt^{(n)}$ for each $n$). Further, to convince oneself that the set $\Hh_N \df \Tt_N \oplus \Tt$ is
closed, we use the definition of the topology on $\Tt$: it suffices to check that $\Hh_N \cap
\Tt^{(n)}$ is closed in $\Tt^{(n)}$ (for all $n$). But $\Hh_N \cap \Tt^{(n)} = \Tt^{(n)}_N \oplus
\Tt^{(n)}$ (because $\Tt^{(n)}$ is a subtower of $\Tt$) and the restriction of the addition
of $\Tt^{(n)}$ to $\Tt^{(n)}_N \times \Tt^{(n)}$ is a closed embedding, from which we conclude
the closedness of $\Hh_N$. So, $\Hh_N$ has property $(\sigma)$ with respect to $(\Kk_n \cap
\Hh_N)_{n=1}^{\infty}$. Finally, since the addition on $\Tt$ is one-to-one, to complete the proof
that $\Tt$ is an m-tower, it remains to show that the function $\Psi\dd \Hh_N \ni \xX \oplus \yY
\mapsto (\xX,\yY) \in \Tt_N \times \Tt$ is continuous. Again, we use \LEM{para}, according to which
it is enough that $\Psi\bigr|_{\Hh_N \cap \Kk_n}$ is continuous (for any $n$). But $\Kk_n \subset
\Tt^{(n)}$, $\Hh_N \cap \Tt^{(n)} = \Hh'_N \df \Tt^{(n)}_N \oplus \Tt^{(n)}$ and
$\Psi\bigr|_{\Hh'_N}$ is continuous (by (T1) for $\Tt^{(n)}$).\par
To finish the proof of the whole theorem, it remains to check that $C^*(\Tt)$ is ($*$-isomorphic to)
the inverse limit of the algebras $C^*(\Tt^{(n)})$, which we shall now briefly do. Let $f_n \in
C^*(\Tt^{(n)})$ be such that $\sup_{n\geqsl1} \|f_n\| < \infty$ and $f_{n+1}$ extends $f_n$. Then
there is a (unique) function $f\dd \Tt \to \mmM$ which extends all maps $f_n$. We see that $f$ is
bounded. Moreover, $f\bigr|_{\Tt_m}$ is continuous, because $f\bigr|_{\Kk_n \cap \Tt_m} =
f_n\bigr|_{\Kk_n \cap \Tt_m}$ is so (for any $n$ and $m$). This implies that $f \in C^*(\Tt)$. So,
the assignment $f \mapsto (f\bigr|_{\Tt^{(n)}})_{n=1}^{\infty}$ correctly defines a $*$-isomorphism
from $C^*(\Tt)$ onto the inverse limit of the algebras $C^*(\Tt^{(n)})$, which is in turn
$*$-isomorphic to $\aaA$.
\end{proof}

\begin{rem}{nonuniq}
There is a delicate difference, specified in \LEM{char}, between residually finite-dimensional (RFD)
$C^*$-algebras and inverse limits of subhomogeneous $C^*$-algebras. Indeed, it is well-known that
in any RFD $C^*$-algebra $\aaA$ there always exists a descending sequence $\JjJ_1,\JjJ_2,\ldots$
of ideals such that conditions (id1) and (id2) hold. So, (id3) is the property that distinguishes
inverse limits among RFD algebras.\par
It is also worth noting that, in general, the m-tower $\Tt$ which appears in item (i) of \THM{inv}
is nonunique (up to isomorphism).
\end{rem}

\section{Shrinking algebras}

\begin{dfn}{shrink}
A $C^*$-algebra $\aaA$ is said to be \textit{shrinking} if:
\begin{enumerate}[(SH1)]
\item $\aaA$ is residually finite-dimensional; that is, finite-dimensional representations of $\aaA$
 separate points of $\aaA$; and
\item whenever $\pi_1,\pi_2,\ldots$ is a sequence of finite-dimensional irreducible representations
 of $\aaA$ such that $\lim_{n\to\infty} d(\pi_n) = \infty$, then $\lim_{n\to\infty} \|\pi(a)\| = 0$
 for all $a \in \aaA$.
\end{enumerate}
\end{dfn}

Note that each subhomogeneous $C^*$-algebra is shrinking. (Conversely, if a shrinking $C^*$-algebra
is unital, then it is subhomogeneous.)\par
The proof of the following simple characterization of shrinking $C^*$-algebras is left
to the reader.

\begin{lem}{shrink}
For any element $x$ of a $C^*$-algebra $\aaA$ and each $n > 0$ put $c_n(x) \df \sup\{\|\pi(x)\|\dd\
\pi \in \core(\Zz(\aaA)) \cap \Zz_n(\aaA)\}$ \textup{(}where $\sup(\varempty) \df 0$\textup{)}. Then
$\aaA$ is shrinking iff $\aaA$ is residually finite-dimensional and $\lim_{n\to\infty} c_n(x) = 0$
for each $x \in \aaA$.
\end{lem}

The aim of this section is (in a sense) the converse to the following simple

\begin{pro}{shrink}
For any elc m-tower $\Tt$, $C^*_0(\Tt)$ is shrinking.
\end{pro}
\begin{proof}
Just use the definition of $C^*_0(\Tt)$ and apply \COR{irr}.
\end{proof}

\begin{dfn}{p-solid}
A pointed tower $(\Tt,\theta)$ is said to be \textit{solid} if $\Tt$ is a proper m-tower and for any
$\tT \in \core(\Tt)$ different from $\theta$ there exists $g \in C^*_0(\Tt,\theta)$ such that
$g(\tT) \neq 0$. (Recall that proper m-towers are elc.)
\end{dfn}

It is an easy exercise that a pointed subtower of a solid pointed tower is a solid pointed tower
as well.

\begin{thm}{shrink}
For every shrinking $C^*$-algebra $\aaA$, $(\Zz(\aaA),\theta_{\aaA})$ is a solid pointed tower and
$J_{\aaA}\dd \aaA \to C^*_0(\Zz(\aaA),\theta_{\aaA})$ is a well defined $*$-isomorphism.\par
Conversely, if $(\Ss,\kappa)$ is a solid pointed tower, then $C^*_0(\Ss,\kappa)$ is a shrinking
$C^*$-algebra and for every its finite-dimensional representation $\pi$ there exists a unique point
$\sS \in \Ss$ such that $\pi = \pi_{\sS}$ where
\begin{equation}\label{eqn:pit}
\pi_{\sS}\dd C^*_0(\Ss,\kappa) \ni u \mapsto u(\sS) \in \mmM_{d(\sS)}.
\end{equation}
Moreover, the assignment $\sS \mapsto \pi_{\sS}$ correctly defines an isomorphism between
the pointed towers $(\Ss,\kappa)$ and $\Zz(C^*_0(\Ss,\kappa),\theta_{C^*_0(\Ss,\kappa)})$.
\end{thm}
\begin{proof}
First assume $\aaA$ is a shrinking $C^*$-algebra. It follows from \LEM[s]{concr} and \LEM[]{shrink}
that $(\Zz(\aaA),\theta_{\aaA})$ is a solid pointed tower and that $J_{\aaA}$ takes values
in $C^*_0(\Zz(\aaA),\theta_{\aaA})$ and is one-to-one. So, we only need to check that the closure
of $\EeE \df J_{\aaA}(\aaA)$ in $C^*_0(\Zz(\aaA))$ coincides with $C^*_0(\Zz(\aaA),\theta_{\aaA})$,
which immediately follows from \THM{SW} (cf.\ the proof of \THM{subh1}).\par
Now assume $(\Ss,\kappa)$ is a solid pointed tower. It follows from \PRO{shrink} that
$C^*_0(\Ss,\kappa)$ is shrinking (use the fact that every nonzero irreducible representation
of $C^*_0(\Ss,\kappa)$ extends uniquely to an irreducible representation of $C^*_0(\Ss)$). Take
an arbitrary finite-dimensional representation $\pi$ of $C^*_0(\Ss,\kappa)$. Then there are
a unitary matrix $U$ (of a respective degree) and a finite system $\pi^{(1)},\ldots,\pi^{(k)}$
of irreducible (possibly zero) representations such that $\pi = U . (\bigoplus_{j=1}^n \pi^{(j)})$.
If $\pi^{(j)} = 0$, it coincides with $\pi_{\kappa}$ (see \eqref{eqn:pit}); and if $\pi^{(j)}$ is
nonzero, we infer from \COR{irr} that there is $\sS_j \in \core(\Ss)$ different from $\kappa$ such
that $\pi^{(j)} = \pi_{\sS_j}$ (since every nonzero irreducible representation
of $C^*_0(\Ss,\kappa)$ extends uniquely to an irreducible representation of $C^*_0(\Ss)$). Thus,
we see that $\pi = \pi_{\sS}$ for some $\sS \in \Ss$. We turn to the uniqueness of $\sS$. Let
$\sS_1$ and $\sS_2$ be two elements of $\Ss$ such that $d(\sS_1) = d(\sS_2)$. In order to construct
a map $u \in C^*_0(\Ss,\kappa)$ with $u(\sS_1) \neq u(\sS_2)$, we shall improve the argument
presented in the proof of \COR{dist}. First observe that for any mutually disjoint irreducible
elements $\tT_0,\ldots,\tT_k \in \Ss$ with $\tT_0 \neq \kappa$ and each $A \in \mmM_{d(\tT_0)}$
there is $g \in C^*_0(\Ss,\kappa)$ such that $g(\tT_0) = A$ and $g(\tT_j) = 0$ for $j=1,\ldots,k$.
Indeed, there is $u_0 \in C^*_0(\Ss,\kappa)$ such that $u_0(\tT_0) = I_{d(\tT_0)}$ (since
$\{u(\tT_0)\dd\ u \in C^*_0(\Ss,\kappa)\}$ is a nonzero ideal in $\{u(\tT_0)\dd\ u \in C^*(\Ss)\} =
\mmM_{d(\tT_0)}$). Further, there is $u_1 \in C^*(\Ss)$ such that $u(\tT_0) = A$ and $u(\tT_j) = 0$
for $j=1,\ldots,k$. Then $g \df u_0 u_1 \in C^*_0(\Ss,\kappa)$ is a map we searched for.
It immediately follows from this property that for any system $\tT_1,\ldots,\tT_k$ of mutually
disjoint irreducible elements of $\Ss$ different from $\kappa$ and any system $A_1,\ldots,A_k$
of matrices such that $d(A_j) = d(\tT_j)$ for each $j$, there exists a map $g \in C^*_0(\Ss,\kappa)$
such that $g(\tT_j) = A_j$ for any $j$.\par
First assume that $\sS_1 \not\equiv \sS_2$. Let $\aA_1^{(1)},\ldots,\aA_{p_1}^{(1)}$ and
$\aA_1^{(2)},\ldots,\aA_{p_2}^{(2)}$ be two systems of mutually disjoint irreducible elements
of $\Ss$ such that, for $\epsi=1,2$,
\begin{equation*}
\sS_{\epsi} = U_{\epsi} . \Bigl(\bigoplus_{j=1}^{p_{\epsi}} (\alpha_j^{(\epsi)} \odot
\aA_j^{(\epsi)})\Bigr)
\end{equation*}
for some positive integers $\alpha_1^{(\epsi)},\ldots,\alpha_{p_{\epsi}}^{(\epsi)}$ and a unitary
matrix $U_{\epsi}$. Since $\sS_1 \not\equiv \sS_2$, as in the proof of \COR{dist}, we may and do
assume (with no loss of generality) that either $\aA^{(1)}_1$ is disjoint from each of $\aA_k^{(2)}$
or $\aA_1^{(1)} \equiv \aA_1^{(2)}$ and $\alpha_1^{(1)} \neq \alpha_1^{(2)}$. In the former case
we deduce that there is a map $u \in C^*_0(\Ss,\kappa)$ such that
\begin{itemize}
\item either $u$ vanishes at any $\aA_k^{(2)}$ and $\aA_j^{(1)}$ for $j > 1$, and $u(\aA_1^{(1)})
 = I_{d(\aA_1^{(1)})}$ (this is possible when $\aA_1^{(1)} \neq \kappa$); or
\item $u(\aA_j^{(\epsi)}) = I_{d(\aA_j^{(\epsi)})}$ for any $(\epsi,j) \neq (1,1)$ and
 $u(\aA_1^{(1)}) = 0$ (this is possible when $\aA_1^{(1)} = \kappa$).
\end{itemize}
In both the above situations one easily deduces that $u(\sS_1) \neq u(\sS_2)$. In the latter case
(when $\aA_1^{(1)} \equiv \aA_1^{(2)}$ and $\alpha_1^{(1)} \neq \alpha_1^{(2)}$) we proceed
similarly: there is a map $u \in C^*_0(\Ss,\kappa)$ such that
\begin{itemize}
\item either $u$ vanishes at any $\aA_k^{(2)}$ and $\aA_j^{(1)}$ for $j > 1$ and $k > 1$, and
 $u(\aA_1^{(1)}) = u(\aA_1^{(2)}) = I_{d(\aA_1^{(1)})}$ (this is possible when $\aA_1^{(1)} \neq
 \kappa$); or
\item $u(\aA_j^{(\epsi)}) = I_{d(\aA_j^{(\epsi)})}$ for any $(\epsi,j) \notin \{(1,1),(2,1)\}$ and
 $u(\aA_1^{(1)}) = u(\aA_1^{(2)}) = 0$ (this is possible when $\aA_1^{(1)} = \kappa$).
\end{itemize}
Again, we see that $u(\sS_1) \neq u(\sS_2)$ in both the above situations.\par
Finally, we assume $\sS_1 \equiv \sS_2$. Then we can find two unitary matrices $U_1$ and $U_2$, and
finite systems $\aA_1,\ldots,\aA_p$ and $\alpha_1,\ldots,\alpha_p$ of, respectively, mutually
disjoint irreducible elements of $\Ss$ and positive integers such that $\sS_{\epsi} = U_{\epsi} .
(\bigoplus_{j=1}^p (\alpha_j \odot \aA_j))$ for $\epsi=1,2$. Since $\sS_1 \neq \sS_2$, we see that
$U_2^{-1} U_1 \notin \stab(\bigoplus_{j=1}^p (\alpha_j \odot \aA_j))$. Axioms (mT1)--(mT4) imply
that $\stab(\bigoplus_{j=1}^p (\alpha_j \odot \aA_j)) = \{\bigoplus_{j=1}^p (I_{d(\aA_j)} \otimes
V_j)\dd\ V_j \in \uuU_{\alpha_j}\}$. Now take a system $A_1,\ldots,A_p$ of mutually disjoint (that
is, mutually unitarily inequivalent) irreducible matrices such that $d(A_j) = d(\aA_j)$, and $A_j =
0$ provided $\aA_j = \kappa$. There exists $g \in C^*_0(\Ss,\kappa)$ such that $g(\aA_j) = A_j$ for
each $j$ (by the previous part of the proof). Then $g(\sS_1) \neq g(\sS_2)$ (cf.\ the proof
of \COR{dist}).\par
It is now clear that the assignment $\sS \mapsto \pi_{\sS}$ correctly defines a bijective morphism
$\Phi$ from $(\Ss,\kappa)$ onto $(\Zz(C^*_0(\Ss,\kappa)),\theta_{C^*_0(\Ss,\kappa)})$. Since $\Ss$
is proper, we infer that $\Phi$ is a homeomorphism, and we are done.
\end{proof}

\begin{cor}{CCR}
Every irreducible representation of a shrinking $C^*$-algebra is finite-dimensional.
\end{cor}
\begin{proof}
We may think of a shrinking $C^*$-algebra as of $C^*_0(\Tt,\theta)$ (where $(\Tt,\theta)$ is a solid
pointed tower). We know from \THM{irr} that all irreducible representations of $C^*_0(\Tt)$ are
finite-dimensional. So, the assertion of the corollary follows from the fact that
$C^*_0(\Tt,\theta)$ either coincides with $C^*_0(\Tt)$ or is an ideal in $C^*_0(\Tt)$ of codimension
equal to $1$.
\end{proof}

\begin{cor}{sub-quo}
Quotients and $C^*$-subalgebras of shrinking $C^*$-algebras as well as direct products of their
collections \textup{(}of arbitrary size\textup{)} are shrinking $C^*$-algebras.
\end{cor}
\begin{proof}
For quotients condition (SH1) follows from \COR{CCR}, whereas (SH2) is immediate. In the opposite,
for subalgebras (SH1) is trivial, whereas (SH2) is implied by \COR{CCR}. Instead of proving that
direct products of shrinking $C^*$-algebras are shrinking as well, we shall show much more general
property, which reads as follows. If a $C^*$-algebra $\aaA$ admits a collection of ideals
$\{\JjJ_{\xi}\}_{\xi\in\Xi}$ each of which is a shrinking $C^*$-algebra and $\aaA$ is a unique
(closed) ideal that includes their union, then $\aaA$ itself is shrinking. This property follows
from the following two facts:
\begin{itemize}
\item for each irreducible representation $\pi$ of $\aaA$ there exists $\xi \in \Xi$ such that
 $\pi|_{\JjJ_{\xi}}$ is an irreducible representation of $\JjJ_{\xi}$, and thus $\pi$ is
 finite-dimensional, by \COR{CCR} (this implies that $\aaA$ is residually finite-dimensional);
\item for any sequence $\pi_1,\pi_2,\ldots$ of irreducible representations of $\aaA$ such that
 $\lim_{n\to\infty} d(\pi_n) = \infty$, the set of all $a \in \aaA$ for which $\lim_{n\to\infty}
 \|\pi_n(a)\| = 0$ is a closed ideal that contains $\bigcup_{\xi\in\Xi} \JjJ_{\xi}$.
\end{itemize}
(In particular, if $\aaA$ coincides with the closure of $\bigcup_{n=0}^{\infty} \JjJ_n$ where
$\JjJ_0 = \{0\}$ and, for $n > 0$, $\JjJ_n \supset \JjJ_{n-1}$ is an ideal in $\aaA$ such that
$\JjJ_n / \JjJ_{n-1}$ is subhomogeneous, then $\aaA$ is shrinking.)
\end{proof}

\begin{cor}{shrink}
For a $C^*$-algebra $\aaA$ \tfcae
\begin{enumerate}[\upshape(i)]
\item $\aaA$ is shrinking;
\item there is a collection $\rrR$ of finite-dimensional representations of $\aaA$ that separate
 points of $\aaA$ and satisfy the following condition: if $\pi_1,\pi_2,\ldots \in \rrR$ are such
 that $\lim_{n\to\infty} d(\pi_n) = 0$, then $\lim_{n\to\infty} \pi_n(a) = 0$ for each $a \in \aaA$;
\item there is a collection $\jjJ$ of ideals in $\aaA$ such that:
 \begin{enumerate}[\upshape(s1)]
 \item $\bigcap_{\JjJ\in\jjJ} \JjJ = \{0\}$; and
 \item for any $\JjJ \in \jjJ$, the quotiont algebra $\aaA / \JjJ$ is shrinking; and
 \item for any $a \in \aaA$ and $\epsi > 0$ there is a finite set $\jjJ_0 \subset \jjJ$ such that
  $\|\pi_{\JjJ}(a)\| \leqsl \epsi$ for each $\JjJ \in \jjJ \setminus \jjJ_0$ where $\pi_{\JjJ}\dd
  \aaA \to \aaA / \JjJ$ is the canonical projection.
 \end{enumerate}
\end{enumerate}
\end{cor}
\begin{proof}
First assume $\rrR$ is a collection as specified in (ii). For any $n > 0$, denote by $\JjJ_n$
the set of all points $a \in \aaA$ which each $n$-dimensional representations belonging to $\rrR$
vanishes at. It is easy to see that then the family $\jjJ \df \{\JjJ_n\dd\ n > 0\}$ satisfies all
conditions (s1)--(s3). So, (iii) follows from (ii). Now let $\jjJ$ be as specified in (i). Let
$\ddD$ denote the direct product of all algebras $\aaA / \JjJ$ (where $\JjJ$ runs over all elements
of $\jjJ$) and let $\Phi\dd \aaA \to \ddD$ be given by $\Phi(a) \df (\pi_{\JjJ}(a))_{\JjJ\in\jjJ}$.
Condition (s3) implies that $\Phi$ is a well defined $*$-homomorphism, whereas (s1) shows that
$\Phi$ is one-to-one. So, $\aaA$ is $*$-isomorphic to a $C^*$-subalgebra of the direct product
of shrinking $C^*$-algebras (by (s2)), and thus $\aaA$ itself is shrinking, thanks
to \COR{sub-quo}.\par
Since (ii) is trivially implied by (i), the proof is complete.
\end{proof}

\begin{rem}{extens}
Since all unital shrinking $C^*$-algebras are subhomogeneous, we see that the unitization
of a shrinking $C^*$-algebra is not shrinking in general. So, the class of shrinking $C^*$-algebras
is not closed under extensions.
\end{rem}

\begin{pro}{ext}
If $(\Ss,\theta)$ is a pointed subtower of a solid pointed tower $(\Tt,\theta)$, then each map
$u \in C^*_0(\Ss,\theta)$ extends to a map $v \in C^*_0(\Tt,\theta)$ whose norm is arbitrarily close
to $\|u\|$.
\end{pro}
\begin{proof}
Just apply \THM{ext}.
\end{proof}

The next three results are consequences of \THM{shrink} (and \PRO{ext}) and may be shown similarly
as their counterparts for subhomogeneous $C^*$-algebras (which were proved in Section~5) and
therefore their proofs are skipped and left to the reader.

\begin{cor}{id}
Let $(\Tt,\theta)$ be a solid pointed tower. For any ideal $\JjJ$ in $C^*_0(\Tt,\theta)$ there is
a unique pointed subtower $(\Ss,\theta)$ of $(\Tt,\theta)$ such that $\JjJ = \JjJ_{\Ss}$ where
$\JjJ_{\Ss}$ consists of all maps from $C^*_0(\Tt,\theta)$ that vanish at each point of $\Ss$.
\end{cor}

\begin{cor}{*homo}
For every $*$-homomorphism $\Phi\dd C^*_0(\Tt,\theta) \to C^*_0(\Ss,\kappa)$ \textup{(}where
$(\Tt,\theta)$ and $(\Ss,\kappa)$ are solid pointed towers\textup{)} there exists a unique morphism
$\tau\dd (\Ss,\kappa) \to (\Tt,\theta)$ such that $\Phi = \Phi_{\tau}$ where $\Phi_{\tau}\dd
C^*_0(\Tt,\theta) \ni f \mapsto f \circ \tau \in C^*_0(\Ss,\kappa)$.\par
In particular, two shrinking $C^*$-algebras are isomorphic iff their pointed towers of all
finite-dimensional representations are isomorphic.
\end{cor}

\begin{cor}{sur}
Let $\tau\dd (\Ss,\kappa) \to (\Tt,\theta)$ be a morphism between two solid pointed m-towers.
The $*$-homomorphism $\Phi_{\tau}\dd C^*_0(\Tt,\theta) \to C^*_0(\Ss,\kappa)$ is surjective iff
$\tau$ is one-to-one. If this happens, $\tau(\Ss)$ is a subtower of $\Tt$ and $\tau$ is
an isomorphism from $(\Ss,\kappa)$ onto $(\tau(\Ss),\theta)$.
\end{cor}

In the next two results we establish two typical properties.

\begin{pro}{weight}
For any solid pointed tower $(\Tt,\theta)$, the topological spaces $\Tt$ and $C^*_0(\Tt,\theta)$
have the same topological weight. In particular, $\Tt$ is metrizable iff $C^*_0(\Tt,\theta)$ is
separable.
\end{pro}
\begin{proof}
For simplicity, we shall denote the weight of a topological space $Y$ by $w(Y)$ (here the weights
are understood as infinite cardinal numbers, even for finite topological spaces). Let $\hat{\Tt} =
\Tt \sqcup \{\infty\}$ and $\ddD$ denote, respectively, a one-point compactification of $\Tt$ and
the direct product the $C^*$-algebras $\mmM_n$; that is, $\ddD$ consists of all sequences
$(X_n)_{n=1}^{\infty} \in \prod_{n=1}^{\infty} \mmM_n$ such that $\lim_{n\to\infty} \|X_n\| = 0$.
For each $A \in \mmM_k$, we denote by $A^{\#}$ the element $(X_n)_{n=1}^{\infty}$ of $\ddD$ such
that $X_k = A$ and $X_n = 0$ for $n \neq k$. Further, for any $u \in C^*_0(\Tt,\theta)$ we define
a function $\Psi(u)\dd \hat{\Tt} \to \ddD$ by $\Psi(u)(\tT) = (u(\tT))^{\#}$ for $\tT \in \Tt$ and
$\Psi(u)(\infty) = 0 \in \ddD$ (cf.\ the proof of \THM{SW}). In this way one obtains an isometric
$*$-homomorphism $\Psi\dd C^*_0(\Tt,\theta) \to C(\hat{\Tt},\ddD)$. Consequently,
$w(C^*_0(\Tt,\theta)) \leqsl w(C(\hat{\Tt},\ddD))$. But $w(C(\hat{\Tt},\ddD)) \leqsl w(\hat{\Tt})$,
thanks to Theorem~3.4.16 in \cite{eng} (because $\ddD$ is separable), and hence
$w(C^*_0(\Tt,\theta)) \leqsl w(\Tt) (= w(\hat{\Tt}))$. To prove the reverse inequality, take a dense
subset $\Lambda$ of the closed unit ball in $\aaA \df C^*_0(\Tt,\theta)$ whose cardinality does not
exceed $w(\aaA)$. Now it suffices to repeat the reasoning presented in the proof of \LEM{concr}
to conclude that $\Zz(\aaA)$ is isomorphic (and hence homeomorphic) to a standard tower $\ttT$
in $\mmM[\Lambda]$. One concludes that then $w(\Zz(\aaA)) = w(\ttT) \leqsl \card(\Lambda)$. Finally,
\THM{shrink} implies that $\Zz(\aaA)$ is homeomorphic to $\Tt$ and therefore $w(\Tt) = w(\Zz(\aaA))
\leqsl w(\aaA)$. The additional claim readily follows from the main part of the proposition.
\end{proof}

For the purpose of the next result, for any $\ueX = (X_1,\ldots,X_N) \in \mmM[\{1,\ldots,N\}]$,
we shall denote by $\|\ueX\|$ the number $\max(\|X_1\|,\ldots,\|X_N\|)$. Additionally, the subspace
of $\mmM[\Lambda]$ consisting of all $N$-tuples of selfadjoint matrices will be denoted
by $\mmM_s(\{1,\ldots,N\})$.

\begin{pro}{gener}
A shrinking $C^*$-algebra $\aaA$ is generated by $N$ elements \textup{(}resp.\ $N$ selfadjoint
elements\textup{)} iff $(\Zz(\aaA),\theta_{\aaA})$ is isomorphic to a standard pointed tower
$(\ttT,\theta_{\ttT}) \subset \mmM[\{1,\ldots,N\}]$ \textup{(}resp.\ $(\ttT,\theta_{\ttT}) \subset
\mmM_s[\{1,\ldots,N\}]$\textup{)} such that
\begin{equation*}
\lim_{n\to\infty} \sup\{\|\ueX\|\dd\ \ueX \in \core(\ttT) \cap \ttT_n\} = 0.
\end{equation*}
\end{pro}
\begin{proof}
First assume $x_1,\ldots,x_N \in \aaA$ are (selfadjoint) elements which have norms not greater than
$1$ and generate $\aaA$. We leave it as a simple exercise that then the set $\ttT \df \{(\pi(x_1),
\ldots,\pi(x_N))\dd\ \pi \in \Zz(\aaA)\} \subset \mmM[\{1,\ldots,N\}]$ contains $\theta_{\ttT} \df
(0,\ldots,0) \in \mmM_1^N$ and $(\ttT,\theta_{\ttT})$ is a standard tower isomorphic
to $(\Zz(\aaA),\theta_{\aaA})$ which has all desired properties.\par
Now assume $(\Zz(\aaA),\theta_{\aaA})$ is isomorphic to $(\ttT,\theta_{\ttT})$ where
$(\ttT,\theta_{\ttT})$ is as specified in the proposition. Then $\aaA$ is $*$-isomorphic
to $\ddD \df C^*_0(\ttT,\theta_{\ttT})$ (by \THM{shrink}) and the projections $p_1,\ldots,p_N\dd
\ttT \to \mmM$ onto respective coordinates belong to $\ddD$. It suffices to check that
the $*$-algebra $\EeE$ generated by $p_1,\ldots,p_N$ is dense in $\ddD$. To this end, fix two
disjoint irreducible elements $\ueX, \ueY \in \core(\ttT)$ different from $\theta_{\ttT}$. Since
$\WwW'(\ueX)$ contains only scalar multiples of the unit matrix, we infer from von Neumann's double
commutant theorem that $\mmM_{d(\ueX)}$ coincides with the smallest von Neumann algebra generated
by all entries of $\ueX$. Since $\ueX \neq \theta_{\ttT}$, it follows that the smallest $*$-algebra
containing all entries of $\ueX$ has codimension (in $\mmM_{d(\ueX)}$) not greater than $1$ and thus
is a (nonzero) ideal in $\mmM_{d(\ueX)}$. Consequently,
\begin{equation}\label{eqn:aux21}
\{p(\ueX)\dd\ p \in \EeE\} = \mmM_{d(\ueX)} \qquad (\textup{and similarly }
\{p(\ueY)\dd\ p \in \EeE\} = \mmM_{d(\ueY)}).
\end{equation}
Further, since $\ueX \perp \ueY$, the matrix $P \df I_{d(\ueX)} \oplus 0_{d(\ueY)}$ is a central
projection in $\WwW'(\ueX \oplus \ueY)$ (cf.\ the proof of \PRO{std2}). This implies that $P$
belongs to the smallest von Neumann algebra $\wwW$ generated by all entries of $\ueX$ and $\ueY$.
This property, combined with \eqref{eqn:aux21}, yields that $\wwW = \mmM_{d(\ueX\oplus\ueY)}$. So,
arguing as before, we infer that $\{p(\ueX \oplus \ueY)\dd\ p \in \EeE\} =
\mmM_{d(\ueX \oplus \ueY)}$, which implies that there is $p \in \EeE$ such that $p(\ueX) =
I_{d(\ueX)}$ and $p(\ueY) = 0$. Now it suffices to apply \THM{SW} to conclude that $\EeE$ is dense.
\end{proof}

The rest of the section is devoted to investigations of towers which are isomorphic to some
$\Zz(\aaA)$ where $\aaA$ runs over all shrinking $C^*$-algebras. We begin with a simple

\begin{lem}{solid}
For a proper m-tower $\Tt$ \tfcae
\begin{enumerate}[\upshape(i)]
\item there exists $\kappa \in \Tt_1$ such that $(\Tt,\kappa)$ is a solid pointed tower;
\item exactly one of the following two conditions holds:
 \begin{itemize}
 \item[(reg)] either $\Tt_1 \neq \varempty$ and $\Tt_{(0)} = \varempty$ \textup{(}in that case
  $C^*_0(\Tt,\theta)$ is a solid pointed tower and $C^*_0(\Tt) / C^*_0(\Tt,\theta)$ is
  one-dimensional for any $\theta \in \Tt_1$\textup{)}; or
 \item[(sng)] $\Tt_{(0)} = \{n \odot \kappa\dd\ n >0 \}$ for some $\kappa \in \Tt_1$
  \textup{(}in that case $C^*_0(\Tt) = C^*_0(\Tt,\kappa)$ and $\kappa$ is a unique point $\theta \in
  \Tt_1$ such that $(\Tt,\theta)$ is solid\textup{)}.
 \end{itemize}
\end{enumerate}
\end{lem}
\begin{proof}
All we need to observe is that for any $\theta \in \Tt_1$, the set $\Ss_{\theta} \df \{n \odot
\theta\dd\ n > 0\}$ is a subtower of $\Tt$ (thanks to (mT4)), and that $\Tt_{(0)} \subset
\Ss_{\theta}$ provided $(\Tt,\theta)$ is solid. The assertion of the lemma may now readily be
derived. We leave the details as an exercise.
\end{proof}

\begin{dfn}{solid}
A tower $\Tt$ is said to be \textit{solid} if $\Tt$ is a proper m-tower and either $\Tt_{(0)}$ is
empty or $\Tt_{(0)} = \{n \odot \theta\}$ for some $\theta \in \Tt_1$. In the former case $\Tt$ is
called \textit{regular}, whereas in the latter \textit{singular}.
\end{dfn}

The usage of the adjectives ``regular'' and ``singular'' in the above naming shall be explained
in the next section.\par
These are easy observations that for each unital subhomogeneous $C^*$-algebra $\aaA$, the tower
$\Xx(\aaA)$ (of all unital finite-dimensional representations) is solid and regular; and that
the tower $\Zz(\ddD)$ of any shrinking $C^*$-algebra $\ddD$ is solid.\par
The following is an immediate consequence of \THM{shrink} and \LEM{solid} (we skip the proof).

\begin{cor}{uniq-solid}
If $\aaA$ is a shrinking $C^*$-algebra whose tower $\Zz(\aaA)$ is singular, then $\aaA$ is naturally
$*$-isomorphic to $C^*_0(\Zz(\aaA))$.
\end{cor}

For the purpose of the next result, let us call each set of the form $\core(\Tt) \cap
\bigcup_{k=N}^{\infty} \Tt_k$ (where $\Tt$ is a tower and $N > 0$) a \textit{tail} of the core.

\begin{pro}{T0}
Let $\Tt$ be a proper m-tower. For an element $\tT \in \Tt$ \tfcae
\begin{enumerate}[\upshape(i)]
\item $\tT \in \Tt_{(0)}$;
\item every subtower of $\Tt$ that includes a tail of the core contains $\tT$;
\item whenever $g \in C^*_0(\Tt)$ vanishes at each point of some tail of the core, then
 $g(\tT) = 0$.
\end{enumerate}
\end{pro}
\begin{proof}
First assume (ii) holds. Take arbitrary $g \in C^*_0(\Tt)$ and $\epsi > 0$. Observe that the set
$\{\sS \in \Tt\dd\ \|g(\sS)\| \leqsl \epsi\}$ is a subtower that includes a tail of the core, and
thus $\|g(\tT)\| \leqsl \epsi$. This proves (i). Since (iii) immediately follows from (i), we only
need to check that (ii) follows from (iii). We argue by a contradiction. Assume there is a subtower
$\Ss$ of $\Tt$ that includes a tail of the core and excludes $\tT$. Then there exists $\sS \in
\core(\Tt)$ such that $\sS \notin \Ss$ and $\sS \preccurlyeq \tT$. Now it follows from \COR{dist}
that there is $g \in C^*(\Tt)$ that vanishes at each point of $\Ss$, but not in $\sS$. Then also
$g(\tT) \neq 0$. Since then automatically $g \in C^*_0(\Tt)$, we obtain a contradiction with (iii).
\end{proof}

\begin{dfn}{T[n]}
For any tower $\Tt$ and a positive integer $n$, we denote by $\Tt[n]$ the smallest subtower of $\Tt$
that contains $\bigcup_{k=n}^{\infty} \Tt_k$. Additionally, we put $\Tt[\infty] \df
\bigcap_{n=1}^{\infty} \Tt[n]$.
\end{dfn}

\begin{cor}{T0}
For any proper m-tower $\Tt$, $\Tt_{(0)} = \Tt[\infty]$.
\end{cor}
\begin{proof}
Just apply \PRO{T0}.
\end{proof}

\begin{pro}{solid}
A proper m-tower $\Tt$ is solid and regular iff for any $N > 0$ there is $n$ such that $\Tt[n] \cap
(\bigcup_{j=1}^N \Tt_j) = \varempty$.
\end{pro}
\begin{proof}
The `if' part follows from \COR{T0}. To show the `only if' part, assume $\Tt$ is solid and regular.
Using \COR{T0}, for any $\tT \in \bigcup_{j=1}^N \Tt_j$ take $\nu(\tT) > 0$ such that $\tT \in \Tt
\setminus \Tt[\nu(\tT)]$. Then, it follows from the compactness argument that there is a finite
system $\tT_1,\ldots,\tT_s$ of elements of $\Kk \df \bigcup_{j=1}^N \Tt_j$ for which $\Kk \subset
\bigcup_{k=1}^s (\Tt \setminus \Tt[\nu(\tT_k)])$. Now it suffices to define $n$ as the maximum
of $\nu(\tT_1),\ldots,\nu(\tT_s)$.
\end{proof}

Recall that if $(\Ss,\kappa)$ is a solid pointed tower, then $\Ss$ is solid as well. If $\Ss$ is
singular, $C^*_0(\Ss,\kappa)$ coincides with $C^*_0(\Ss)$ (by \COR{uniq-solid}). So, the case when
$\Ss$ is regular seems to be more interesting.

\begin{pro}{solid-solid}
If $\Tt$ is a regular solid tower and $(\Ss,\kappa)$ is a solid pointed tower such that $C^*_0(\Tt)$
is $*$-isomorphic to $C^*_0(\Ss,\kappa)$, then $\Ss$ is regular as well. What is more, there exists
$g \in C^*_0(\Ss)$ such that $g(\kappa) = I_1$ and $g$ vanishes at each point of $\core(\Ss)$
different from $\kappa$.
\end{pro}
\begin{proof}
(It follows from \PRO{shrink} and \THM{shrink} that such a pointed tower $(\Ss,\kappa)$ exists.)
Observe that $\Ss$ is regular if $\kappa \notin \Ss_{(0)}$. Let $\Psi\dd C^*_0(\Ss,\kappa) \to
C^*_0(\Tt)$ be a $*$-isomorphism. For any $\tT \in \Tt$ there exists a unique point $\tau(\tT) \in
\Ss$ such that $\Psi(u)(\tT) = u(\tau(\tT))$. It follows from \COR{nondeg} and the second part
of \THM{shrink} (and the fact that $\Tt$ is proper) that $\tau\dd \Tt \to \Ss$ is a one-to-one
morphism. We conclude that $\Tt' \df \tau(\Tt)$ is a subtower of $\Ss$ (see the proof of \PRO{sur}).
Now \COR{irr} shows that $\Tt'$ includes $\core(\Ss) \setminus \{\kappa\}$. Finally, it follows from
\COR{dist} that there is $g \in C^*(\Ss)$ that vanishes at each point of $\Tt'$ and sends $\kappa$
onto $I_1$. Then automatically $g \in C^*_0(\Ss)$ and hence $\kappa \notin \Ss_{(0)}$.
\end{proof}

\begin{rem}{T0}
\COR{T0} and \LEM{sub} enable giving a transparent description of $\Tt[N]$ for each $N > 0$ as well
as $\Tt_{(0)}$. Namely, when $N > 0$ is fixed, for each $n \geqsl N$, put $\Dd_N(n) \df \Tt_n \cap
\overline{\core}(\Tt)$. Further, if $k < N$ is such that $\Dd_N(n)$ is defined for all $n > k$, let
$\Dd_N(k)$ be the closure of all $\tT \in \core(\Tt) \cap \Tt_k$ for which there is $\sS \in
\bigcup_{n>k} \Dd_N(n)$ with $\tT \preccurlyeq \sS$. In this way we define $\Dd_N(n-1),\ldots,
\Dd_N(1)$. It follows from the construction that all sets $\Dd_N \df \bigcup_{n=1}^{\infty}
\Dd_N(n)$ as well as $\Dd_{\infty} \df \bigcap_{N=1}^{\infty} \Dd_N$ satisfy the assumption
of \LEM{sub}, from which one readily infers that $\Tt[N] = \grp{\Dd_N}$ and $\Tt_{(0)} =
\grp{\Dd_{\infty}}$ (see \eqref{eqn:subtower} and \COR{T0}).
\end{rem}

\section{Shrinking multiplier algebras}

In this section we would like to propose a new concept, specific only for shrinking $C^*$-algebras,
the so-called \textit{shrinking multiplier algebras} of shrinking $C^*$-algebras. We assume that
the reader is familiar with the basics on multiplier algebras of (arbitrary) $C^*$-algebras
(consult, e.g., \S{}II.7.3 in \cite{bla}). For any $C^*$-algebra $\aaA$, $M(\aaA)$ will stand for
the multiplier algebra of $\aaA$. We shall think of $\aaA$ as an ideal in $M(\aaA)$. Additionally,
for any nonzero irreducible representation $\pi$ of $\aaA$, we shall denote by $\tilde{\pi}$
the unique extension of $\pi$ to an irreducible representation of $M(\aaA)$ (see Proposition~2.10.4
in \cite{di2}).

\begin{dfn}{multi}
For any shrinking $C^*$-algebra $\aaA$, the \textit{shrinking multiplier algebra} of $\aaA$, denoted
by $S(\aaA)$, is the set of all points $x \in M(\aaA)$ such that $\lim_{n\to\infty} \tilde{\pi}_n(x)
= 0$ for any sequence $\pi_1,\pi_2,\ldots$ of irreducible representations of $\aaA$ with
$\lim_{n\to\infty} d(\pi_n) = \infty$. It is easy to see that $S(\aaA)$ is a (closed) ideal
in $M(\aaA)$.
\end{dfn}

Our first aim is to show that $S(\aaA)$ is indeed shrinking (provided $\aaA$ is so). To this end,
we need

\begin{lem}{multi}
Let $(\Tt,\theta)$ be a solid pointed tower, $\aaA = C^*_0(\Tt,\theta)$ and $\Dd = \core(\Tt)
\setminus \{\theta\}$. There exists a one-to-one $*$-homomorphism $\Phi\dd M(\aaA) \to
C^*_{\Tt}(\Dd)$ such that $\Phi(f) = f\bigr|_{\Dd}$ for any $f \in \aaA$.
\end{lem}
\begin{proof}
Fix $z \in M(\aaA)$. Since $\aaA$ has an approximate unit (and is an ideal in $M(\aaA)$),
we conclude that there exists a uniformly bounded net $\{v_{\sigma}\}_{\sigma\in\Sigma} \subset
\aaA$ such that $z u = \lim_{\sigma\in\Sigma} v_{\sigma} u$ for any $u \in \aaA$. For any $\tT \in
\Dd$ there is $u_{\tT} \in \aaA$ with $u_{\tT}(\tT) = I_{d(\tT)}$. Then $(z u_{\tT})(\tT) =
\lim_{\sigma\in\Sigma} v_{\sigma}(\tT)$. This shows that the formula $f_z(\tT) \df
\lim_{\sigma\in\Sigma} v_{\sigma}(\tT)$ correctly defines a (necessarily bounded and compatible)
function $f_z\dd \Dd \to \mmM$ such that
\begin{equation}\label{eqn:left}
(z u)(\tT) = f_z(\tT) u(\tT) \qquad (\tT \in \Dd).
\end{equation}
It is also clear that $f_z = z\bigr|_{\Dd}$ for each $z \in \aaA$. We shall now show that $f_z$ is
continuous. To this end, fix $\tT \in \Dd$. It follows from \LEM{core} that $\Dd$ is open in $\Tt$.
So, there is a compact unitarily invariant neighbourhood $\Uu$ of $\tT$ contained in $\Tt_{d(\tT)}
\cap \Dd$ such that $\|u_{\tT}(\sS) - I_{d(\tT)}\| \leqsl \frac12$ for any $\sS \in \Uu$. Then $\Uu$
is a semitower and thus, by \THM{extend}, there is a map $w \in C^*(\Tt)$ such that $w(\sS) =
(u_{\tT}(\sS))^{-1}$ for any $\sS \in \Uu$. Observe that $u \df w \cdot u_{\tT}$ belongs to $\aaA$
and $u(\sS) = I_{d(\tT)}$ for any $\sS \in \Uu$. We now infer from \eqref{eqn:left} that
$f_z$ is continuous on $\Uu$. So, $f_z \in C^*_{\Tt}(\Dd)$. We put $\Phi\dd M(\aaA) \ni z \mapsto
f_z \in C^*_{\Tt}(\Dd)$. It follows from \eqref{eqn:left} that $\Phi$ is linear and multiplicative.
It is also one-to-one, because $\|z\| = \sup\{\|zu\|\dd\ u \in \aaA,\ \|u\| \leqsl 1\}$ for any
$z \in M(\aaA)$. Finally, for any $z \in M(\aaA)$ and $\tT \in \Dd$,
\begin{equation*}
f_{z^*}(\tT) = (u^*_{\tT} z^* u^*_{\tT})(\tT) = (u_{\tT} z u_{\tT})^*(\tT) =
((u_{\tT} z u_{\tT})(\tT))^* = (f_z(\tT))^*,
\end{equation*}
which shows that $f_{z^*} = f_z^*$, and we are done.
\end{proof}

\begin{cor}{multi}
Let $\aaA$ be a shrinking $C^*$-algebra. Then:
\begin{enumerate}[\upshape(a)]
\item $M(\aaA)$ is residually finite-dimensional;
\item if $\aaA$ is $n$-subhomogeneous, so is $M(\aaA)$ and $M(\aaA) = S(\aaA)$;
\item $S(\aaA)$ is shrinking;
\item $S(S(\aaA)) = S(\aaA)$ and $M(S(\aaA)) = M(\aaA)$.
\end{enumerate}
\end{cor}
\begin{proof}
It follows from \THM{shrink} that $\aaA$ is $*$-isomorphic to $C^*_0(\Tt,\theta)$ for some solid
pointed tower $(\Tt,\theta)$. So, we may and do assume that $\aaA = C^*_0(\Tt,\theta)$. Let $\Dd$
and $\Phi\dd M(\aaA) \to C^*_{\Tt}(\Dd)$ be as specified in \LEM{multi}. Since $C^*_{\Tt}(\Dd)$ is
readily seen to be residually finite-dimensional, we get (a). Point (b) is well-known and follows
from the same argument (if $\aaA$ is $n$-subhomogeneous, then $\core(\Tt) \subset \bigcup_{k=1}^n
\Tt_n$ and, consequently, $C^*_{\Tt}(\Dd)$ is $n$-subhomogeneous; it is also trivial that $S(\aaA) =
M(\aaA)$). Now we turn to (c). We shall identify $M(\aaA)$ with $\ddD \df \Phi(\aaA) \subset
C^*_{\Tt}(\Dd)$. Fix $\tT \in \Tt$ and denote by $e_{\tT}$ and $\pi_{\tT}$ the representations
of $\ddD$ and $\aaA$, respectively, which assign to each map $u$ its value at $\tT$. We see that
$e_{\tT}$ is irreducible. So, we infer that $e_{\tT}$ coincides with $\tilde{\pi}_{\tT}$, and hence
$u \in S(\aaA)$ iff $\lim_{n\to\infty} \|e_{\tT_n}(u)\| = 0$ for any sequence $\tT_n \in \Dd$ such
that $\lim_{n\to\infty} d(\tT_n) = \infty$. This shows that the collection $\rrR = \{e_{\tT}\dd \tT
\in \Dd\}$ satisfies condition (ii) of \COR{shrink} (with $S(\aaA)$ in place of $\aaA$).
Consequently, $S(\aaA)$ is shrinking. Finally, $\ddD$ may naturally be identified with $M(S(\aaA))$,
which follows from the fact that $\aaA \subset S(\aaA) \subset M(\aaA)$ and $S(\aaA)$ is an ideal
in $M(\aaA)$, and may be shown as follows. For any $z \in M(S(\aaA))$ denote by $b_z$ a unique point
of $M(\aaA)$ such that $z u = b_z u$ for any $u \in \aaA$. Then $b_z v \in S(\aaA)$ for each $v \in
S(\aaA)$ and hence $z v = b_z v$ (for any $v \in S(\aaA)$), because $\aaA$ is an essential ideal
of $S(\aaA)$. So, $z \mapsto b_z$ defines a $*$-isomorphism of $M(S(\aaA))$ onto $M(\aaA)$ which
leaves the points of $S(\aaA)$ fixed.\par
Knowing that $M(S(\aaA)) = \ddD$, we readily get (d).
\end{proof}

\begin{dfn}{reg-sng}
A shrinking $C^*$-algebra $\aaA$ is \textit{regular} (resp.\ \textit{singular}) if $\Zz(\aaA)$ is
so. $\aaA$ is said to be \textit{closed} if $S(\aaA) = \aaA$.
\end{dfn}

As for (full) multiplier algebras, we have

\begin{pro}{ext-multi}
If $\aaA$ is a shrinking $C^*$-algebra and $\JjJ$ is an essential ideal in $\aaA$, then there exists
a unique one-to-one $*$-homomorphism $\Phi\dd \aaA \to S(\JjJ)$ which leaves the points of $\JjJ$
fixed.
\end{pro}
\begin{proof}
There exists a unique one-to-one $*$-homomorphism $\Phi\dd \aaA \to M(\JjJ)$ which leaves the points
of $\JjJ$ fixed. It suffices to show that $\Phi(\aaA) \subset S(\aaA)$, which simply follows from
the definition of $S(\aaA)$ and the fact that $\aaA$ is shrinking (and $\Phi$ is isometric).
\end{proof}

The above result asserts that $S(\aaA)$ is the largest shrinking $C^*$-algebra that contains $\aaA$
as an essential ideal.

\begin{pro}{closed}
For any regular solid tower $\Tt$, $M(C^*_0(\Tt))$ is $*$-isomorphic to $C^*(\Tt)$. In particular,
$C^*_0(\Tt)$ is closed.
\end{pro}
\begin{proof}
It suffices to prove the first claim. The proof of \COR{multi} shows that for any $z \in
M(C^*_0(\Tt))$ there is a uniformly bounded net $\{v_{\sigma}\}_{\sigma\in\Sigma} \subset
C^*_0(\Tt)$ such that $z u = \lim_{\sigma\in\Sigma} v_{\sigma} u$ for any $u \in C^*_0(\Tt)$. Since
$\Tt_{(0)} = \varempty$, for any $\tT \in \Tt$ there is $u_{\tT} \in C^*_0(\Tt)$ with $u_{\tT}(\tT)
= I_{d(\tT)}$. We conclude that the limit $f_z(\tT) \df \lim_{\sigma\in\Sigma} v_{\sigma}(\tT)$
exists. In this way we have obtained a bounded compatible function $f\dd \Tt \to \mmM$. Arguing
similarly as in the proof of \COR{shrink}, we see that for each $\tT \in \Tt$, there is a map $v \in
C^*_0(\Tt)$ which is constantly equal to $I_{d(\tT)}$ on some neighbourhood of $\tT$ (it suffices
to apply \THM{extend} to the function $\sS \mapsto (u_{\tT}(\sS))^{-1}$ defined on a subtower $\{\sS
\in \Tt\dd\ \|u(\sS) - I_{d(\sS)}\| \leqsl \frac12\}$), and therefore $f_z$ is continuous.
We conclude that the formula $\Phi(z) \df f_z$ correctly defines a function $\Phi\dd M(C^*_0(\Tt))
\to C^*(\Tt)$. It is easy to show that $\Phi$ is a one-to-one $*$-homomorphism such that $(z u)(\tT)
= (\Phi(z))(\tT) u(\tT)$ for any $z \in M(C^*_0(\Tt))$, $\tT \in \Tt$ and $u \in C^*_0(\Tt)$.
The surjectivity of $\Phi$ is left as a simple exercise.
\end{proof}

A partial converse to \PRO{closed} may be formulated as follows.

\begin{pro}{reg-clo}
Every closed regular shrinking $C^*$-algebra is $*$-isomorphic to $C^*_0(\Tt)$ for some regular
solid tower $\Tt$.
\end{pro}
\begin{proof}
Let $\aaA$ be a closed regular shrinking $C^*$-algebra. We know from \THM{shrink} that $\aaA$ is
$*$-isomorphic to $C^*_0(\Zz(\aaA),\theta_{\aaA})$. Since $\aaA$ is regular, so is $\Zz(\aaA)$. Now
put $\Tt \df \{\pi \in \Zz(\aaA)\dd\ \pi \perp \theta_{\aaA}\}$. We claim that $\Tt$ is closed.
Indeed, if $\Tt$ is not closed, one readily concludes that $C^*_0(\Zz(\aaA),\theta_{\aaA})$ is
an essential ideal in $C^*_0(\Zz(\aaA))$. But, since $\Zz(\aaA)$ is regular,
$C^*_0(\Zz(\aaA),\theta_{\aaA}) \neq C^*_0(\Zz(\aaA))$ and thus $\aaA$ is nonclosed (see
\PRO{ext-multi}).\par
Knowing that $\Tt$ is closed, we see that $\Tt$ is a subtower of $\Zz(\aaA)$. Hence $\Tt$ is
regular. Finally, it is easy to see that each $g \in C^*_0(\Tt)$ admits a unique extension to a map
in $C^*_0(\Zz(\aaA),\theta_{\aaA})$ (the extension exists by \THM{extend}) and therefore
the function $C^*_0(\Zz(\aaA),\theta_{\aaA}) \ni f \mapsto f\bigr|_{\Tt} \in C^*_0(\Tt)$ is
a $*$-isomorphism, which finishes the proof.
\end{proof}

The following is a consequence of \PRO[s]{closed} and \PRO[]{reg-clo}. Its proof is left
as an exercise.

\begin{cor}{reg-clo}
There is a one-to-one correspondence between all closed regular shrinking $C^*$-algebras
\textup{(}up to $*$-isomorphims\textup{)} and all regular solid towers
\textup{(}up to isomorphism\textup{)}.
\end{cor}

It is an easy observation that a shrinking $C^*$-algebra is both closed and subhomogeneous iff it is
unital (and then it is automatically regular). Furthermore, for subhomogeneous $C^*$-algebras,
the operations ``unitization of a nonunital algebra'' and ``one-dimensional essential extension
of a nonclosed algebra to a closed one'' are equivalent. The advantage of the latter statement is
that it makes (at least theoretical) sense for all shrinking $C^*$-algebras. One may also think
of this operation as of a ``one-point compactificafion'' of the solid pointed tower that corresponds
to a nonclosed algebra. (More generally, one may think of any closed shrinking $C^*$-algebra that
contains a fixed shrinking $C^*$-algebra $\aaA$ as an essential ideal as of a ``compactification''
of the pointed solid tower of $\aaA$). As the following result shows, regular shrinking
$C^*$-algebras behave (in this matter) in an expected way.

\begin{cor}{one-reg}
For every nonclosed regular shrinking $C^*$-algebra $\aaA$ there is a unique closed shrinking
$C^*$-algebra $\ddD$ that contains $\aaA$ as an essential ideal of codimension $1$. Moreover, $\ddD$
is regular.\par
Conversely, if $\ddD$ is a closed regular shrinking $C^*$-algebra and $\JjJ$ is an essential ideal
in $\ddD$ of codimension $1$, then $\JjJ$ is a nonclosed regular shrinking $C^*$-algebra.
\end{cor}
\begin{proof}
We start from the first part. We think of $\aaA$ as of $C^*_0(\Tt,\theta)$ for some solid pointed
tower $(\Tt,\theta)$. Since $\aaA$ is regular, so is $\Tt$ and therefore $\ddD \df C^*(\Tt)$ is
a closed regular shrinking $C^*$-algebra in which $\aaA$ is an ideal of codimension $1$. Moreover,
since $\aaA$ is nonclosed, $C^*_0(\Tt,\theta)$ is an essential ideal in $C^*_0(\Tt)$ (cf.\ the proof
of \PRO{reg-clo}). To establish the uniqueness of $\aaA$, take a solid pointed tower $(\Ss,\kappa)$
such that $\ddD' \df C^*_0(\Ss,\kappa)$ is closed and contains an essential ideal $\JjJ$
of codimension $1$ which is $*$-isomorphic to $\aaA$. First of all, we want to show that $\ddD'$ is
regular. We infer from \COR{ideal} that $\JjJ = \JjJ_{\Zz}$ where $\Zz$ is a subtower of $\Ss$ that
contains $\kappa$. Further, \THM{ext} implies that the quotient algebra $C^*_0(\Ss,\kappa) /
\JjJ_{\Zz}$ is $*$-isomorphic to $C^*_0(\Zz,\kappa)$. So, $C^*_0(\Zz,\kappa)$ is one-dimensional and
hence we deduce from \COR{dist} that $\core(\Zz)$ coincides with $\Zz_1$ and contains a single
element different from $\kappa$, say $\mu$. This implies that
\begin{equation}\label{eqn:aux23}
\JjJ = \{u \in C^*_0(\Ss,\kappa)\dd\ u(\mu) = 0\}.
\end{equation}
Further, \COR{*homo} provides the existence of a morphism $\tau\dd (\Ss,\kappa) \to (\Tt,\theta)$
such that the $*$-homomorphism $\Phi\dd C^*_0(\Tt,\theta) \ni u \mapsto u \circ \tau \in
C^*_0(\Ss,\kappa)$ is one-to-one and $\Phi(C^*_0(\Tt,\theta)) = \JjJ$. To conclude that $\Ss$ is
regular, we apply \PRO{solid}. Fix $n > 0$. Since $\Tt$ is regular, there is $N > 1$ such that
$\Tt[N] \cap \bigcup_{j=1}^n \Tt_j = \varempty$. Put $\Ss' \df \tau^{-1}(\Tt[N])$. It is readily
seen that $\Ss'$ is a subtower of $\Ss$ which is disjoint from $\bigcup_{j=1}^n \Ss_j$. We claim
that $\Ss[N] \subset \Ss'$ (which will imply that $\Ss$ is regular). To prove this, it suffices
to check that $\core(\Ss) \cap \bigcup_{k=N}^{\infty}\Ss_k \subset \Ss'$. To this end, take $\sS \in
\core(\Ss) \cap \bigcup_{k=N}^{\infty}\Ss_k$. Since $d(\mu) = 1 < d(\sS)$, we infer from
\eqref{eqn:aux23} that there is $u \in \JjJ$ with $u(\sS) = X$ where $X$ is a fixed irreducible
matrix of degree $d(\sS)$. Since $\Phi(C^*_0(\Tt,\theta)) = \JjJ$, we see that there is $v \in
C^*_0(\Tt,\theta)$ for which $\Phi(v) = u$. This means that $v(\tau(\sS)) = X$ and therefore
$\stab(\tau(\sS))$ consists of scalars multiples of the unit matrix (because $X$ is irreducible).
So, $\tau(\sS) \in \core(\Ss)$ (by \LEM{core}) and consequently $\sS \in \Ss'$.\par
Knowing that $\Ss$ is regular, we conclude that $\ddD'$ is so. But, in addition, $\ddD'$ is closed.
So, \PRO{reg-clo} yields that $\ddD'$ is $*$-isomorphic to $C^*_0(\Xx)$ for some regular solid tower
$\Xx$. Then $\JjJ$ corresponds to $C^*_0(\Xx,\lambda)$ for some $\lambda \in \Xx_1$. So,
$C^*_0(\Tt,\theta)$ and $C^*_0(\Xx,\lambda)$ are $*$-isomorphic and therefore $\Tt$ and $\Xx$ are
isomorphic, by \THM{shrink}. So, $\ddD'$ and $\ddD$ are $*$-isomorphic and we are done.\par
The second claim of the corollary is very simple and therefore we leave it to the reader.
\end{proof}

The next result enables characterizing all \textit{separable} closed shrinking $C^*$-algebras.

\begin{thm}{closed}
For a metrizable solid pointed tower $(\Tt,\theta)$ \tfcae
\begin{enumerate}[\upshape(i)]
\item $C^*_0(\Tt,\theta)$ is closed;
\item for all $N > 0$, every point of the closure of $\core(\Tt) \setminus \Tt[N]$ is disjoint from
 $\theta$.
\end{enumerate}
\end{thm}

We do not know whether the assumption of metrizability in the above result may be dropped.
In \EXM{clo} we give an example of a singly generated singular shrinking $C^*$-algebra which is
closed.\par
Our proof of \THM{closed} is complicated and based on a certain extension of one of celebrated
selection theorems due to Michael \cite{mi1,mi2,mi3} (consult also \S7 in \cite{b-p}, especially
Theorem~7.1 there). For its formulation, we recall some necessary definitions.
A \textit{multifunction} from a topological space $X$ into a topological space $Y$ is any function
defined on $X$ which assigns to each point of $X$ a nonempty subset of $Y$. The fact that $F$ is
a multifunction from $X$ into $Y$ shall be denoted by $F\dd X \to 2^Y$. A multifunction $F\dd X \to
2^Y$ is \textit{closed} (resp.\ \textit{convex}) if $F(x)$ is so for any $x \in X$. $F$ is said
to be \textit{lower semicontinuous} if for any open set $U$ in $Y$, the set
\begin{equation*}
F^{-1}(U) \df \{x \in X\dd\ F(x) \cap U \neq \varempty\}
\end{equation*}
is open (in $X$). A function $f\dd X \to Y$ is called a \textit{selection} for $F$ if $f(x) \in
F(x)$ for all $x \in X$.\par
The next result may be already known. However, we could not find anything about it
in the literature.

\begin{pro}{select}
Let $X$ be a paracompact topological space and
\begin{equation*}
\uuU_N \times X \ni (U,x) \mapsto U . x \in X
\end{equation*}
be a continuous action of $\uuU_N$ on $X$. For any closed convex lower semicontinuous multifunction
$F\dd X \to 2^{\mmM_N}$ such that
\begin{equation}\label{eqn:multi-unit}
F(U . x) = U . F(x) \qquad (x \in X,\ U \in \uuU_N)
\end{equation}
there exists a continuous selection $f\dd X \to \mmM_N$ such that
\begin{equation}\label{eqn:select-unit}
f(U . x) = U . f(x) \qquad (x \in X,\ U \in \uuU_N).
\end{equation}
\end{pro}
\begin{proof}
(To avoid misunderstandings, we explain that, for any set $\bbB \subset \mmM_N$ and $U \in \uuU_N$,
$U . \bbB$ is defined as the set of all matrices of the form $U . X$ where $X \in \bbB$. Similarly,
for any set $\bbB' \subset \mmM_N$, $\bbB - \bbB'$ is the set of all matrices of the form $X - Y$
where $X \in \bbB$ and $Y \in \bbB'$.) We shall improve the proof of Theorem~7.1 in \cite{b-p}. Let
$\lambda$ denote the probabilistic Haar measure on $\uuU_N$. Lemma~7.1 in \cite{b-p} shows that for
every convex lower semicontinuous (not necessarily closed) multifunction $\Phi\dd X \to 2^{\mmM_N}$
and any open ball $\vvV$ around zero of $\mmM_N$ (with respect to the classical operator norm
on $\mmM_N$) there exists a map $f_{\Phi,\vvV}\dd X \to \mmM_N$ such that the set
$(f_{\Phi,\vvV}(x)+\vvV) \cap \Phi(x)$ is nonempty for each $x \in X$. We claim that if,
in addition, condition \eqref{eqn:multi-unit} holds for $F = \Phi$, then the above function
$f_{\Phi,\vvV}$ may be chosen so that \eqref{eqn:select-unit} holds for $f = f_{\Phi,\vvV}$. Indeed,
it suffices to replace the above $f_{\Phi,\vvV}$ by $f_{\Phi,\vvV}'\dd X \to \mmM_N$ given
by $f_{\Phi,\vvV}'(x) \df \int_{\uuU_N} U^{-1} . f_{\Phi,\vvV}(U . x) \dint{\lambda(x)}$. Let us
briefly check that
\begin{equation}\label{eqn:aux22}
(\Psi(x) \df\,) (f_{\Phi,\vvV}'(x)+\vvV) \cap \Phi(x) \neq \varempty
\end{equation}
for any $x \in X$. Since $f_{\Phi,\vvV}(U . x) \in \Phi(U . x) - \vvV = U . (\Phi(x) - \vvV)$
(by \eqref{eqn:multi-unit} for $\Phi$), we see that $U^{-1} . f_{\Phi,\vvV}(U . x) \in \Phi(x)
- \vvV$ for any $U \in \uuU_N$. Further, the set $\kkK \df \{U^{-1} . f_{\Phi,\vvV}(U . x)\dd\ U \in
\uuU_N\}$ is compact and therefore its convex hull $\llL$ is compact as well. Consequently,
$f_{\Phi,\vvV}'(x) \in \llL \subset \Phi(x) - \vvV$, which is equivalent to \eqref{eqn:aux22}.
We leave it as a simple exerice that $\Psi\dd X \to 2^{\mmM_N}$ (given by \eqref{eqn:aux22}) is
a convex lower semicontinuous multifunction such that \eqref{eqn:multi-unit} holds for $F =
\Psi$.\par
Now we repeat the reasoning presented in \cite{b-p} (see pages 85--86 there). Let $\vvV_n$ denote
the open ball around zero of $\mmM_N$ of radius $2^{-n}$. Using the above argument, we define
inductively sequences $F_1,F_2,\ldots\dd X \to 2^{\mmM_N}$ and $f_1,f_2,\ldots\dd X \to \mmM_N$ of,
respectively, convex lower semicontinuous multifunctions and maps such that $f_n =
f_{F_{n-1},\vvV_n}$ (with $F_0 \df F$), $F_n(x) = (f_n(x) + \vvV_n) \cap F_{n-1}(x)$ (for any $x \in
X$ and $n > 0$) and respective conditions \eqref{eqn:multi-unit} and \eqref{eqn:select-unit} for
$F_n$ and $f_n$ hold. In a similar manner as in \cite{b-p} one shows that for each $x \in X$,
$(f_n(x))_{n=1}^{\infty}$ is a Cauchy sequence in $\mmM_N$ and the formula $f(x) \df
\lim_{n\to\infty} f_n(x)$ correctly defines a map $f\dd X \to \mmM_N$ which is a selection for $F$.
What is more, it follows from our construction that $f$ satisfies \eqref{eqn:select-unit} and we are
done.
\end{proof}

For the purpose of the proof of \THM{closed}, let us introduce

\begin{dfn}{multif}
Let $(\Ss,\theta)$ be a pointed tower and $\Aa$ be an arbitrary subset of $\Ss$ that contains
$\theta$. A multifunction $F\dd \Aa \to 2^{\mmM}$ is said to be \textit{compatible} if
\begin{itemize}
\item $F(\aA)$ is a compact convex subset of $\mmM_{d(\aA)}$ for any $\aA \in \Aa$; and
\item $F(U . (\bigoplus_{j=1}^n \aA_j)) = U . (\bigoplus_{j=1}^n F(\aA_j))\
 (= \{U . (\bigoplus_{j=1}^n X_j)\dd\ X_j \in F(\aA_j)\})$ for all $\aA_1,\ldots,\aA_n \in \Aa$ and
 $U \in \uuU_N$ with $N = \sum_{j=1}^n d(\aA_j)$ such that $U . (\bigoplus_{j=1}^n \aA_j)$ belongs
 to $\Aa$.
\end{itemize}
A compatible multifunction $F\dd \Aa \to 2^{\mmM}$ is \textit{bounded} if there is a positive real
constant $R$ such that $\|X\| \leqsl R$ for any $X \in \bigcup_{\aA\in\Aa} F(\aA)$. The infimum over
all such numbers $R$ is denoted by $\|F\|$. $F$ is called \textit{indefinite at $\theta$}
if $\card(F(\aA)) = 1$ for any $\aA \in \Aa$ disjoint from $\theta$. Finally, $F$ is said to be
\textit{upper semicontinuous} if for any open set $\vvV$ in $\mmM$, the set $\{\aA \in \Aa\dd\
F(\aA) \subset \vvV\}$ is relatively open in $\Aa$.
\end{dfn}

We need a result similar to \LEM{semi}, on extending compatible multifunctions that are indefinite
at $\theta$. Namely,

\begin{lem}{multi-ext}
Let $\Dd$ be a semitower of a solid pointed tower $(\Ss,\theta)$ such that $\theta \in \Dd$ and let
$F\dd \Dd \to 2^{\mmM}$ be a bounded compatible multifunction that is upper semicontinuous and
indefinite at $\theta$. Then $F$ \textup{(}uniquely\textup{)} extends to a compatible multifunction
$\bar{F}\dd \grp{\Dd} \to 2^{\mmM}$. Moreover, also $\bar{F}$ is bounded, upper semicontinuous and
indefinite at $\theta$.
\end{lem}
\begin{proof}
Recall that $\grp{\Dd}$ is given by \eqref{eqn:subtower}. Uniqueness of $\bar{F}$ is readily seen,
whereas its existence follows from \LEM{key} (cf.\ the proof of \LEM{semi}). $\bar{F}$ operates
as follows:
\begin{equation*}
\bar{F}\Bigl(U . \Bigl(\bigoplus_{j=1}^n \dD_j\Bigr)\Bigr) = U . \Bigl(\bigoplus_{j=1}^n
F(\dD_j)\Bigr)
\end{equation*}
where $n > 0$, $\dD_1,\ldots,\dD_n \in \Dd$ are irreducible, and $U \in \uuU_N$ with $N =
\sum_{j=1}^n d(\dD_j)$. Indefiniteness at $\theta$ and boundedness are immediate. Here we shall
focus only on showing that $\bar{F}$ is upper semicontinuous. Observe that, since $\bar{F}(\dD)$ is
closed for any $\dD \in \grp{\Dd}$ and $\bar{F}$ is bounded (and $\mmM_N$ is locally compact), upper
semicontinuity of $\bar{F}$ is equivalent to
\begin{itemize}
\item[(usc)] if $(\aA_{\sigma})_{\sigma\in\Sigma}$ is a net in $\grp{\Dd}$ that converges to $\aA
 \in \grp{\Dd}$ and $(X_{\sigma})_{\sigma\in\Sigma}$ is a convergent net of matrices such that
 $X_{\sigma} \in \bar{F}(\aA_{\sigma})$ for any $\sigma \in \Sigma$ and $\lim_{\sigma\in\Sigma}
 X_{\sigma} = X \in \mmM$, then $X \in \bar{F}(\aA)$.
\end{itemize}
So, let $(\aA_{\sigma})_{\sigma\in\Sigma}$, $\aA$, $(X_{\sigma})_{\sigma\in\Sigma}$ and $X$ be
as specified in (usc). Passing to a subnet (if applicable), we may and do assume that $\aA_{\sigma}
= U_{\sigma} . (\bigoplus_{j=1}^s \aA^{(\sigma)}_j)$ and $X_{\sigma} = U_{\sigma} .
(\bigoplus_{j=1}^s X^{(\sigma)}_j)$ where $s$ is independent of $\sigma$, $U_{\sigma} \in
\uuU_{d(\aA)}$, $\aA^{(\sigma)}_j \in \Dd$, $X^{(\sigma)}_j \in F(\aA^{(\sigma)}_j)$ and
$U_{\sigma}$, $\aA^{(\sigma)}_j$ and $X^{(\sigma)}_j$ converge (when $\sigma$ runs over elements
of $\Sigma$) to, respectively, $U \in \uuU_{d(\aA)}$, $\aA_j \in \Dd$ and $X_j \in \mmM$. Then $\aA
= U . (\bigoplus_{j=1}^s \aA_j)$ and $X = U . (\bigoplus_{j=1}^s X_j)$. Furthermore, since $F$ is
upper semicontinuous, we infer from (usc) that $X_j \in F(\aA_j)$ for $j=1,\ldots,s$, which simply
yields $X \in \bar{F}(\aA)$.
\end{proof}

\begin{lem}{contin}
Let $w\dd \Ss \to \mmM$ be a bounded compatible function on a solid tower $\Ss$. A necessary and
sufficient condition for the continuity of $w$ is that
\begin{itemize}
\item[(cc)] whenever $(\sS_{\sigma})_{\sigma\in\Sigma} \in \core(\Ss) \cap \Ss_N$ is a net that
 converges to $\sS \in \Ss$ and satisfies $\lim_{\sigma\in\Sigma} w(\sS_{\sigma}) = X \in \mmM_N$,
 then $w(\sS) = X$.
\end{itemize}
\end{lem}
\begin{proof}
Necessity of (cc) is clear. To see the sufficiency, we argue by a contradiction. We assume (cc)
holds, but for some net $(\sS_{\sigma})_{\sigma\in\Sigma}$ of elements of $\Ss$ that converges
to $\sS \in \Ss$, the net $(w(\sS_{\sigma}))_{\sigma\in\Sigma}$ of matrices does not converge
to $w(\sS)$. Passing to a subnet (and using boundedness of $w$ and compactness of $\Ss_{d(\sS)}$),
we may and do assume that $\sS_{\sigma} = U_{\sigma} . (\bigoplus_{j=1}^k \sS^{(\sigma)}_j)$ where
$k$ is independent of $\sigma$, $U_{\sigma} \in \uuU_{d(\sS)}$, $\sS^{(\sigma)}_j \in \Ss$ are
irreducible, and the nets $U_{\sigma}$, $\sS^{(\sigma)}_j$, $w(\sS_{\sigma})$ and
$w(\sS^{(\sigma)}_j)$ (for $j=1,\ldots,k$) converge (when $\sigma$ runs over elements of $\Sigma$)
to, respectively, $U \in \uuU_N$, $\sS_j \in \Ss$, $X \neq w(\sS)$ and $X_j \in \mmM$. We now infer
from (cc) that $w(\sS_j) = X_j$ for $j = 1,\ldots,k$. Consequently, $X = U . (\bigoplus_{j=1}^k X_j)
= U . (\bigoplus_{j=1}^k w(\sS_j)) = w(\sS)$, a contradiction.
\end{proof}

\begin{proof}[Proof of \THM{closed}]
First we shall show that (i) follows from (ii) for any (unnecessarily metrizable) solid pointed
tower $(\Tt,\theta)$ and then that (i) does not hold provided (ii) is not satisfied and $\Tt$ is
metrizable. So, assume (ii) holds. Let $\Dd \df \core(\Tt) \setminus \{\theta\}$. According
to \LEM{multi}, we may identify $M(\aaA)$ with a $*$-subalgebra of $C^*_{\Tt}(\Dd)$. Under such
an identification, it is easy to show that a map $u \in C^*_{\Tt}(\Dd)$ belongs to $S(\aaA)$ iff
\begin{itemize}
\item[(ext)] both $v\bigr|_{\Dd} u$ and $u v\bigr|_{\Dd}$ extend to maps from $C^*_0(\Tt,\theta)$
 for any $v \in C^*_0(\Tt,\theta)$,
\end{itemize}
and
\begin{equation}\label{eqn:u}
\lim_{n\to\infty} \sup\{\|u(\tT)\|\dd\ \tT \in \Dd \cap \Tt_n\} = 0.
\end{equation}
We shall now verify that each such a map $u$ extends to a map from $C^*_0(\Tt,\theta)$. It is easy
to see that $u$ extends uniquely to a compatible function $\bar{u}\dd \Tt \to \mmM$ that vanishes
at $\theta$. Moreover, $\bar{u}$ is bounded. So, it suffices to check that $\bar{u}$ is continuous.
To this end, we employ \LEM{contin}. Let $(\tT_{\sigma})_{\sigma\in\Sigma}$ be a net of irreducible
elements in $\Tt$ that converge to $\tT \in \Tt$. If $\tT \perp \theta$, then there is $v \in
C^*_0(\Tt,\theta)$ with $v(\tT) = I_{d(\tT)}$. So, arguing similarly as in the proof
of \COR{shrink}, we see that there exists $w \in C^*_0(\Tt,\theta)$ that is constantly equal
to $I_{d(\tT)}$ on a neighbourhood of $\tT$. Since $w\bigr|_{\Dd} u$ extends to a map from
$C^*_0(\Tt,\theta)$, one readily concludes that $\bar{u}$ is continuous at $\tT$ (and hence
$\lim_{\sigma\in\Sigma} \bar{u}(\tT_{\sigma}) = \bar{u}(\tT)$). Finally, assume $\theta \preccurlyeq
\tT$. We claim that then $\bar{u}(\tT) = 0$ and that for any $\epsi > 0$ there is $\sigma_{\epsi}
\in \Sigma$ with
\begin{equation}\label{eqn:aux24}
\|\bar{u}(\tT_{\sigma})\| \leqsl \epsi \qquad (\sigma \geqsl \sigma_{\epsi}).
\end{equation}
To this end, we employ condition (ii). Let $\Ss \df \{\sS \in \Tt\dd\ \|\bar{u}(\sS)\| \leqsl
\epsi\}$. Since $\bar{u}$ is compatible, we see that $\Ss$ is a subtower of $\Tt$ provided $\Ss$ is
closed. Let us now show that $\Ss$ is indeed closed. So, let $\{\sS_{\lambda}\}_{\lambda\in\Lambda}$
be a net of elements of $\Ss$ that converges to $\sS \in \Tt$. If $\sS \perp \theta$, then
$\lim_{\lambda\in\Lambda} \bar{u}(\sS_{\lambda}) = \bar{u}(\sS)$ (as shown above). And if $\theta
\preccurlyeq \sS$, we may and do assume that $\sS$ is not of the form $n \odot \theta$ (because
$n \odot \theta \in \Ss$ for all $n$) and consequently we may write $\sS = U . ((k \odot \theta)
\oplus \xX)$ where $U$ is a unitary matrix (of a respective degree), $k > 0$ and $\xX \perp \theta$.
Then there exists $v \in C^*_0(\Tt,\theta)$ with $v(\sS) = U . (0_k \oplus I_{d(\xX)})$ (where $0_k$
denotes the zero of $\mmM_k$) and $\|v\| \leqsl 1$. Since $v\bigr|_{\Dd} u$ extends to a map
on $\Ss$, we conclude that $v \bar{u}$ is continuous and therefore $\|\bar{u}(\sS)\| = \|v(\sS)
\bar{u}(\sS)\| = \lim_{\lambda\in\Lambda} \|v(\sS_{\lambda}) \bar{u}(\sS_{\lambda})\| \leqsl \epsi$,
which gives $\sS \in \Ss$. So, $\Ss$ is a subtower of $\Tt$. This, combined with \eqref{eqn:u},
implies that $\Tt[N] \subset \Ss$ for some $N$. Since $\tT$ is a limit of irreducible elements and
$\theta \preccurlyeq \tT$, condition (ii) implies that there exists $\sigma_{\epsi}$ such that
$\tT_{\sigma} \in \Tt[N]$ for all $\sigma \geqsl \sigma_{\epsi}$. Consequently, \eqref{eqn:aux24}
holds and $\|\bar{u}(\tT)\| \leqsl \epsi$ (because $\Ss$ is closed), which finishes the proof
of (i).\par
Now we assume $\Tt$ is metrizable and (ii) does not hold. We shall construct a function $u \in
C^*_{\Tt}(\Dd)$ which satisfies (ext) and \eqref{eqn:u}, but is not extendable to a map from
$C^*_0(\Tt,\theta)$ (which will mean that $C^*_0(\Tt,\theta)$ is not closed). Let $N > 0$ be such
that the closure of $\core(\Tt) \setminus \Tt[N]$ contains an element $\zZ \in \Tt$ with $\theta
\preccurlyeq \zZ$. The function $u$ we shall construct will have no limit at $\zZ$. Let $\zZ_1,
\zZ_2,\ldots$ be a sequence of elements of $\core(\Tt) \cap \Tt_{d(\zZ)} \setminus \Tt[N]$ that
converges to $\zZ$. We may and do assume that the sets $\uuU . \{\zZ_n\}$ are pairwise disjoint. For
simplicity, we put $\ell \df d(\zZ)\ (< N)$ and $\Tt_q^{\perp} \df \{\tT \in \Tt_q\dd\ \tT \perp
\theta\}$ (for $q > 0$). For any $\tT \in \Dd \cap (\Tt[N] \cup \bigcup_{j=1}^{\ell-1} \Tt_j)$
(where $\Dd = \core(\Tt) \setminus \{\theta\}$) we put $F_0''(\tT) \df \{0_{d(\tT)}\}$; and let
$F_0''(\theta) = \{\beta I_1\dd\ |\beta| \leqsl 1\} \subset \mmM_1$. It is easy to check that
$F_0''$ extends uniquely to a compatible multifunction $F_0'\dd \Ff_0' \to 2^{\mmM}$ where $\Ff_0'
\df \grp{\Tt[N] \cup \bigcup_{j=1}^{\ell-1} \Tt_j}$ (see \eqref{eqn:subtower}). Moreover, $\|F_0'\|
\leqsl 1$, $F_0'$ is indefinite at $\theta$ and upper semicontinuous. Notice that $\Tt_{\ell}
\setminus \core(\Tt) \subset \Ff_0'$ and $\card(F_0'(\zZ)) > 1$. So, there is nonzero $P \in
F_0'(\zZ)$. It is easy to check that all cluster points of the set $\Ee_0 \df \uuU_{\ell} .
\{\zZ_n\dd\ n > 0\}$ belong to $\Ee_0 \cup (\uuU . \{\zZ\})$ and therefore there is a bounded
compatible multifunction $F_0\dd \Ff_0 \to 2^{\mmM}$ from $\Ff_0 \df \grp{\Ff_0' \cup \Ee_0}$ which
extends $F_0'$, satisfies $F_0(U . \zZ_n) = \{U . P\}$ (for all $n > 0$ and $U \in \uuU_{\ell}$) and
is upper semicontinuous and indefinite at $\theta$ (consult \LEM{multi-ext}). It is also readily
seen that $\|F_0\| \leqsl 1$. It is worth noting here that:
\begin{itemize}
\item[(sel)] if $u_0\dd \Ff_0 \to \mmM$ is a compatible function such that $u_0(\theta) = 0$ and
 $u_0(\tT) \in F_0(\tT)$ for any $\tT \in \Ff_0$, then $\lim_{n\to\infty} u_0(\zZ_n) \neq 0 =
 u_0(\zZ)$
\end{itemize}
(the above property follows from the above construction). Now we run an inductive procedure. Assume
for some $k \in \{0,\ldots,N-\ell-1\}$ we have defined a compatible multifunction $F_k\dd \Ff_k \to
2^{\mmM}$ such that:
\begin{enumerate}[\upshape(F1$_k$)]
\item $\Ff_k$ is a subtower of $\Tt$ that contains $\Tt[N] \cup \bigcup_{j=1}^{\ell+k-1} \Tt_j$;
\item $F_k$ extends $F_0$;
\item $\|F_k\| \leqsl 1$, $F_k$ is indefinite at $\theta$ and upper semicontinuous.
\end{enumerate}
Our aim is to construct a compatible multifunction $F_{k+1}\dd \Ff_{k+1} \to 2^{\mmM}$ with
roperties (F1$_{k+1}$)--(F3$_{k+1}$). To this end, we put $q \df \ell+k$ and take
a $\uuU_q$-invariant metric $\varrho$ on $\Tt_q$ which induces the topology of $\Tt_q$; that is,
$\varrho(U . \xX,U . \yY) = \varrho(\xX,\yY)$ for any $\xX, \yY \in \Tt_q$ and $U \in \uuU_q$
(the existence of such a metric in that case is almost trivial, since $\Tt_q$ and $\uuU_q$ are
compact; for a similar result on existence of invariant metrics in more general cases consult, e.g.,
\cite{amn}). Now we define a multifunction $G\dd \Tt_q^{\perp} \to 2^{\mmM_q}$ as follows: $G(\tT)
\df F_k(\tT)$ for $\tT \in \Tt_q^{\perp} \cap \Ff_k$ and for $\tT \in \Tt_q^{\perp} \setminus
\Ff_k$,
\begin{equation}\label{eqn:G}
G(\tT) \df \overline{\conv}\Bigl(\bigcup\{F_k(\sS)\dd\ \sS \in \Ff_k \cap \Tt_q,\ \varrho(\tT,\sS) <
2 \dist_{\varrho}(\tT,\Ff_k \cap \Tt_q)\}\Bigr)
\end{equation}
where, for a nonempty set $\Aa \subset \Tt_q$, $\dist_{\varrho}(\tT,\Aa) \df \inf_{\aA\in\Aa}
\varrho(\tT,\aA)$ ($\Ff_k \cap \Tt_q$ is nonempty, because it contains $q \odot \theta$) and
$\overline{\conv}(\ssS)$ stands for the closed convex hull of a set $\ssS \subset \mmM_q$. Since
$\Ff_k$ is unitarily invariant, $F_k$ is compatible and $\varrho$ is $\uuU_q$-invariant, it is easy
to see that condition \eqref{eqn:multi-unit} holds for $F = G$. Of course, $G$ is closed and convex.
We shall now check that $G$ is lower semicontinuous. First of all, since $F_k$ is indefinite
at $\theta$ and upper semicontinuous, there is a continuous function $\varphi\dd \ddD \to \mmM_q$
such that $F_k(\sS) = \{\varphi(\sS)\}$ for each $\sS \in \ddD \df \Ff_k \cap \Tt_q^{\perp}$. Now
let $\vvV$ be an arbitrary open set in $\mmM_q$ and $\tT \in Tt_q^{\perp}$ be such that $G(\tT) \cap
\vvV \neq \varempty$. We consider two cases. First assume $\tT \in \Ff_k$. It follows from (T1) and
\LEM{unitary} that $\Tt_q^{\perp}$ is an open set in $\Tt_q$. So, we infer from the continuity
of $\varphi$ that there is $r > 0$ such that $B_{\varrho}(\tT,2r) \subset \Tt_q^{\perp}$ and
$\varphi(\ddD \cap B_{\varrho}(\tT,2r)) \subset \vvV$ where $B_{\varrho}(\tT,\delta)$, for each
$\delta > 0$, denotes the open $\varrho$-ball in $\Tt_q$ around $\tT$ of radius $\delta$. Let $\tT'
\in \Tt_q^{\perp}$ be an arbitrary element such that $\varrho(\tT',\tT) < r$. If $\tT' \in \ddD$,
$G(\tT') = \{\varphi(\tT')\} \subset \vvV$. And if $\tT' \notin \ddD$, then
$\dist_{\varrho}(\tT',\Ff_k \cap \Tt_q) < r$, which means that there is $\sS \in \Ff_k \cap \Tt_q$
for which $\varrho(\tT',\sS) < r$. Then $\varrho(\sS,\tT) < 2r$ and hence $\sS \in \ddD$.
Consequently, $\varphi(\sS) \in G(\tT') \cap \vvV$. So, in that case we are done. Now assume $\tT
\notin \ddD$. Since $\ddD$ is closed in $\Tt_q^{\perp}$, at each point of a small neighbourhood
of $\tT$ the multifunction $G$ is given by \eqref{eqn:G}. Further, since $G(\tT) \cap \vvV$ is
nonempty, there are a finite number of elements $\sS_1,\ldots,\sS_p$ of $\Ff_k \cap \Tt_q$ such that
$\varrho(\tT,\sS_j) < 2 \dist_{\varrho}(\tT,\Ff_k \cap \Tt_q)$ for $j=1,\ldots,p$ and the convex
hull $\wwW$ of $\bigcup_{j=1}^p F_k(\sS_j)$ intersects $\vvV$. Then, for any point $\tT'$ of a small
neighbourhood of $\tT$ we have $\varrho(\tT',\sS_j) < 2 \dist_{\varrho}(\tT',\Ff_k \cap \Tt_q)$ for
each $j \in \{1,\ldots,p\}$ and hence $G(\tT') \cap \vvV \neq \varempty$, since $\wwW \subset
G(\tT')$. So, we have shown that $G$ is lower semicontinuous. It follows from \PRO{select} that
there is a continuous selection $g\dd \Tt_q^{\perp} \to \mmM_q$ for $G$ such that $g(U . \tT) = U .
g(\tT)$ for any $\tT \in \Tt_q^{\perp}$ and $U \in \uuU_q$. Observe that $\Ff_k \cup \Tt_q^{\perp}
\supset \Tt_q$ (by (F1$_k$)) and $\{g(\tT)\} = F_k(\tT)$ for any $\tT \in \Ff_k \cap \Tt_q^{\perp}$.
These two remarks imply that we may correctly define a multifunction $F_{k+1}'\dd \Ff_k \cup \Tt_q
\to 2^{\mmM}$ by $F_{k+1}'(\tT) \df F_k(\tT)$ for $\tT \in \Ff_k$ and $F_{k+1}'(\tT) \df \{g(\tT)\}$
for any $\tT \in \Tt_q^{\perp}$. One readily verifies that $F_k$ is compatible (because $\Tt_q
\setminus \core(\Tt_q) \subset \Ff_k$) and indefinite at $\theta$, and $\|F_{k+1}'\| \leqsl 1$. Our
main claim on $F_{k+1}'$ is that it is upper semicontinuous. To check this property, we use
criterion (usc). So, let $(\tT_n)_{n=1}^{\infty}$ and $(X_n)_{n=1}^{\infty}$ be sequences of,
respectively, elements of $\Ff_k \cup \Tt_q$ and matrices such that $\lim_{n\to\infty} \tT_n = \tT
\in \Ff_k \cup \Tt_q$, $X_n \in F_{k+1}'(\tT_n)$ and $\lim_{n\to\infty} X_n = X \in \mmM$ (recall
that $\Tt$ is metrizable). We only need to show that $X \in F_{k+1}'(\tT)$. Passing, if applicable,
to two subsequences, we may assume that either $\tT_n \in \Ff_k$ for all $n$ or $\tT_n \in
\Tt_q^{\perp} \setminus \Ff_k$ for all $n$. In the former case we have nothing to do: the conclusion
follows from the upper semicontinuity of $F_k$ (recall that $\Ff_k$ is closed). In the latter case
we may assume $\tT \notin \Tt_q^{\perp}$ (because otherwise $F_{k+1}'(\tT) = \{g(\tT)\}$ and $g(\tT)
= \lim_{n\to\infty} g(\tT_n) = \lim_{n\to\infty} X_n = X$). So, $\tT \in \Ff_k \cap \Tt_q$ and
$F_{k+1}'(\tT) = F_k(\tT)$. It suffices to check that
\begin{equation}\label{eqn:aux25}
\lim_{n\to\infty} \dist_{\|\cdot\|}(X_n,F_k(\tT)) = 0.
\end{equation}
Since $F_k\bigr|_{\Ff_k \cap \Tt_q}$ is upper semicontinuous, we infer that
\begin{itemize}
\item[(usc')] for any $\epsi > 0$ there is $\delta > 0$ such that $\dist_{\|\cdot\|}(T,F_k(\tT))
 \leqsl \epsi$ whenever $T \in F_k(\sS)$ where $\sS \in \Ff_k \cap \Tt_q$ satisfies
 $\varrho(\sS,\tT) \leqsl \delta$.
\end{itemize}
For a fixed $\epsi > 0$ let $\delta$ be as specified in (usc'). Fix $n > 0$ such that
$\varrho(\tT_n,\tT) \leqsl \frac13 \delta$. Since $X_n = g(\tT_n) \in G(\tT_n)$, we infer from
\eqref{eqn:G} that there are $\sS_1,\ldots,\sS_p \in \Ff_k \cap \Tt_q$, $T_j \in F_k(\sS_j)$ and
$\alpha_j \in [0,1]$ (for $j=1,\ldots,p$) such that $\|X_n - W_n\| \leqsl \epsi$ where $W_n \df
\sum_{j=1}^p \alpha_j T_j$ and $\varrho(\sS_j,\tT_n) < 2 \dist_{\varrho}(\tT_n,\Ff_k \cap \Tt_q)$.
Then $\varrho(\sS_j,\tT) \leqsl 3 \varrho(\tT_n,\tT) \leqsl \delta$ and therefore
$\dist_{\|\cdot\|}(T_j,F_k(\tT)) \leqsl \epsi$, by (usc'). Since $F_k(\tT)$ is a convex set, so is
the set $\{T \in \mmM_q\dd\ \dist_{\|\cdot\|}(T,F_k(\tT)) \leqsl \epsi\}$, from which we infer that
$\dist_{\|\cdot\|}(W_n,F_k(\tT)) \leqsl \epsi$. Consequently, $\dist_{\|\cdot\|}(X_n,F_k(\tT))
\leqsl 2 \epsi$ (provided $\varrho(\tT_n,\tT) \leqsl \frac13 \delta$), which yields
\eqref{eqn:aux25}.\par
Now taking into account the fact that $\Ff_k \cup \Tt_q$ is a semitower in $\Tt$, we conclude from
\LEM{multi-ext} that $F_{k+1}'$ extends to a compatible multifunction $F_{k+1}\dd \Ff_{k+1} \to
2^{\mmM}$ where $\Ff_{k+1} \df \grp{\Ff_k \cup \Tt_q}$ for which conditions
(F1$_{k+1}$)--(F3$_{k+1}$) are fulfilled.\par
Finally, when we substitute $k \df N - \ell - 1$, we shall obtain a bounded upper semicontinuous
compatible multifunction $F \df F_{k+1}\dd \Tt \to 2^{\mmM}$ that extends $F_0$ and is indefinite
at $\theta$. We define $u\dd \Dd \to \mmM$ by the rule $\{u(\tT)\} = F(\tT)$. It is easy to see that
$u$ extends to a unique compatible function $\bar{u}\dd \Tt \to \mmM$ that vanishes at $\theta$ and
satisfies $\bar{u}(\tT) \in F(\tT)$ for any $\tT \in \Tt$. We deduce from (sel) that $u$ does not
extend to a map from $C^*_0(\Tt,\theta)$. Moreover, it follows from the definition of $F_0$ that
$u(\tT) = 0$ for any $\tT \in \core(\Tt) \cap \bigcup_{j=N}^{\infty} \Tt_j$. So, condition
\eqref{eqn:u} is trivially satisfied. Thus, to complete the proof, it is enough to check (ext).
To this end, we fix $v \in C^*_0(\Tt,\theta)$. We shall apply \LEM{contin} to deduce that both
$v \bar{u}$ and $\bar{u} v$ are continuous. Since the proof for $\bar{u} v$ is quite similar,
we shall only check that $v \bar{u}$ is continuous. Let $\tT_1,\tT_2,\ldots \in \core(\Tt)$ be such
that $\lim_{n\to\infty} \tT_n = \tT \in \Tt$ and $\lim_{n\to\infty} v(\tT_n) \bar{u}(\tT_n) = X \in
\mmM$. We need to show that $X = v(\tT) \bar{u}(\tT)$. Passing to a subsequence, we may additionally
assume that $\lim_{n\to\infty} \bar{u}(\tT_n) = Y$. Then
\begin{equation}\label{eqn:XY}
X = v(\tT) Y.
\end{equation}
We may also assume $\tT_n \neq \theta$ for all $n$. We consider three brief cases. If $\tT \perp
\theta$, then $\tT_n \perp \theta$ as well for almost all $n$ and hence $Y = \lim_{n\to\infty}
\bar{u}(\tT_n) = \bar{u}(\tT)$ (because in that case $F(\tT_n) = \{\bar{u}(\tT_n)\}$, $F(\tT) =
\{\bar{u}(\tT)\}$ and $F$ is upper semicontinuous). So, $X = v(\tT) \bar{u}(\tT)$. Next, assume $\tT
= n \odot \theta$ for some $n > 0$. Then, since $v(n \odot \theta) = 0$ and $u$ is bounded,
we readily get $\lim_{n\to\infty} v(\tT_n) \bar{u}(\tT_n) = 0 = v(\tT) u(\tT)$. Finally, we assume
that $\theta \preccurlyeq \tT$ and $\tT$ is not of the form $n \odot \theta$. Then there is
a unitary matrix $U$, a positive integer $k$ and an element $\sS \in \Tt$ such that $\tT = U . ((k
\odot \theta) \oplus \sS)$ and $\sS \perp \theta$. Note that $F(\tT) = U . (F(k \odot \theta) \oplus
\{\bar{u}(\sS)\})$, $\bar{u}(\tT) = U . (0_k \oplus \bar{u}(\sS))$ and $v(\tT) = U . (0_k \oplus
v(\sS))$. Since $\lim_{n\to\infty} \tT_n = \tT$, we conclude from the upper semicontinuity of $F$
that $Y \in F(\tT)$. So, $Y$ has the form $Y = A \oplus \bar{u}(\sS)$ where $A \in \mmM_k$. Finally,
thanks to \eqref{eqn:XY}, we obtain $X = (U . (0_k \oplus v(\sS))) \cdot (U . (A \oplus
\bar{u}(\sS))) = U . (0_k \oplus (v(\sS)\bar{u}(\sS))) = v(\sS) \bar{u}(\sS)$, which finishes
the proof.
\end{proof}

\begin{prb}{clo}
Characterize all closed singular shrinking $C^*$-algebras.
\end{prb}

We finally come to the point where the undertaken naming---``regular'' and ``singular'' (for solid
towers or shrinking $C^*$-algebras)---is explained.

\begin{pro}{non-one}
There exists a singly generated nonclosed singular shrinking $C^*$-algebra which cannot be embedded
into a closed shrinking $C^*$-algebra as an essential ideal of codimension $1$.
\end{pro}
\begin{proof}
Let $\ttT$ be the subspace of $\mmM$ consisting of all matrices whose norms do not exceed $1$. It is
clear that $\ttT$ is a standard tower. For each $n > 1$, let $X_n \in \ttT_n$ be an irreducible
matrix whose norm is $1$. Further, let $\theta \in \mmM_1$ denote the zero matrix. Now put $\Dd \df
\{n \odot \theta\dd\ n > 0\} \cup \{2^{-k} (U . X_n)\dd\ k > n > 1,\ U \in \uuU_n\}$. Observe that
$\Dd$ is a semitower in $\ttT$ and thus, according to \LEM{sub}, $\Tt \df \grp{\Dd}$ is a subtower
of $\ttT$. Moreover, $(\Tt,\theta)$ is a solid pointed tower and $\aaA \df C^*_0(\Tt,\theta)$ is
singly generated, by \PRO{gener} (observe that $\core(\Tt) \cap \Tt_n = \{2^{-k} (U . X_n)\dd\
k > n,\ U \in \uuU_n\}$ for $n > 1$). We claim that $\aaA$ is the $C^*$-algebra we search for. First
of all, observe that
\begin{equation}\label{eqn:TN}
\Tt[N] = \{n \odot \theta\dd\ n > 0\} \cup \bigcup_{j=N}^{\infty} \Tt_j
\end{equation}
(see \REM{T0}) and therefore $\Tt$ is singular and, thanks to \THM{closed}, $\aaA$ is not closed.
Further, suppose, on the contrary, that there exists a closed shrinking $C^*$-algebra $\ddD$ that
contains an essential ideal $\JjJ$ of codimension $1$ such that $\aaA$ is $*$-isomorphic to $\JjJ$.
It follows from \COR{one-reg} that $\ddD$ is singular (since $\aaA$ is so). We represent $\ddD$
in the form $C^*_0(\Ss,\kappa)$ where $\Ss$ is a singular tower. Similarly as argued in the proof
of \COR{one-reg}, one deduces that then:
\begin{enumerate}[($\tau$1)]
\item $\JjJ = \{u \in C^*_0(\Ss,\kappa)\dd\ u(\mu) = 0\}$ for some $\mu \in \Ss_1$ different from
 $\kappa$; and
\item there is a morphism $\tau\dd (\Ss,\kappa) \to (\Tt,\theta)$ such that the $*$-homomorphism
 \begin{equation*}
 \Phi\dd C^*(\Tt_0,\theta) \ni v \mapsto v \circ \tau \in C^*_0(\Ss,\kappa)
 \end{equation*}
 is one-to-one and $\Phi(C^*_0(\Tt,\theta)) = \JjJ$.
\end{enumerate}
Further, we infer from the above properties that
\begin{enumerate}[($\tau$1)]\addtocounter{enumi}{2}
\item $\tau(\mu) = \tau(\kappa) = \theta$; and
\item $\tau$ establishes a one-to-one correspondence between $\core(\Tt)$ and $\core(\Ss) \setminus
 \{\kappa\}$.
\end{enumerate}
Property ($\tau$3) is readily seen. Further, if $\sS \in \core(\Ss) \setminus \{\kappa,\mu\}$, there
is $g \in \JjJ$ such that $g(\sS)$ is an irreducible matrix. Since $g = f \circ \tau$ for some $f
\in C^*_0(\Tt,\theta)$, we see that $g(\sS) \in \core(\Tt)$ (cf.\ the proof of \COR{one-reg}), and
hence $\tau(\core(\Ss)) \subset \core(\Tt)$. One concludes that therefore $\tau(\Ss)$ is a subtower
of $\Tt$ (see, for example, the proof of \PRO{sur}). So, \COR{dist}, combined with the property that
$\Phi$ is one-to-one, yields that $\tau(\Ss) = \Tt$ (consult \REM{inj}). We infer that for any $\tT
\in \core(\Tt)$ there is $\sS \in \Ss$ for which $\tau(\sS) = \tT$. But then $\stab(\sS) \subset
\stab(\tT)$ and hence \LEM{core} implies that $\sS \in \core(\Ss)$. Finally, if the above $\sS$
coincides with $\kappa$, we may replace it by $\mu$. All these remarks and arguments prove
($\tau$4).\par
Properties ($\tau$3)--($\tau$4) imply that $\Ss_1 = \{\kappa,\mu\}$ (since $\Tt_1 = \{\theta\}$) and
$\tau^{-1}(\Tt[N]) \supset \Ss[N]$ for each $N$ (because $\tau^{-1}(\Tt[N])$ is a subtower). Now
we apply \THM{closed}: since $C^*_0(\Ss,\kappa)$ is closed, each point of the closure of $\core(\Ss)
\setminus \tau^{-1}(\Tt[N])$ is disjoint from $\kappa$. But \eqref{eqn:TN} and ($\tau$3)--($\tau$4)
show that $\core(\tau^{-1}(\Tt[N])) \subset \{\theta,\kappa\} \cup \bigcup_{j=N}^{\infty} \Ss_j$.
So, for all $N > 2$, $\core(\Ss) \cap \Ss_{N-1} \subset \core(\Ss) \setminus \tau^{-1}(\Tt[N])$ and
therefore any element of $\Dd' \df \overline{\core}(\Ss) \setminus \{\kappa\}$ is disjoint from
$\kappa$ (recall that $\Ss_1 = \{\kappa,\theta\}$). Since $\Dd'$ is closed, \LEM{sub} yields that
$\Ss_0 \df \grp{\Dd'}$ is a subtower of $\Ss$. Since $\kappa \notin \Ss_0$, \COR{dist} implies that
there is $f \in C^*(\Ss)$ that vanishes at each point of $\Ss_0$, but not at $\kappa$. But then
automatically $f \in C^*_0(\Ss)$ and therefore $\kappa \notin \Ss_{(0)}$, which contradicts the fact
that $\Ss$ is singular.
\end{proof}

The above result may be interpreted in a way that certain solid towers do not have ``one-point
compactifications'' (which is in contrast to regular solid towers, see \COR{one-reg}). This is
a strange property which motivated us to call such towers singular.\par
Closed regular shrinking $C^*$-algebras may be characterized intrinsically:

\begin{pro}{clo}
For a shrinking $C^*$-algebra $\aaA$ \tfcae
\begin{enumerate}[\upshape(i)]
\item $\aaA$ is closed and regular;
\item for each $n > 0$, the space of all $n$-dimensional nondegenerate representations of $\aaA$
 is compact in the pointwise convergence topology;
\item for each $n > 0$, the closure \textup{(}in the pointwise convergence topology\textup{)}
 of all $n$-dimensional nonzero irreducible representations of $\aaA$ consists of nondegenerate
 representations;
\item for any $n > 0$, there exist $N > 0$ and $a \in \aaA$ such that $\pi(a) = I_k$ for any
 $k$-dimensional nondegenerate representation $\pi$ of $\aaA$ with $k \leqsl n$, and $\pi(a) = 0$
 for any $k$-dimensional nondegenerate representation $\pi$ of $\aaA$ with $k > N$.
\end{enumerate}
\end{pro}
\begin{proof}
We identify $\aaA$ with $C^*_0(\Ss,\theta)$ where $(\Ss,\theta)$ is a solid pointed tower. Put $\Tt
\df \{\sS \in \Ss\dd\ \sS \perp \theta\}$ and $\Dd \df \core(\Ss) \setminus \{\theta\}$. Observe
that (iii) trivially follows from (ii), and (ii) is a simple consequence of (iv); whereas (iii)
means that the closure of $\Dd$ is contained in $\Tt$, and then $\Tt$ is closed, by \LEM{sub} (since
$\Tt = \grp{\Dd}$). But when $\Tt$ is closed, the last paragraph of the proof of \PRO{reg-clo} shows
that then $\Tt$ is a solid tower and $C^*_0(\Ss,\theta)$ is $*$-isomorphic to $C^*_0(\Tt)$.
Moreover, $\Tt$ is regular because $\theta \notin \Tt$. So, (i) is implied by (iii). Finally,
if $\aaA$ is closed and regular, we may identify $\aaA$ with $C^*_0(\Tt)$ for some regular solid
tower. Then it follows from \PRO{solid} that for any $n > 0$ there is $N > 0$ with $\Tt[N] \cap
\bigcup_{k=1}^n \Tt_k = \varempty$. Noticing that $\Tt[N] \cup \bigcup_{k=1}^n \Tt_k$ is
a semitower, we infer from \THM{extend} that there is $g \in C^*(\Tt)$ that vanishes at each point
of $\Tt[N]$ and satisfies $g(\tT) = I_{d(\tT)}$ for any $\tT \in \bigcup_{k=1}^n \Tt_k$. Then $g \in
C^*_0(\Tt)$ and thus the conclusion of (iv) follows.
\end{proof}

We conclude the paper with the following

\begin{exm}{clo}
Let us give a simple example of a singly generated (and thus separable) singular shrinking
$C^*$-algebra which is closed. Denote by $\ttT \subset \mmM$ a standard tower consisting of all
matrices whose norms do not exceed $1$. Below we shall use a classical result stating that
$\core(\ttT)$ is dense in $\ttT$. We define inductively sets $\Dd_1,\Dd_2,\ldots$ such that:
\begin{enumerate}[(D1)]\addtocounter{enumi}{-1}
\item $\Dd_1 = \{\theta\}$ where, as usual, $\theta$ is the zero of $\mmM_1$;
\item $\Dd_n \subset \ttT_n$ is compact, countable and $\|X\| \leqsl 2^{-n}$ for any $X \in \Dd_n$;
\item $\Dd_n$ coincides with the closure of $\Dd_n \cap \core(\ttT)$;
\item for $n > 1$, $\Dd_n \setminus \core(\ttT)$ coincides with the set $\Dd_n'$ of all matrices
 of the form $(n/k) \odot \dD$ where $k < n$, $k$ divides $n$, $\dD \in \Dd_k$ and $\|\dD\| \leqsl
 2^{-n}$.
\end{enumerate}
The above sets $\Dd_n$ may simply be constructed: $\Dd_1$ is given in (D0). Further, if $\Dd_1,
\ldots,\Dd_{n-1}$ are defined, we may define $\Dd_n'$ as specified in (D3). Since $\Dd_n'$ is
compact and countable and $\core(\ttT)$ is dense in $\ttT$, we may find a countable set $\Dd_n''
\subset \core(\ttT) \cap \ttT_n$ whose closure (in $\ttT$) coincides with $\Dd_n \df \Dd_n'' \cup
\Dd_n'$. It is clear that, in addition, $\Dd_n''$ may be chosen so that (D1) holds for $\Dd_n$. Now
we put $\Dd \df \uuU . (\bigcup_{n=1}^{\infty} \Dd_n)$. Since $\Dd$ satisfies all assumptions
of \LEM{sub}, we infer from that result that $\Tt \df \grp{\Dd}$ is a well defined subtower
of $\ttT$. Moreover, the above construction shows that $\overline{\core}(\Tt) = \Dd$. Using
\REM{T0}, one may now check that $\Tt_{(0)} = \{n \odot \theta\dd\ n > 0\}$ (and thus $\aaA \df
C^*_0(\Tt,\theta)$ is singular) and each point of the closure of $\core(\Tt) \setminus \Tt[N]$ is
disjoint from $\theta$ (by (D2) and (D3)). So, \THM{closed} yields that $\aaA$ is closed, whereas
\PRO{gener} implies that $\aaA$ is singly generated. The details are left to the reader.
\end{exm}

\end{document}